\documentclass[12pt]{article}
\usepackage{a4wide}
\usepackage{amsmath, amsthm, amsfonts, amssymb, bbm, bm}
\usepackage{graphics}
\usepackage[english]{babel}
\usepackage[font=small,format=plain,labelfont=bf,up]{caption}
\usepackage{hyperref}

\usepackage{tikz}
\usepackage{tikz-cd}
\usetikzlibrary{math,decorations.markings}

\theoremstyle{plain}

\newtheorem{theorem}{Theorem}
\newtheorem{lemma}[theorem]{Lemma}
\newtheorem{proposition}[theorem]{Proposition}
\newtheorem{corollary}[theorem]{Corollary}

\theoremstyle{definition}

\newtheorem{remark}[theorem]{Remark}
\newtheorem{example}[theorem]{Example}
\newtheorem{definition}[theorem]{Definition}

 \numberwithin{equation}{section}
 \numberwithin{theorem}{section}

\DeclareMathOperator{\tr}{tr}
\DeclareMathOperator{\Dim}{Dim}
\DeclareMathOperator{\MCG}{MCG}

\DeclareMathOperator{\evL}{\overset{\xleftarrow{\phantom{ev}}}{ev}\!}
\DeclareMathOperator{\coevL}{\overset{\xleftarrow{\phantom{coev}}}{coev}}
\DeclareMathOperator{\evR}{\overset{\xrightarrow{\phantom{ev}}}{ev}\!}
\DeclareMathOperator{\coevR}{\overset{\xrightarrow{\phantom{coev}}}{coev}}

\newcommand{\be}{\begin{equation}}
\newcommand{\ee}{\end{equation}}

\newcommand{\Graph}{\mathrm{Graph}}
\newcommand{\VGraph}{\mathcal{V}\mathrm{Graph}}
\newcommand{\R}{\mathcal{R}}
\newcommand{\VR}{\mathcal{VR}}
\newcommand{\SN}{\mathrm{SN}}
\newcommand{\BSig}{\mathbf{\Sigma}}
\newcommand{\BQ}{\mathbf{Q}}
\newcommand{\forget}{\mathbf{U}}

\newcommand\eps           {\varepsilon}
\newcommand\id            {id}
\newcommand\one           {{\bf1}}

\newcommand\Cb            {\mathbb{C}}

\newcommand\Rb            {\mathbb{R}}
\newcommand\Zb            {\mathbb{Z}}

\newcommand\Cc            {\mathcal{C}}
\newcommand\Ec            {\mathcal{E}}
\newcommand\Ic            {\mathcal{I}}
\newcommand\Jc            {\mathcal{J}}
\newcommand\Mc            {\mathcal{M}}
\newcommand\Zc            {\mathcal{Z}}

\begin{document}

\thispagestyle{empty}
\def\thefootnote{\fnsymbol{footnote}}
\begin{flushright}
ZMP-HH/19-15\\
Hamburger Beitr\"age zur Mathematik 798\\
KCL-MTH-19-02
\end{flushright}
\vskip 3em
\begin{center}\LARGE
	String-net models for\\
non-spherical pivotal fusion categories
\end{center}

\vskip 1em

\begin{center}
{\large 
Ingo Runkel~\footnote{Email: {\tt ingo.runkel@uni-hamburg.de}}}
\\[1em]
Fachbereich Mathematik, Universit\"at Hamburg\\
Bundesstra\ss e 55, 20146 Hamburg, Germany
\\[.5em]
and
\\[.5em]
Department of Mathematics, King's College London,\\
The Strand, London WC2R 2LS, United Kingdom
\end{center}

\vskip 1em

\begin{center}
  July 2019
\end{center}

\vskip 1em

\begin{abstract}
A string-net model associates a vector space to a surface in terms of graphs decorated by objects and morphisms of a pivotal fusion category modulo local relations. String-net models are usually considered for spherical fusion categories, and in this case the vector spaces agree with the state spaces of the corresponding Turaev-Viro topological quantum field theory. 

In the present work some effects of dropping the sphericality condition are investigated. In one example of non-spherical pivotal fusion categories, the string-net space counts the number of $r$-spin structures on a surface and carries an isomorphic representation of the mapping class group. Another example concerns the string-net space of a sphere with one marked point labelled by a simple object $Z$ of the Drinfeld centre. This space is found to be non-zero iff $Z$ is isomorphic to a non-unit simple object determined by the non-spherical pivotal structure. 

The last example mirrors the effect of deforming the stress tensor of a two-dimensional conformal field theory, such as in the topological twist of a supersymmetric theory.
\end{abstract}

\setcounter{footnote}{0}
\def\thefootnote{\arabic{footnote}}

\newpage

\tableofcontents

\section{Introduction}

String-net models were introduced in \cite{Levin:2004mi} to describe states in so-called doubled topological phases. In \cite{Kadar:2009fs,Koenig:2010,Kirillov:2011mk}, string-net spaces were shown to be equal to state spaces of Turaev-Viro topological quantum field theories. These works start their constructions from a pivotal fusion category $\Cc$ which is in addition spherical. The present paper explores some of the effects of dropping the sphericality condition.

\medskip

In brief, given a pivotal fusion category $\Cc$ over the complex numbers $\Cb$, the string-net construction assigns a $\Cb$-vector space 
\be
\SN(\Sigma,\Cc) 
\ee
to an oriented surface $\Sigma$, possibly with boundary. The boundary can contain marked points labelled by objects of $\Cc$. The vector space $\SN(\Sigma,\Cc)$ is defined by first taking the $\Cb$-linear span of all graphs drawn on $\Sigma$ with edges labelled by objects of $\Cc$ -- possibly ending at a marked boundary point -- and vertices labelled by morphisms in $\Cc$. 
Then one takes the quotient by all relations that hold in embedded discs when evaluating the graph inside the disc as a string diagram in $\Cc$. In particular, if $\Sigma$ itself is a disc with marked points $U_1,\dots,U_n$ on the boundary, then $\SN(\Sigma,\Cc)$ is isomorphic to the Hom-space $\Cc(\one, U_1 \otimes \cdots \otimes U_n)$. The details of the string-net construction are given in Section~\ref{sec:stringnet}. The important point is that it does not depend on $\Cc$ being spherical or not.

As already mentioned, for spherical $\Cc$ the string-net space for a given marked surface is canonically isomorphic to the state space of the Turaev-Viro TQFT for $\Cc$ \cite{Kadar:2009fs,Koenig:2010,Kirillov:2011mk}. This space in turn is isomorphic to the state space of the Reshetikhin-Turaev theory for the Drinfeld double $\Zc(\Cc)$ of $\Cc$ \cite{Balsam:2010,Turaev:2010pp,Balsam:2010c}. Some properties of state spaces that one can immediately conclude in the spherical case are:
\begin{enumerate}
\item The state space for a sphere with no marked points is one-dimensional.
\item The dimension of the state space of a torus with no marked points is given by the number of simple objects in $\Zc(\Cc)$.
\item The state space of a genus $g$ surface with no marked points is non-zero.
\item The state space for a sphere with one marked point labelled by a simple object $Z \in \Zc(\Cc)$ is zero-dimensional unless $Z \cong \one$.
\end{enumerate}
Apart from the last point, these are directly statements about the corresponding string-net spaces $\SN(\Sigma,\Cc)$. To relate the last statement to string-nets, one needs to extend the string-net formalism to include marked points labelled by objects in $\Zc(\Cc)$. This has been done in \cite{Kirillov:2011mk}.

\medskip

One can now ask what happens to properties 1--4 if one starts from a non-spherical pivotal fusion category $\Cc$. The most direct effect is for the sphere and gives the reason for the name ``spherical'' (Proposition~\ref{prop:SN-S2-1d-or-0d}).

\begin{proposition}\label{prop:intro-S2dim}
For a pivotal fusion category $\Cc$ we have
\be
	\dim \, \SN(S^2,\Cc) 
	=
	\begin{cases}
		1 &;~ \Cc \text{ spherical} \\
		0 &;~ \Cc \text{ not spherical}
	\end{cases}
\ee
\end{proposition}

The string-net space on the torus, on the other hand, is not affected by $\Cc$ being spherical or not (Proposition~\ref{prop:SN-basis-torus}).

\begin{proposition}
For a pivotal fusion category $\Cc$, the dimension of $\SN(T^2,\Cc)$ is given by the number of simple objects in $\Zc(\Cc)$. 
\end{proposition}

The main motivation for carrying out the investigation presented in this paper are the following two observations made in specific examples.

\medskip

The first observation concerns the case when as a fusion category, $\Cc$ is simply the category $\Cc_r$ of $\Zb_r$-graded $\Cb$-vector spaces for some $r>0$. The pivotal structures on $\Cc_r$ are classified by $r$'th roots of unity $\zeta$. 
Write $\Cb_u$ for the one-dimensional vector space concentrated in grade $u \in \Zb_r$.
The right and left quantum dimension of $\Cb_u$ are given by $\dim_r(\Cb_u) = \zeta^u$ and $\dim_l(\Cb_u) = \zeta^{-u}$. We take $\zeta$ to be a primitive $r$'th root of unity, so that $\Cc$ is spherical iff $r \in \{1,2\}$. One finds (Proposition~\ref{prop:SN-Zr-genusg}):

\begin{proposition}\label{prop:intro-Zr-spaces}
Let $\Sigma_g$ be a surface of genus $g$ with empty boundary. Then	
\be
	\dim \SN(\Sigma_g,\Cc_r) = 
	\begin{cases}
	r^{2g} &;~ 2-2g \text{ divisible by } r \\
	0 &;~ \text{else}
\end{cases}
\qquad .
\ee
\end{proposition}

This illustrates that in the non-spherical case, property 3 from above is in general violated for $g \neq 1$. One might also notice that the above condition is the same as that for the existence of an $r$-spin structure on $\Sigma_g$. Let us denote the set of isomorphism classes of $r$-spin structures on $\Sigma_g$ by $\R^r(\Sigma_g)$. Then the above dimension agrees with the number of $r$-spin structures, $\dim\SN(\Sigma_g,\Cc_r) = |\R^r(\Sigma_g)|$ for all $g\ge 0$. There is a natural action of $\MCG(\Sigma_g)$ -- the mapping class group of $\Sigma_g$ -- by push-forward on $r$-spin structures and on graphs on $\Sigma_g$ . It turns out that these actions agree (Theorem~\ref{thm:MCG-acts-on-subspace}).

\begin{theorem}
The $\MCG(\Sigma_g)$-representations $\mathrm{span}_{\Cb}(\R^r(\Sigma_g))$ and
$\SN(\Sigma_g,\Cc_r)$ are isomorphic.
\end{theorem}

A future application of this isomorphism that one could hope for is that the (variant of a) modular functor obtained from $\SN(-,\Cc_r)$ could be used in the construction of two-dimensional conformal field theories defined on Riemann surfaces equipped with an $r$-spin structure, i.e.\ with an $r$'th root of the canonical bundle.

\medskip

The second observation concerns the case where we are given a modular fusion category\footnote{
	The term ``modular tensor category'' is sometimes used to include cases where the underlying monoidal category is not necessarily semisimple, but instead is a finite tensor category in the sense of \cite{Etingof:2004}. Here we work only in the semisimple setting and stress this by using the term modular fusion category.
	\label{fn:mod-fus-cat}
}	 
$\Mc$ and take $\Cc = \Mc$ as fusion categories. The possible pivotal structures on $\Mc$ are a torsor over the group of natural monoidal isomorphism of the identity functor on $\Mc$, and this group in turn is isomorphic to the group of (isoclasses of) invertible simple objects in $\Mc$ \cite{Drinfeld:0906}. Let $J \in \Mc$ be such an  invertible simple object and equip $\Cc$ with the pivotal structure defined by $J$. 
The right and left quantum dimension of an object $U \in \Cc$ are
$\dim_r^\Cc(U) = s^{\Mc}_{J,U}/s^{\Mc}_{J,\one}$ and $\dim_l^\Cc(U) = s^{\Mc}_{J^\vee,U}/s^{\Mc}_{J^\vee,\one}$, where $s^{\Mc}_{X,Y} = \tr^\Mc(c_{Y,X} \circ c_{X,Y})$ refers to the invariant of the Hopf link computed in $\Mc$. It follows that $\Cc$ is spherical iff the order of $J$ is 1 or 2. Since $\Mc$ is modular, we have $\Zc(\Cc) \cong \Cc \boxtimes \Cc^{\text{rev}}$, so that we can label simple objects in $\Zc(\Cc)$ by pairs $(U,V)$ with $U,V \in \Cc$ simple. Denote by $S^2(U,V)$ a sphere with one marked point labelled by $(U,V) \in \Zc(\Cc)$. 
If one extends the string-net construction to include marked points labelled by objects of $\Zc(\Cc)$ as in \cite{Kirillov:2011mk}, one finds (Proposition~\ref{prop:S2-ZC-markedpoint}):

\begin{proposition}
$\SN(S^2(U,V),\Cc)$ is one-dimensional if $U \cong V^\vee \cong J\otimes J$ and zero-dimensional else.
\end{proposition}

The precise definition of $\SN(S^2(U,V),\Cc)$ used here is given in Section~\ref{sec:bg-charge}. If the marked point is labelled by the tensor unit $\one \in \Zc(\Cc)$, one obtains the string-net space on the sphere without marked points. 
Since $J \otimes J \cong \one$ iff $\Cc$ is spherical, this agrees with Proposition~\ref{prop:intro-S2dim}. The interesting aspect of this result is that the dimension can be nonzero for a non-unit insertion.

As for the first observation, the above result may have applications in the construction of two-dimensional conformal field theories. The focus would be on theories with deformed stress-energy tensor (cf.\ Remark~\ref{rem:bg-charge}). A prominent example of this are topologically twisted $N=2$ supersymmetric conformal field theories. The appearance of non-spherical pivotal categories in this context has been observed in \cite{Carqueville:2010hu}. In fact, in the examples treated there one finds the non-spherical pivotal category $\Cc_r$ from the first observation as a pivotal subcategory of a certain category of matrix factorisations.

\medskip

The two potential applications to conformal field theory mentioned above rest on the 
construction of consistent systems of correlators in the sense of \cite{Fjelstad:2005ua,Fuchs:2016wjr}. The description of such systems in terms of string-net spaces is being developed in \cite{SY-prep}.

\medskip

Pivotal fusion categories which are not necessarily spherical have appeared in contexts related to string-net models in the literature before:

In the context of Turaev-Viro TQFT, 
in \cite[Sec.\,3.1]{Turaev:2010pp} vector spaces assigned to coloured graphs on surfaces are considered for pivotal fusion categories, but only before the projection to the state space.

In \cite{Morrison:2010} a construction of TQFT state spaces in terms of so-called fields is described, where the state space is obtained via a quotient by local relations. The construction can be applied to pivotal fusion categories \cite[Sec.\,2.2]{Morrison:2010} and in this case should be equivalent to the string-net construction.

It was shown in \cite{Douglas:2013aea} that fusion categories are fully dualisable in the 3-category of finite tensor categories. They thus define three-dimensional TQFTs on framed manifolds. In particular, one obtains an assignment of vector spaces to framed 2-manifolds without requiring a pivotal or spherical structure. It would interesting to understand the precise relation to the string-net construction for non-spherical pivotal fusion categories.

In \cite{MV-prep}, the action of surface mapping class groups is investigated for generalisations of Kitaev models defined in terms of pivotal (but not necessarily spherical) Hopf algebras.

String-net spaces for fusion categories with $\Zb_r$ fusion rules have been considered in \cite{Hung:2012kc,Lin:2014aca}. The pivotal structures used there are always spherical, 
but the associators are more general than those considered here.
In \cite{Hung:2012kc} some quantum dimensions are allowed to take the value $-1$. In \cite{Lin:2014aca} more general pivotal structures appear implicitly (see Eqn.\,21 there), but it is argued that for their choice of lattice one can fix all quantum dimensions to be $+1$ by a gauge transformation.

\medskip

This paper is organised as follows. In Section~\ref{sec:pivotal} some notation and properties of pivotal fusion categories are introduced. Section~\ref{sec:stringnet} contains the definition of string-net spaces, and Section~\ref{sec:string-compute} explores some of their general properties. Sections~\ref{sec:stringnet} and~\ref{sec:string-compute} are mostly a review of \cite{Kirillov:2011mk}, but some extra care has to be taken in the non-spherical setting. In Section~\ref{sec:Zr-graded-vsp} the string-net spaces for $\Zb_r$-graded vector spaces are computed, and Sections~\ref{sec:r-spin} and~\ref{sec:bg-charge} contain the application to $r$-spin structures and to spheres with one marked point, respectively. The slightly lengthy proof of a technical lemma has been moved to the appendix. 

\medskip

\noindent
{\bf Acknowledgements:} I would like to thank
	Alexei Davydov,
	Catherine Meusburger,
	Christoph Schweigert,
	L\'or\'ant Szegedy,
	G\'erard Watts,
and
	Yang Yang
for helpful discussions and comments. I am grateful to the Department of Mathematics at King's College London for hospitality during a half-year sabbatical in 2019, when this research has been conducted.

\vfill

\noindent
{\bf Convention:} Throughout this paper, $\Bbbk$ denotes an algebraically closed field of characteristic zero.

\newpage

\section{Pivotal fusion categories}\label{sec:pivotal}

Let $\Cc$ be a fusion category over $\Bbbk$, that is, $\Cc$ is a $\Bbbk$-linear finitely semisimple abelian rigid monoidal category with bilinear tensor product functor and simple tensor unit, see e.g.~\cite[Sec.\,4.1]{EGNO-book}. 
Let $\Ic$ denote a choice of representatives of the isomorphism classes of simple objects in $\Cc$.

We denote the left dual of an object $X \in\Cc$ by $X^\vee$ and write the duality maps as
\be\label{eq:left-duals}
	\evL_X: X^\vee \otimes X \to \one \quad , \qquad 
	\coevL_X : \one \to X \otimes X^\vee \ .
\ee
We assume that $\Cc$ is pivotal, i.e.\ that there is a natural monoidal isomorphism
\be
\delta : (-) \longrightarrow (-)^{\vee\vee} \ .
\ee
Since $\Cc$ is pivotal, we may take the right and left dual of an object $U$ to be identical, and we will write $U^\vee$ for both. The right duality maps are given by
\begin{align}
\evR_X &= \big[ X\otimes X^\vee \xrightarrow{\delta_X \otimes \id} X^{\vee\vee} \otimes X^\vee \xrightarrow{\evL_{\!\!X^\vee}} \one \big]
\ , \nonumber \\
\coevR_X &= \big[ \one \xrightarrow{\coevL_{\!\!X^\vee}} X^\vee \otimes X^{\vee\vee} \xrightarrow{\id \otimes \delta_X^{-1}} X^\vee \otimes X \big] 
\ .
\label{eq:right-duals}
\end{align}

Let $f : X \to X$ be a morphism in $\Cc$. The left/right trace of $f$ and the left/right dimension of $X$ are defined as (conventions taken from \cite{Barrett:1993zf})
\begin{align}
	\tr_l(f) &= 
	[\one \xrightarrow{\coevR_{\!\!X}} X^\vee \otimes X \xrightarrow{\id \otimes f} X^\vee \otimes X \xrightarrow{\evL_{\!X}} \one \big] \ ,
	\quad
	&\dim_l(X) &= \tr_l(\id_X) \ ,
\nonumber \\
	\tr_r(f) &= 
	[\one \xrightarrow{\coevL_{\!\!X}} X \otimes X^\vee \xrightarrow{f \otimes \id} X \otimes X^\vee \xrightarrow{\evR_{\!X}} \one \big] \ ,
	\quad
	&\dim_r(X) &= \tr_r(\id_X) \ .
\end{align}
The left and right dimension of $X \in \Cc$ are related by
\be\label{eq:dimlr-trlr-rel}
\dim_l(X) = \dim_r(X^\vee) \ .
\ee
The global dimension of $\Cc$ is defined as \cite[Def.\,2.5]{Muger2001a}
\be
	\Dim(\Cc) = \sum_{U \in \Ic} \dim_l(U) \dim_r(U) \ .
\ee

\begin{lemma}\label{lem:dim-non-zero}
For simple objects $U\in \Cc$ we have $\dim_{l/r}(U)\neq 0$. Furthermore, $\Dim(\Cc) \neq 0$.
\end{lemma}

The proof can be found in \cite[Sec.\,2.4]{Baki-book} and \cite[Thm.\,2.3]{Etingof2002}. 
To show $\Dim(\Cc)\neq 0$ one uses the assumption that $\mathrm{char}(\Bbbk)=0$, see \cite[Sec.\,9.1]{Etingof2002}.

A pivotal fusion category is called {\em spherical} \cite{Barrett:1993zf} if $\tr_l(h) = \tr_r(h)$ for all $X$ and $h : X \to X$, or, equivalently, if $\dim_l(U) = \dim_r(U)$ for all $U \in \Ic$.

\begin{lemma}\label{lem:dim-sum-zero}
	The following statements are equivalent.
	\begin{enumerate}
		\item
	$\Cc$ is spherical.
	\item $\sum_{U \in \Ic} \dim_l(U) \dim_l(U) \neq 0$ \ .
	\item $\sum_{U \in \Ic} \dim_r(U) \dim_r(U) \neq 0$ \ .
\end{enumerate}
\end{lemma}

\begin{proof}
Parts 2 and 3 are equivalent by \eqref{eq:dimlr-trlr-rel}. Part 1 implies Part 2 by Lemma~\ref{lem:dim-non-zero}. 

Assume now that $\sum_{U \in \Ic} \dim_l(U) \dim_l(U) \neq 0$. We will show that then $\dim_l(U) = \dim_r(U)$ for all $U \in \Ic$.

Let $A$ be the $\Bbbk$-linear Grothendieck ring of $\Cc$, i.e.\ the finite-dimensional $\Bbbk$-algebra with basis $[U]$ for $U \in \Ic$ and structure constants given by the fusion rules $N_{UV}^{~W}$. Then $A$ is a symmetric Frobenius algebra with non-degenerate pairing $N_{UV}^{~\one} = \delta_{[U],[V^\vee]}$. For any two $A$-modules $M,N$ and any linear map $h : M\to N$, 
the map $m \mapsto \sum_{U \in \Ic} [U].h\big([U^\vee].m\big)$ is an $A$-intertwiner (see e.g.\ \cite[Lem.\,(62.8)]{Curtis-Reiner-book}).
The $\Bbbk$-linear extensions $\dim_{l/r}(-) : A \to \Bbbk$ are one-dimensional representations of $A$. Applying the above results to $M,N = \Bbbk$ with representation maps $\dim_r(-)$, $\dim_l(-)$, respectively, and $h = \id_{\Bbbk}$, results in the sum $\sum_{U \in \Ic} \dim_l(U) \dim_r(U^\vee)$, which is non-zero by assumption (recall \eqref{eq:dimlr-trlr-rel}). But this implies there is a non-zero $A$-intertwiner $M \to N$, and this can only be the case if $\dim_l(-) = \dim_r(-)$.
\end{proof}

We will employ the standard graphical calculus in pivotal fusion categories, following the conventions in \cite[Sec.\,2.3]{Baki-book}. The duality maps in \eqref{eq:left-duals} and \eqref{eq:right-duals} are represented in string diagram notation as:
\be
\evL_X = 
\begin{tikzpicture}[baseline=0em,very thick,decoration={markings,mark=at position 0.5 with {\arrow{>}}}]
\draw [postaction={decorate}](0,0)  arc  (0:180:0.5);
\node at (-1,-0.3) {\small $X^\vee$};
\node at (0,-0.3) {\small $X$};
\end{tikzpicture}
~,~~
\evR_X = 
\begin{tikzpicture}[baseline=0em,very thick,decoration={markings,mark=at position 0.5 with {\arrow{<}}}]
\draw [postaction={decorate}](0,0)  arc  (0:180:0.5);
\node at (-1,-0.3) {\small $X$};
\node at (0,-0.3) {\small $X^\vee$};
\end{tikzpicture}
~,~~
\coevL_X = 
\begin{tikzpicture}[baseline=0em,very thick,decoration={markings,mark=at position 0.5 with {\arrow{>}}}]
\draw [postaction={decorate}](0,0)  arc  (0:-180:0.5);
\node at (-1,0.3) {\small $X$};
\node at (0,0.3) {\small $X^\vee$};
\end{tikzpicture}
~,~~
\coevR_X = 
\begin{tikzpicture}[baseline=0em,very thick,decoration={markings,mark=at position 0.5 with {\arrow{<}}}]
\draw [postaction={decorate}](0,0)  arc  (0:-180:0.5);
\node at (-1,0.3) {\small $X^\vee$};
\node at (0,0.3) {\small $X$};
\end{tikzpicture}
\ .
\ee
In particular, string diagrams in this paper are read from bottom to top.
The left/right traces and dimensions are
\be
\mathrm{tr}_l(f) = 
\begin{tikzpicture}[baseline=-1em,very thick]
\begin{scope}[decoration={markings,mark=at position 0.5 with {\arrow{>}}}]
\draw [postaction={decorate}](0,0)  arc  (0:180:0.5);
\end{scope}
\begin{scope}[decoration={markings,mark=at position 0.5 with {\arrow{<}}}]
\draw [postaction={decorate}](0,-0.8)  arc  (0:-180:0.5);
\end{scope}
\draw (-1,0) -- (-1,-0.8);
\draw (0,0) -- (0,-0.8);
\node [very thick,draw,outer sep=0,inner sep=4,minimum size=10, fill=white] at (0,-0.4) {\small $f$};
\node at (0.2,0.3) {\small $X$};
\end{tikzpicture}
~,~~
\mathrm{tr}_r(f) = 
\begin{tikzpicture}[baseline=-1em,very thick]
\begin{scope}[decoration={markings,mark=at position 0.5 with {\arrow{<}}}]
\draw [postaction={decorate}](0,0)  arc  (0:180:0.5);
\end{scope}
\begin{scope}[decoration={markings,mark=at position 0.5 with {\arrow{>}}}]
\draw [postaction={decorate}](0,-0.8)  arc  (0:-180:0.5);
\end{scope}
\draw (-1,0) -- (-1,-0.8);
\draw (0,0) -- (0,-0.8);
\node [very thick,draw,outer sep=0,inner sep=4,minimum size=10, fill=white] at (-1,-0.4) {\small $f$};
\node at (-1.2,0.3) {\small $X$};
\end{tikzpicture}
~,~~
\dim_l(X) = 
\begin{tikzpicture}[baseline=-.2em,very thick]
\begin{scope}[decoration={markings,mark=at position 0.5 with {\arrow{>}}}]
\draw [postaction={decorate}](0,0)  arc  (0:180:0.5);
\end{scope}
\begin{scope}[decoration={markings,mark=at position 0.5 with {\arrow{<}}}]
\draw [postaction={decorate}](0,0)  arc  (0:-180:0.5);
\end{scope}
\node at (-0.3,0) {\small $X$};
\end{tikzpicture}
~,~~
\dim_r(X) = 
\begin{tikzpicture}[baseline=-.2em,very thick]
\begin{scope}[decoration={markings,mark=at position 0.5 with {\arrow{<}}}]
\draw [postaction={decorate}](0,0)  arc  (0:180:0.5);
\end{scope}
\begin{scope}[decoration={markings,mark=at position 0.5 with {\arrow{>}}}]
\draw [postaction={decorate}](0,0)  arc  (0:-180:0.5);
\end{scope}
\node at (-0.7,0) {\small $X$};
\end{tikzpicture}
\ .
\ee

As in \cite{Kirillov:2011mk}, the pivotal structure gives isomorphisms $\Cc(\one,V_1 \otimes \cdots \otimes V_n) \to \Cc(\one,V_n \otimes V_1 \otimes \cdots \otimes V_{n-1})$ whose $n$-fold composition is the identity on $\Cc(\one,V_1 \otimes \cdots \otimes V_n)$. We will, however, not use this cyclic structure here (see Footnote~\ref{fn:differences} below).

\section{String-net construction}\label{sec:stringnet}

String-net spaces were introduced in \cite{Levin:2004mi} and were shown in \cite{Kadar:2009fs,Koenig:2010,Kirillov:2011mk} to be equal to state spaces of Turaev-Viro TQFTs. These references assume the underlying pivotal fusion category to be spherical. In this section we review the string-net construction following \cite[Sec.\,2\,\&\,3]{Kirillov:2011mk}, but without assuming sphericality.

\medskip

Let $\Cc$ be a pivotal fusion category over $\Bbbk$. To simplify notation, we will assume in addition that $\Cc$ is strict. By a {\em marked surface} we mean a tuple $\BSig = (\Sigma,B,V,\nu)$, where
\begin{itemize}
\item
$\Sigma$ is an oriented surface, possibly with non-empty boundary, and possibly non-compact.
\item
$B \subset \partial \Sigma$ is a finite subset, possibly empty, of the boundary of $\Sigma$.
\item
$V : B \to \Cc$, $b \mapsto V_b$ is a function that assigns an object of $\Cc$ to every point in $B$, and $\nu : B \to \{\pm 1\}$, $b \mapsto \nu_b$, assigns a sign.
\end{itemize}

Let $\Sigma$ be an oriented surface and let $\Gamma$ be a finite graph with oriented edges embedded in $\Sigma$.
Write $E(\Gamma)$ for the set of edges of $\Gamma$ and $V(\Gamma)$ for its set of vertices. We denote by $V_\circ(\Gamma)$ the subset of vertices which lie in the interior of $\Sigma$.
For an edge $e \in E(\Gamma)$ with non-empty boundary (i.e.\ which is not a loop), we write $\partial_-e \in V(\Gamma)$ for its boundary vertex in the direction of the orientation of $e$, and $\partial_+e \in V(\Gamma)$ for the boundary vertex in the opposite direction, see Figure~\ref{fig:graph-conventions}\,(a). If an edge starts and ends on the same vertex, we have $\partial_+e = \partial_-e$. 
For a vertex $v \in V_\circ(\Gamma)$, set
\be\label{eq:E(v)-def}
\Ec(v) = \big\{ \, (e,\nu) \, \big|\, e \in E(\Gamma) , \nu \in \{ \pm 1 \} \text{ such that } \partial_\nu e = v \,\big\} \ .	
\ee
One can think of $\Ec(v)$ as the set of half-edges attached to $v$.
The orientation of $\Sigma$ equips $\Ec(v)$ with a cyclic ordering obtained by passing clockwise around the vertex, see Figure~\ref{fig:graph-conventions}\,(b). 

\begin{figure}[tb]
\begin{center}
a)~~
\begin{tikzpicture}[baseline=8em]
\begin{scope}[very thick,blue!80!black,decoration={markings,mark=at position 0.5 with {\arrow{>}}}]
\draw[postaction={decorate}] (0,0) -- (0,2);
\draw (0,2) -- (0.5,2.5);
\draw[dashed] (0.5,2.5) -- (1,3);
\draw (0,2) -- (-0.5,2.5);
\draw[dashed] (-0.5,2.5) -- (-1,3);
\draw (0,0) -- (0.5,-0.5);
\draw[dashed] (0.5,-0.5) -- (1,-1);
\draw (0,0) -- (-0.5,-0.5);
\draw[dashed] (-0.5,-0.5) -- (-1,-1);
\end{scope}
\draw[very thick,blue!80!black,fill=blue!80!black] (0,0) circle (0.1);
\draw[very thick,blue!80!black,fill=blue!80!black] (0,2) circle (0.1);

\node at (0.3,1) {\small $e$};
\node at (0.5,1.8) {\small $\partial_- e$};
\node at (0.5,0.2) {\small $\partial_+ e$};
\end{tikzpicture}
\hspace{3em}
b)~
\begin{tikzpicture}[baseline=11em]
\begin{scope}[very thick,orange!80!black]
\draw (0,0) -- (4,0) -- (4,4.5) -- (0,4.5) -- (0,0);
\draw (0.6,-0.2) -- (0.8,0.2);
\draw (0.7,-0.2) -- (0.9,0.2);
\draw (0.6, 4.3) -- (0.8,4.7);
\draw (0.7, 4.3) -- (0.9,4.7);
\draw (-0.2,3.1) -- (0.2,3.3);
\draw (3.8,3.1) -- (4.2,3.3);
\end{scope}
\draw[very thick,blue!80!black,fill=blue!80!black] (1.5,1.5) circle (0.1);

\draw [line width=4pt,gray!60!white] (3,3.5) ellipse (0.4 and 0.4);
\draw [very thick,black] (3,3.5) ellipse (0.5 and 0.5);

\begin{scope}[very thick,blue!80!black,decoration={markings,mark=at position 0.5 with {\arrow{>}}}]
\draw[postaction={decorate}] (1.5,1.5) -- (4,1.5);
\draw[postaction={decorate}] (1.5,0) -- (1.5,1.5);
\draw[postaction={decorate}] (1.5,1.5) -- (1.5,4.5); 
\draw[postaction={decorate}] (0,1.5) -- (1.5,1.5);
\end{scope}
\begin{scope}[very thick,blue!80!black,decoration={markings,mark=at position 0.5 with {\arrow{<}}}]
\draw[postaction={decorate}] (1.5,1.5) -- (2.7,3.1);
\end{scope}
\draw[very thick,blue!80!black,fill=blue!80!black] (2.7,3.1) circle (0.1);

\node at (0.5,1.2) {\small $f_2$};
\node at (3,1.2) {\small $f_2$};
\node at (1.8,3.3) {\small $f_3$};
\node at (1.8,0.5) {\small $f_3$};
\node at (2.5,2.3) {\small $f_1$};
\node at (1.3,1.7) {\small $v$};
\node at (2.4,3) {\small $p$};
\end{tikzpicture}
\hspace{3em}
c)~
\begin{tikzpicture}[baseline=11em]
\begin{scope}[very thick,orange!80!black]
\draw (0,0) -- (4,0) -- (4,4.5) -- (0,4.5) -- (0,0);
\draw (0.6,-0.2) -- (0.8,0.2);
\draw (0.7,-0.2) -- (0.9,0.2);
\draw (0.6, 4.3) -- (0.8,4.7);
\draw (0.7, 4.3) -- (0.9,4.7);
\draw (-0.2,3.1) -- (0.2,3.3);
\draw (3.8,3.1) -- (4.2,3.3);
\end{scope}
\draw[very thick,blue!80!black,fill=blue!80!black] (1.5,1.5) circle (0.1);
\draw[very thick,blue!80!black, dashed] (1.5,1.5) -- (1.7,2.2);

\draw [line width=4pt,gray!60!white] (3,3.5) ellipse (0.4 and 0.4);
\draw [very thick,black] (3,3.5) ellipse (0.5 and 0.5);

\begin{scope}[very thick,blue!80!black,decoration={markings,mark=at position 0.5 with {\arrow{>}}}]
\draw[postaction={decorate}] (1.5,1.5) -- (4,1.5);
\draw[postaction={decorate}] (1.5,0) -- (1.5,1.5);
\draw[postaction={decorate}] (1.5,1.5) -- (1.5,4.5); 
\draw[postaction={decorate}] (0,1.5) -- (1.5,1.5);
\end{scope}
\begin{scope}[very thick,blue!80!black,decoration={markings,mark=at position 0.5 with {\arrow{<}}}]
\draw[postaction={decorate}] (1.5,1.5) -- (2.7,3.1);
\end{scope}
\draw[very thick,blue!80!black,fill=blue!80!black] (2.7,3.1) circle (0.1);

\node at (0.3,1.2) {\small $X$};
\node at (3,1.2) {\small $X$};
\node at (1.7,3.3) {\small $Y$};
\node at (1.7,0.4) {\small $Y$};
\node at (2.1,2.6) {\small $U$};
\node at (1.3,1.7) {\small $\varphi$};
\node at (3.1,2.7) {\small $(U,-)$};
\end{tikzpicture}
\end{center}

\caption{{\bf a)} The two boundary vertices of an edge.\\ 
{\bf b)} Example of a graph embedded in a torus with one boundary component. The boundary contains a marked point $p$. We have $\Ec(v) = \{ (f_1,-), (f_2,+), (f_3,-), (f_2,-), (f_3,+) \}$, where the elements are already listed in the cyclic order.\\ 
{\bf c)} Example of a coloured graph. The dashed line at the vertex $v$ separates the first and last edge of the total order of half-edges around the vertex. In this example, the total order around $v$ is such that $V_{f_1} = U$, $\nu_1 = -$, $V_{f_2}=X$, $\nu_2 = +$, etc. Furthermore
$V_p = U$, $\nu_p = -$, and $\varphi$ is a morphism in $\Cc(\one, U^\vee \otimes X  \otimes Y^\vee  \otimes X^\vee  \otimes Y)$.}
\label{fig:graph-conventions}
\end{figure}
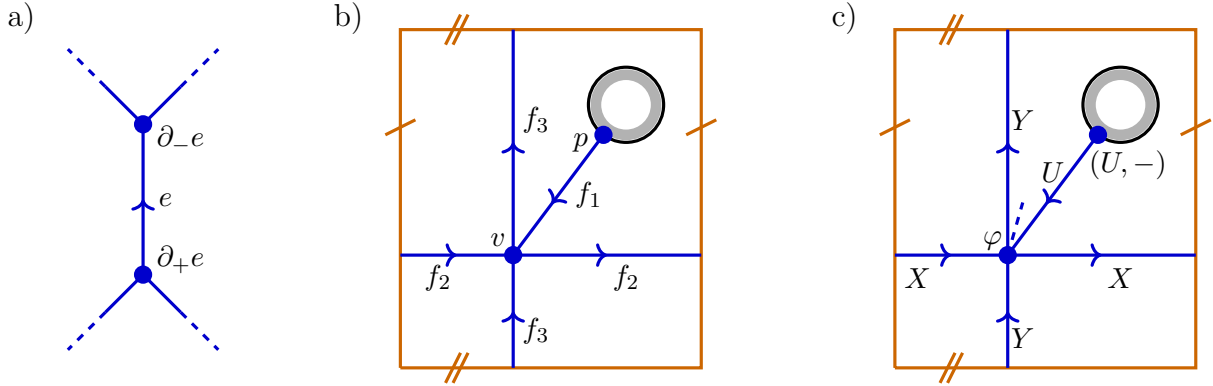

With these preparations, we can define a {\em coloured graph embedded in a marked surface} $\BSig = (\Sigma,B,V,\nu)$ to be the following collection of data:
\begin{itemize}
\item
A finite graph $\Gamma$ with with oriented edges embedded in $\Sigma$, such that $\Gamma \cap \partial\Sigma = B$.
We require that each $b \in B$ is a one-valent vertex of $\Gamma$, and that the (unique) edge containing $b$ is oriented towards $b$ if $\nu_b=+1$ and away from $b$ if $\nu_b=-1$.\footnote{
	This convention is chosen opposite to \eqref{eq:E(v)-def} and avoids extra duals in the map \eqref{eq:<>Q-def} below.}

\item
For each $v \in V_\circ(\Gamma)$ a choice of total order of $\Ec(v)$, compatible with the cyclic order described above. We can thus write
\be
\Ec(v) = \{ (e_1,\nu_1), \dots , (e_N,\nu_N) \}	
\ee
with $N = |\Ec(v)|$.

\item 
A map $E(\Gamma) \to \Cc$, $e \mapsto V_e$, assigning an object of $\Cc$ to each edge of $\Gamma$. If $b \in B$ is a boundary point of $e$, we require that $V_e = V_b$.

\item
For each interior vertex $v \in V_\circ(\Gamma)$ a choice of morphism
\be
	\varphi_v \in \Cc(\one, V_{e_1}^{\nu_1} \otimes \cdots \otimes V_{e_N}^{\nu_N}) \ ,
\ee
where $N = |\Ec(v)|$ and for $X \in \Cc$ we set $X^+ = X$ and $X^- = X^\vee$. 
\end{itemize}
See Figure~\ref{fig:graph-conventions}\,(c) for an illustration.
We write $\Graph(\BSig)$ for the set of all coloured graphs embedded in $\BSig$ and denote the $\Bbbk$-vector space it spans freely by
\be
	\VGraph(\BSig) := \mathrm{span}_{\Bbbk} \Graph(\BSig) \ .
\ee

\begin{figure}[tb]
\begin{center}
a)~~
\begin{tikzpicture}[very thick,baseline=13em]
\draw (0,0) -- (0,5) -- (5,5) -- (5,0) -- (0,0);

\coordinate (v2) at (0.9,5);
\coordinate (v3) at (1.7,5) circle (0.1);
\coordinate (v4) at (2.5,5) circle (0.1);
\coordinate (v7) at (3.3,5) circle (0.1);
\coordinate (v6) at (4.1,5) circle (0.1);
\coordinate (v1) at (2,2) circle (0.1);
\coordinate (v5) at (3.5,3) circle (0.1);

\node at (v2) [above] {\small $b_1$};
\node at (v3) [above] {\small $b_2$};
\node at (v4) [above] {\small $b_3$};
\node at (v7) [above] {\small $b_4$};
\node at (v6) [above] {\small $b_5$};
\node at (v1) [left] {\small $v_1$};
\node at (v5) [right] {\small $v_2$};

\draw[very thick,blue!80!black,fill=blue!80!black] (v2) circle (0.1);
\draw[very thick,blue!80!black,fill=blue!80!black] (v3) circle (0.1);
\draw[very thick,blue!80!black,fill=blue!80!black] (v4) circle (0.1);
\draw[very thick,blue!80!black,fill=blue!80!black] (v7) circle (0.1);
\draw[very thick,blue!80!black,fill=blue!80!black] (v6) circle (0.1);

\draw[very thick,blue!80!black,fill=blue!80!black] (v1) circle (0.1);
\draw[very thick,blue!80!black,fill=blue!80!black] (v5) circle (0.1);

\begin{scope}[very thick,blue!80!black,decoration={markings,mark=at position 0.6 with {\arrow{>}}}]
\draw [postaction={decorate}] (1,1) ellipse (0.5 and 0.5);
\draw[postaction={decorate}] (2,2) -- node[above,black] {\small $f_5$} (v5);
\draw[postaction={decorate}]  plot[smooth, tension=.7] coordinates {(v1) (1,3) (v2)};
\draw[postaction={decorate}]  plot[smooth, tension=1.9] coordinates {(v4) (2,3.5) (v3)};
\draw[postaction={decorate}]  plot[smooth, tension=.7] coordinates {(v1) (3.2063,2.0382) (v5)};
\draw[postaction={decorate}]  plot[smooth, tension=1.1] coordinates {(v1) (2.8,1) (4,2.2) (v6)};
\draw[postaction={decorate}] (v7) -- node[right,black] {\small $f_3$} (v5);
\draw[dashed] (v1) -- ++(-30:0.7);
\draw[dashed] (v5) -- ++(-120:0.7);
\end{scope}

\draw[very thick,blue!80!black,fill=blue!80!black] (2,2) node (v1) {} circle (0.1);
\draw[very thick,blue!80!black,fill=blue!80!black] (3.5,3) node (v5) {} circle (0.1);

\node at (1.75,1) {\small $f_7$};
\node at (0.6,3.8) {\small $f_1$};
\node at (1.5,4.1) {\small $f_2$};
\node at (4.4,2.7) {\small $f_4$};
\node at (3.0,1.7) {\small $f_6$};
\end{tikzpicture}
\hspace{4em}
b)
\begin{tikzpicture}[baseline=13.5em]
\coordinate (v1) at (1,5);
\coordinate (v2) at (2,5);
\coordinate (v3) at (3,5);
\coordinate (v4) at (4,5);
\coordinate (v5) at (5,5);

\node at (v1) [above] {\small $V_1$};
\node at (v2) [above] {\small $V_2$};
\node at (v3) [above] {\small $V_2^\vee$};
\node at (v4) [above] {\small $V_3^\vee$};
\node at (v5) [above] {\small $V_4$};

\draw[very thick] (1,1.5) rectangle ++(1.5,0.6) node[pos=.5] {\small $\varphi_{v_1}$};
\draw[very thick] (3.2,3) rectangle ++(1.2,0.6) node[pos=.5] {\small $\varphi_{v_2}$};

\draw[very thick] (v1) .. controls (1,4) and (1.4,3) .. (1.6,2.1);
\begin{scope}[very thick,decoration={markings,mark=at position 0.51 with {\arrow{<}}}]
\draw[postaction={decorate}] (v2) .. controls (1.9761,3.746) and (2.9938,3.7637) .. (v3);
\end{scope}
\draw[very thick] (3.8,3.6) .. controls (3.7637,4.2947) and (3.9938,4.3035) .. (v4);

\begin{scope}[very thick,decoration={markings,mark=at position 0.8 with {\arrow{<}}}]
\draw [postaction={decorate}] (v5) .. controls (5,4) and (5.2327,1.9673) .. (4.5248,1.454) .. controls (3.631,0.8522) and (1.631,0.8364) .. (0.8,1.1);
\end{scope}

\begin{scope}[very thick,decoration={markings,mark=at position 0.75 with {\arrow{<}}}]
\draw [postaction={decorate}]  (0.8,1.1) .. controls (0.1832,1.4363) and (1.1656,3.3902) .. (1.3,2.1);
\end{scope}

\begin{scope}[very thick,decoration={markings,mark=at position 0.2 with {\arrow{<}}}]
\draw[postaction={decorate}] (3.5,3.6) .. controls (3.1177,4.7283) and (1.8558,2.9832) .. (1.9,2.1);
\end{scope}

\begin{scope}[very thick,decoration={markings,mark=at position 0.13 with {\arrow{<}}}]
\draw[postaction={decorate}] (4.1,3.6) .. controls (4.2,4.2) and (4.7647,4.4566) .. (4.7894,3.0036) .. controls (4.7894,2.1637) and (2.3,2.6) .. (2.2,2.1);
\end{scope}

\begin{scope}[very thick,decoration={markings,mark=at position 0.5 with {\arrow{>}}}]
\draw[postaction={decorate}](0.8,0.5)  arc  (0:180:0.5);
\end{scope}

\begin{scope}[very thick,decoration={markings,mark=at position 0.5 with {\arrow{<}}}]
\draw [postaction={decorate}](0.8,0.5)  arc  (0:-180:0.5);
\end{scope}

\node at (2.6,3.1) {\small $V_5$};
\node at (3.5,2.1) {\small $V_6$};
\node at (1.05,0.5) {\small $V_7$};
\end{tikzpicture}
\end{center}

	\caption{a) Coloured graph in $\mathbf{Q}$. The objects $V_1,\dots,V_7 \in \Cc$ are assigned to edges such that $V_i = V_{f_i}$. For the boundary points this implies for example that $V_{b_3}=V_2$ ($\nu_3=-$) and $V_{b_4} = V_3$ ($\nu_4=-$).
	b) The corresponding string diagram describing a morphism $\one \to V_1 \otimes V_2 \otimes V_2^\vee \otimes V_3^\vee \otimes V_4$ in $\Cc$.}
	\label{fig:Qgraph-stringdiag}
\end{figure}
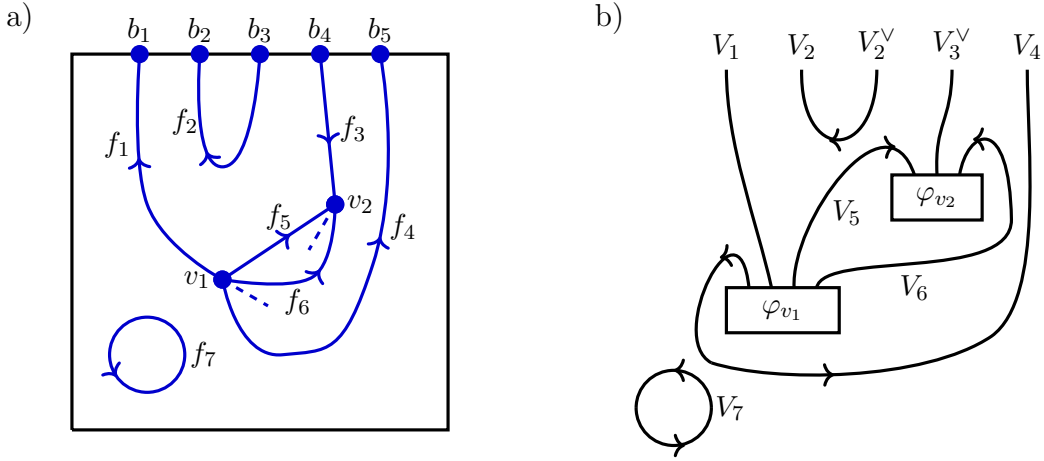

Let $\mathbf{Q} = (Q,B,V,\nu)$ be the marked surface given by the unit square $Q := [0,1]^2$ such that all marked points are contained in the interior of the top horizontal part of the boundary, $B \subset (0,1) \times \{1\}$. The points in $B$ are ordered by their $x$-coordinate, so that we can write $B = \{ b_1,\dots,b_N \}$. We can turn a coloured graph in $Q$ into a string diagram by replacing all interior vertices by the corresponding coupons, see Figure~\ref{fig:Qgraph-stringdiag} for an illustration. This string diagram evaluates to a morphism in $\Cc$ (see e.g.~\cite[Sec.\,2]{Turaev-Virelizier-book}). After linear extension we obtain the map
\be\label{eq:<>Q-def}
	\langle - \rangle_{\mathbf{Q}} : \VGraph(\mathbf{Q}) \longrightarrow \Cc(\one, V_{b_1}^{\nu_1} \otimes \cdots \otimes V_{b_N}^{\nu_N}) \ ,
\ee
By construction, for $\varphi_v \in \Cc(\one, V_{b_1}^{\nu_1} \otimes \cdots \otimes V_{b_N}^{\nu_N})$ we have
\be
	\langle \Gamma \rangle_{\mathbf{Q}} = \varphi_v 
\qquad
\text{where} \qquad
\Gamma
=~
\begin{tikzpicture}[baseline=1.2em]
\coordinate (p) at (1.4,0.3);
\coordinate (b1) at (0,2);
\coordinate (b2) at (0.8,2);
\coordinate (dot) at (1.7,2);
\coordinate (dot2) at (1.6,1.5);
\coordinate (bN) at (2.5,2);
\coordinate (v3) at (3,5);
\coordinate (v4) at (4,5);
\coordinate (v5) at (5,5);

\draw[very thick] (-0.4,-1) rectangle (2.9,2);

\draw[very thick,blue!80!black] (p) -- (b1);
\draw[very thick,blue!80!black] (p) -- (b2);
\draw[very thick,blue!80!black] (p) -- (bN);

\draw[very thick,blue!80!black,fill=blue!80!black] (p) circle (0.1);
\draw[very thick,blue!80!black,dashed] (p) -- ++(-90:0.7);
\node[right] at (p) {\small $\varphi_v$};

\draw[very thick,blue!80!black,fill=blue!80!black] (b1) circle (0.1);
\node[above] at (b1) {\small $V_1^{\nu_1}$};

\draw[very thick,blue!80!black,fill=blue!80!black] (b2) circle (0.1);
\node[above] at (b2) {\small $V_2^{\nu_2}$};
\node[above] at (dot) {\small $\cdots$};
\node at (dot2) {\small $\cdots$};

\draw[very thick,blue!80!black,fill=blue!80!black] (bN) circle (0.1);
\node[above] at (bN) {\small $V_N^{\nu_N}$};

\end{tikzpicture}	
\qquad .
\ee

Let $\BSig$ be a marked surface.
An element $\mathbf{\Gamma} = c_1 \Gamma_1 + \cdots + c_n \Gamma_n \in \VGraph(\BSig)$, $c_i \in \Bbbk$, is called a {\em null graph for $\BSig$} if there exists an embedding $f : Q \to \Sigma$ such that the following properties hold:
\begin{itemize}
	\item The coloured graphs $\Gamma_i$, $i=1,\dots,n$ are equal in the complement of $f(Q)$ in $\Sigma$.

	\item Each graph $\Gamma_i$ intersects the boundary $f(\partial Q)$ transversally (in particular there is no vertex on $f(\partial Q)$), 
	and there is a finite subset $B = \{b_1,\dots,b_n\}$ of the interior of the upper edge $(0,1) \times \{1\} \subset Q$ such that each $i$ we have $\Gamma_i \cap f(\partial Q) = f(B)$.
	
	\item 
	By assumption the label $V_e$ and the orientation of an edge $e$ intersecting $x \in f(B)$ is the same for all $\Gamma_i$. This allows one to define the maps $V : B \to \Cc$ and $\nu : B \to \{\pm 1\}$ to turn $Q$ into a marked surface $\mathbf{Q} = (Q,B,V,\nu)$.

	We can now pull back the restriction of $\Gamma_i$ to $f(Q)$ along $f$ to obtain an element $f^{-1}(\Gamma_i) \in \Graph(\mathbf{Q})$. We have $\langle f^{-1}(\Gamma_i) \rangle_{\mathbf{Q}} \in \Cc(\one, V_{b_1}^{\nu_1} \otimes \cdots \otimes V_{b_n}^{\nu_n})$ for all $i$. For $\mathbf{\Gamma}$ to be a null graph we require that
\be
	\sum_{i=1}^n c_i\,\langle f^{-1}(\Gamma_i) \rangle_{\mathbf{Q}}  \,=\, 0 \ .
\ee	
\end{itemize}
Let $N(\BSig)$ be the sub-vector space of $\VGraph(\BSig)$ spanned by all null graphs for $\BSig$. We define:\footnote{\label{fn:differences}
	The construction of string-net spaces presented here closely follows that in \cite[Sec.\,2\,\&\,3]{Kirillov:2011mk}. 
	One difference is that we do not require $\Cc$ to be spherical, but this does not affect the definition of $\SN(\BSig)$. 
	Another difference is that we do not make use of the cyclic symmetry of 
	$\Cc(\one,V_1 \otimes \cdots \otimes V_n)$ 	mentioned in the end of Section~\ref{sec:pivotal} and instead use a total order on $\Ec(v)$. Otherwise one would e.g.\ be forced to use cyclically invariant states if all objects and orientations at a vertex coincide, which may be a proper subspace of $\Cc(\one,V \otimes \cdots \otimes V)$, see e.g.\ \cite{Lin:2014aca}.
	}

\begin{definition}
The {\em string-net space} $\SN(\BSig)$ of a marked surface $\BSig$ is the quotient vector space
\be
\SN(\BSig) = \VGraph(\BSig) / N(\BSig) \ .
\ee
If we want to emphasise the dependence on the pivotal fusion category $\Cc$, we write $\SN(\BSig,\Cc)$.
\end{definition}

One verifies that the map $\langle - \rangle_{\mathbf{Q}}$ from \eqref{eq:<>Q-def} descends to a linear map $\SN(\mathbf{Q}) \to  \Cc(\one, V_{b_1}^{\nu_1} \otimes \cdots \otimes V_{b_N}^{\nu_N})$, which then by construction must be an isomorphism.

\begin{remark}\label{rem:mcg-action}
Let $\BSig$ and $\BSig'$ be two marked surfaces, and let $\phi : \Sigma \to \Sigma'$ be an isomorphism of the underlying surfaces that is compatible with the marking (i.e.\ $\phi(B) = B'$, etc.). We say that $\phi : \BSig \to \BSig'$ is an {\em isomorphism of marked surfaces}. Such a $\phi$ induces an isomorphism $\Graph(\BSig) \to \Graph(\BSig')$ which maps a coloured graph $\Gamma$ to $\phi(\Gamma)$. The linear extension of this map takes $N(\BSig)$ to $N(\BSig')$ and so induces an isomorphism
\be
	\SN(\phi) : \SN(\BSig) \to \SN(\BSig') \ .
\ee
In this way $\SN(-)$ becomes a functor from marked surfaces and isomorphisms to $\Bbbk$-vector spaces. Denote by $\MCG(\BSig)$ the mapping class group of $\Sigma$ which preserves the boundary of $\Sigma$ pointwise. Then $\SN(-)$ induces a group homomorphism
\be
	\MCG(\BSig) \longrightarrow GL(\SN(\BSig)) \ ,
\ee
i.e.\ one obtains a linear representation of $\MCG(\BSig)$ on $\SN(\BSig)$ (and not just a projective representation).
\end{remark}

\section{Computing string-net spaces}\label{sec:string-compute}

In this section we recall from \cite{Kirillov:2011mk} how to relate the string-net spaces of punctured and unpunctured surfaces. 
As before, $\Cc$ is a strict pivotal fusion category over $\Bbbk$. 

\medskip

The string-net space of a punctured surface is determined in the same way as in the case of spherical categories. It is helpful to describe the construction using the central monad, so we recall its definition first, see \cite{Day:2007} and also \cite{Bruguieres:2008vz}, \cite[Sec.\,3.2]{Shimizu:1402sp}. The central monad is an endofunctor of $\Cc$, which on an object $X \in \Cc$ is defined as the coend
\be
	A(X) = \int^{V \in \Cc} \hspace{-1em} V^\vee \otimes X \otimes V
	\quad , \quad
	\iota(X)_V : V^\vee \otimes X \otimes V \to A(X) \ .
\ee
By definition, the dinatural transformation $\iota(X)$ satisfies the following commuting diagram for every $f : V \to W$ in $\Cc$:
\be
	\begin{tikzcd}
	W^\vee \otimes X \otimes V \ar{r}{\id \otimes f} \ar{d}{f^\vee \otimes \id} & 
	W^\vee \otimes X \otimes W \ar{d}{\iota(X)_W} 
	\\
	V^\vee \otimes X \otimes V \ar{r}{\iota(X)_V} &
	A(X)
	\end{tikzcd}
\ee
Furthermore, the pair $(A(X),\iota(X))$ is universal with this property. That is, given another family of maps $\phi_V : V^\vee \otimes X \otimes V \to Z$ for some $Z$ which satisfies the corresponding commuting diagram, there exists a unique map $f : A(X) \to Z$ such that $\phi_V = f \circ \iota(X)_V$.

In our situation, $\Cc$ is fusion, and we can write $A(X)$ and $\iota(X)$ explicitly as
\begin{align}
A(X) &= \bigoplus_{U \in \Ic} U^\vee \otimes X \otimes U
\ , 
\nonumber\\
\iota(X)_V &= \sum_{U \in \Ic} \sum_\alpha \big[ V^\vee \otimes X \otimes V
\xrightarrow{{\bar\alpha}^\vee \otimes \id \otimes \alpha} U^\vee \otimes X \otimes U \big] \ ,
\label{eq:A(X)-via-simples}
\end{align}
where the $\alpha$ are basis vectors of $\Cc(V,U)$ and $\bar\alpha$ denotes elements of the dual basis of $\Cc(U,V)$ in the sense that $\alpha \circ \bar\beta = \delta_{\alpha,\beta} \, \id_U$. One verifies that $\iota(X)_V$ is independent of the choice of basis.

The qualifier ``central'' in central monad derives from the property that there is a natural lift of $A$ to the Drinfeld centre of $\Cc$: there is a functor $\hat A:\Cc \to \Zc(\Cc)$ such that $A = \forget \circ \hat A$, 
where $\forget : \Zc(\Cc) \to \Cc$ forgets the half-braiding. This amounts to equipping each $A(X)$ with a half-braiding, see \eqref{eq:Ahat-halfbraid} for an explicit expression. The lift $\hat A$ is important because it is left adjoint to the forgetful functor, 
\be\label{eq:forget-hatA-adjunction}
	\Zc(\Cc)(\hat A(X),Z)
	~\cong~ \Cc( X,\forget(Z) ) \ .
\ee

Apart from $A(X)$ we will need another coend. For brevity, 
here and below we will sometimes omit $\otimes$ between objects in expressions with several tensor factors. We define
\be
H = \int^{X,Y \in \Cc} \hspace{-2em} X^\vee Y^\vee X Y
= \int^{X \in \Cc} \hspace{-1em} X^\vee A(X) 
\quad , 
\quad
\jmath_{X,Y} :  X^\vee Y^\vee X Y \longrightarrow H \ .
\ee
The explicit expression in terms of simple objects reads
\be\label{eq:H-via-simples}
H = \bigoplus_{S,T \in \Ic} S^\vee T^\vee S  T
\quad , 
\quad
\jmath_{X,Y} = \sum_{S,T \in \Ic} \sum_{\alpha,\beta} \big[
X^\vee Y^\vee X Y
\xrightarrow{{\bar\alpha}^\vee \otimes{\bar\beta}^\vee \otimes \alpha \otimes \beta} S^\vee T^\vee S  T \big] \ ,
\ee
where as above, $\alpha,\beta$ are basis vectors, this time of $\Cc(X,S)$ and $\Cc(Y,T)$, respectively.

\medskip

\begin{figure}[bt]
a)\hspace{-1.5em}
\begin{tikzpicture}[very thick,baseline=6em,scale=1.2]
\node at (0,0) {\includegraphics[width=18em]{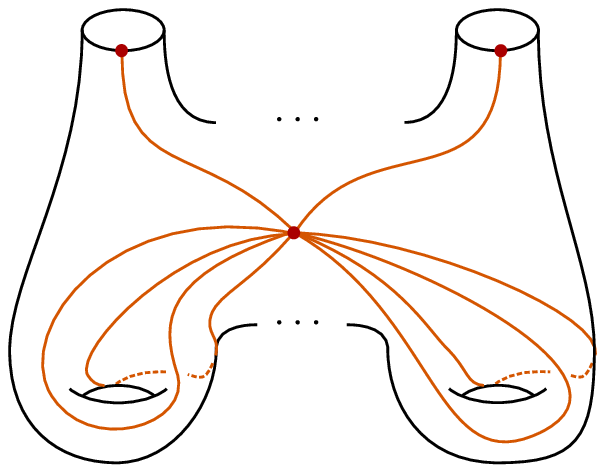}};
\node at (2.1,2.15) {\small $u_1$};
\node at (2,0.9) {\small $e_1$};
\node at (-1.8,2.15) {\small $u_b$};
\node at (-1.8,0.9) {\small $e_b$};
\node at (-2.4,-0.2) {\small $f_1$};
\node at (-0.4,-0.6) {\small $f_1'$};
\node at (0.4,-0.6) {\small $f_g$};
\node at (2.4,-0.3) {\small $f_g'$};
\node at (-0.05,0.45) {\small $p$};
\end{tikzpicture}
\hspace{1em}
b)
\begin{tikzpicture}[very thick,baseline=6em,scale=1.2]
\node at (0,0) {\includegraphics[width=18em]{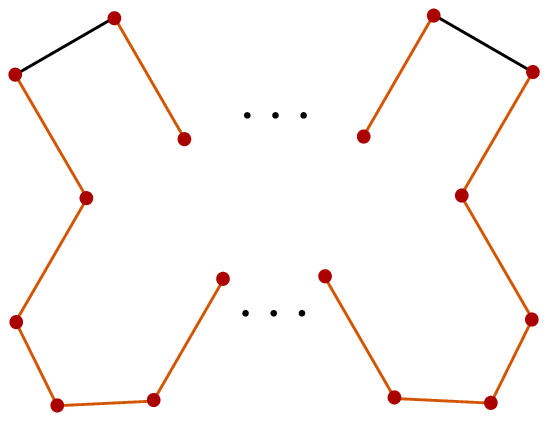}};
\node at (-1.9,2.5) {\small $u_1$};
\node at (-3.0,1.85) {\small $u_1$};
\node at (1.9,2.5) {\small $u_b$};
\node at (3.0,1.85) {\small $u_b$};
\node at (-2.9,0.75) {\small $e_{1,1}$};
\node at (-2.2,1.7) {\small $e_{1,2}$};
\node at (-1.05,1.5) {\small $e_{1,3}$};
\node at (2.9,0.75) {\small $e_{b,3}$};
\node at (2.2,1.75) {\small $e_{b,2}$};
\node at (1.05,1.5) {\small $e_{b,1}$};
\node at (-3.0,-0.5) {\small $f_{g,4}$};
\node at (-3.1,-1.8) {\small $f_{g,3}$};
\node at (-1.9,-2) {\small $f_{g,2}$};
\node at (-0.75,-1.7) {\small $f_{g,1}$};
\node at (3.0,-0.5) {\small $f_{1,1}$};
\node at (3.1,-1.8) {\small $f_{1,2}$};
\node at (1.9,-2) {\small $f_{1,3}$};
\node at (0.75,-1.7) {\small $f_{1,4}$};
\node at (-1.1,0.55) {\small $p$};
\node at (-2,0.1) {\small $p$};
\node at (-2.75,-1.3) {\small $p$};
\node at (-2.75,-2.35) {\small $p$};
\node at (-1.15,-2.3) {\small $p$};
\node at (-0.6,-0.5) {\small $p$};
\node at (1.05,0.55) {\small $p$};
\node at (1.95,0.1) {\small $p$};
\node at (2.8,-1.3) {\small $p$};
\node at (2.75,-2.35) {\small $p$};
\node at (1.15,-2.3) {\small $p$};
\node at (0.6,-0.5) {\small $p$};
\end{tikzpicture}

\caption{a) Genus $g$ surface $\Sigma$ with $b$ boundary components together with a selection of arcs starting on $p$ and ending either on $p$ or on one of the $u_i$, referred to as ``cutting arcs''; b) the polygon $P$ obtained after cutting along the arcs in a). Note that in order to retain paper-plane orientation, the figure was rotated back to front. To obtain a) from b) one glues e.g.\ $e_{1,1}$ and $e_{1,3}$ in b), resulting in edge $e_1$ in a), $f_{1,1}$ and $f_{1,3}$ in b) to obtain $f_1$ in a), and $f_{1,2}$ and $f_{1,4}$ to obtain $f_1'$.}
\label{fig:surface-polygon}
\end{figure}

After these preparations, we can describe string-net spaces for punctured surfaces.
Let $\Sigma$ be a compact surface of genus $g$ with $b \ge 0$ boundary components. Pick a point $u_i$, $i = 1 ,\dots, b$ on each boundary component and a point $p$ in the interior of the surface. Choose arcs as shown in Figure~\ref{fig:surface-polygon}\,(a). We will refer to these as ``cutting arcs''. Let $\BSig = (\Sigma,B,V,\nu)$ be a marked surface such that none of the $u_i$ is contained in $B$.
 Write 
\be
	V_{(i)} = \bigotimes_x V_x^{\nu_x} \ ,
\ee
where the tensor product is over all $x \in B$ which lie on the same boundary component as $u_i$, with the tensor product taken in anti-clockwise order starting at $u_i$ (when the neighbourhood of the boundary component is realised as $\Rb^2$ with a disc removed, see Figure~\ref{fig:poygon-Q-embedding} for an example). 
Abbreviate
\be
	H(\BSig) = \Cc\Big(\,\one \,,\, \Big({\textstyle \bigotimes_{i=1}^b} \, A(V_{(i)}) \Big) \otimes H^{\otimes g} \, \Big) \ .
\ee
Let $G \subset \Graph(\BSig - \{p\})$ consist of coloured graphs whose edges cross the cutting arcs transversally, and for which no vertices lie on these arcs. 
Define the map
\be\label{eq:psi-vgraph}
\tilde\psi : G
\longrightarrow H(\BSig)
\ee
as follows. Let $\Gamma \in G$. Denote the polygon in Figure~\ref{fig:surface-polygon}\,(b) by $P$. For each edge $e$ of $P$, let $W_e$ be the tensor product of the objects (or their duals, depending on orientation) labelling the edges of the graph $\Gamma$ intersecting $e$, see Figure~\ref{fig:poygon-Q-embedding}\,(a). In particular, if $e$ is the $i$'th boundary component of $\BSig$, then $W_e = V_{(i)}$.
Pick an embedding $f : Q \to \Sigma$ as shown in Figure~\ref{fig:poygon-Q-embedding}\,(b). We can now pull back the part of $\Gamma$ that lies in $f(Q)$ to obtain a coloured graph in $\Graph(\mathbf{Q})$. Evaluating this graph gives a morphism
\be\label{eq:evaluate-graph-in-P}
\langle f^{-1}(\Gamma) \rangle_{\BQ}
: \one \to 
W_{e_{1,1}} \otimes \cdots \otimes W_{e_{b,3}} \otimes W_{f_{1,1}} \otimes \cdots \otimes W_{f_{g,4}} \ .
\ee

\begin{figure}[bt]
	
a)\hspace{-1.8em}
\begin{tikzpicture}[very thick,baseline=7em,scale=1.2]
\node at (0,0) {\includegraphics[width=18em]{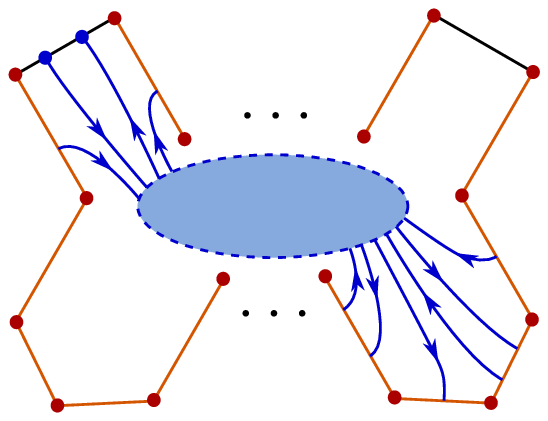}};
\node at (-2.9,0.75) {\small $e_{1,1}$};
\node at (-2.6,2.1) {\small $e_{1,2}$};
\node at (-1.05,1.5) {\small $e_{1,3}$};
\node at (2.9,0.75) {\small $e_{b,3}$};
\node at (2.6,2.1) {\small $e_{b,2}$};
\node at (1.05,1.5) {\small $e_{b,1}$};
\node at (-3.0,-0.5) {\small $f_{g,4}$};
\node at (-3.1,-1.8) {\small $f_{g,3}$};
\node at (-1.9,-2.5) {\small $f_{g,2}$};
\node at (-0.75,-1.7) {\small $f_{g,1}$};
\node at (3.0,-0.5) {\small $f_{1,1}$};
\node at (3.1,-1.8) {\small $f_{1,2}$};
\node at (1.9,-2.5) {\small $f_{1,3}$};
\node at (0.75,-1.7) {\small $f_{1,4}$};

\node at (0,0) {\small $\Gamma$};
\node at (-1.9,0.3) {\footnotesize $A$};
\node at (-2.6,1.4) {\footnotesize $B$};
\node at (-2,1.3) {\footnotesize $C$};
\node at (-1.25,1.0) {\footnotesize $A$};
\node at (2,-0.2) {\footnotesize $D$};
\node at (2.5,-1.1) {\footnotesize $E$};
\node at (2.35,-1.5) {\footnotesize $F$};
\node at (1.8,-1.9) {\footnotesize $D$};
\node at (1.35,-1.2) {\footnotesize $F$};
\node at (0.8,-0.7) {\footnotesize $E$};
\end{tikzpicture}
\hspace{0.5em}	
b)\hspace{-1.3em}
\begin{tikzpicture}[very thick,baseline=7em,scale=1.2]
\node at (0,0) {\includegraphics[width=18em]{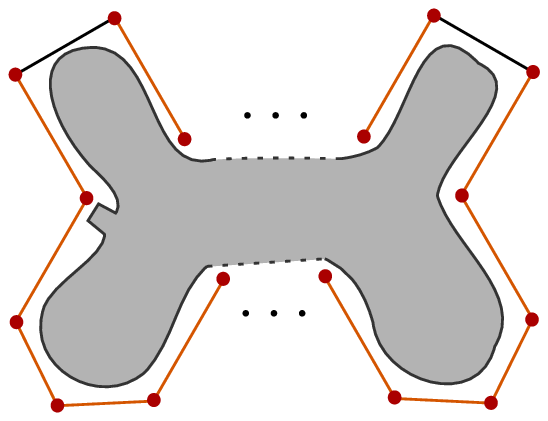}};
\node at (-2.9,0.75) {\small $e_{1,1}$};
\node at (-2.6,2.1) {\small $e_{1,2}$};
\node at (-1.05,1.5) {\small $e_{1,3}$};
\node at (2.9,0.75) {\small $e_{b,3}$};
\node at (2.6,2.1) {\small $e_{b,2}$};
\node at (1.05,1.5) {\small $e_{b,1}$};
\node at (-3.0,-0.5) {\small $f_{g,4}$};
\node at (-3.1,-1.8) {\small $f_{g,3}$};
\node at (-1.9,-2.5) {\small $f_{g,2}$};
\node at (-0.75,-1.7) {\small $f_{g,1}$};
\node at (3.0,-0.5) {\small $f_{1,1}$};
\node at (3.1,-1.8) {\small $f_{1,2}$};
\node at (1.9,-2.5) {\small $f_{1,3}$};
\node at (0.75,-1.7) {\small $f_{1,4}$};
\end{tikzpicture}

	\caption{a) 
Example of a coloured graph on the surface $\Sigma$ in Figure~\ref{fig:surface-polygon}\,(a), drawn on the polygon $P$. The dashed shaded disc stands for any coloured graph. In this example, $V_{(1)} = B^\vee \otimes C$, $W_{f_{1,2}} = E \otimes F^\vee$, and $W_{f_{1,4}} = F \otimes E^\vee$. Note that $W_{f_{1,2}} \cong (W_{f_{1,4}})^\vee$.
\\
b) 
Example of an embedding $f : Q \to \Sigma$ with the image $f(Q)$ shown in grey. The embedding is such that the top edge of $Q$ passes near the edges $e_{1,1},e_{1,2},\dots,f_{g,4}$. It has to contain all inner vertices of the coloured graph and only the top edge of $Q$ is allowed to intersect edges of the coloured graph.
}
	\label{fig:poygon-Q-embedding}
\end{figure}

For the $i$'th boundary component of $\Sigma$, 
the polygon $P$ has three edges
$e_{i,1},e_{i,2},e_{i,3}$. 
To obtain $\Sigma$ from $P$ one needs to glue $e_{i,1}$ to $e_{i,3}$. We have $W_{e_{i,2}} = V_{(i)}$, and
there is a canonical isomorphism $W_{e_{i,1}} \xrightarrow{\sim} W_{e_{i,3}}^\vee$ (see Figure~\ref{fig:poygon-Q-embedding}\,(a) for an example). Define the morphism
\be
	E_i = \big[ W_{e_{i,1}} W_{e_{i,2}}  W_{e_{i,3}}  = W_{e_{i,1}} V_{(i)}  W_{e_{i,3}} \xrightarrow{\sim} 
	W_{e_{i,3}}^\vee  V_{(i)}  W_{e_{i,3}}
	\xrightarrow{\iota(V_{(i)})_{W_{e_{i,3}}}} A(V_{(i)}) \big] \ .
\ee
For the $j$'s handle of $\Sigma$, the polygon $P$ has four edges $f_{j,1},\dots,f_{j,4}$ and one obtains the corresponding handle in $\Sigma$ by glueing $f_{j,1}$ to $f_{j,3}$ and $f_{j,2}$ to $f_{j,4}$. We have canonical isomorphisms $W_{f_{j,1}} \xrightarrow{\sim} W_{f_{j,3}}^\vee$ and $W_{f_{j,2}} \xrightarrow{\sim} W_{f_{j,4}}^\vee$ and we define the morphisms
\be\label{eq:cut-handle-map}
	F_j = \big[ 
	W_{f_{j,1}}W_{f_{j,2}}W_{f_{j,3}} W_{f_{j,4}} 
	\xrightarrow{\sim}
	W_{f_{j,3}}^\vee W_{f_{i,4}}^\vee W_{f_{j,3}} W_{f_{j,4}} 	\xrightarrow{\jmath_{W_{f_{j,3}}, W_{f_{j,4}}}} H \big] \ .
\ee
Combining \eqref{eq:evaluate-graph-in-P}--\eqref{eq:cut-handle-map} gives a map $\tilde\psi : G \to H(\BSig)$ defined as
\be
	\tilde\psi(\Gamma) = 
	\big(E_1 \otimes \cdots \otimes E_b \otimes F_1 \otimes \cdots \otimes F_g\big)
	\circ
	\langle f^{-1}(\Gamma) \rangle_{\BQ} \ .
\ee

The first key ingredient to compute string-net spaces is the following result \cite{Kirillov:2011mk}.

\begin{theorem}\label{thm:punctured-to-Hom}
The linear extension of $\tilde\psi$ descends to an isomorphism 
\be
	\psi : \SN(\BSig - \{p\}) \longrightarrow H(\BSig) \ .
\ee
\end{theorem}

\begin{proof}
Since the proof follows the line of argument in  \cite{Kirillov:2011mk}, we only sketch it here. 

\begin{figure}[bt]
a)\hspace{-1em}
\begin{tikzpicture}[baseline=6em]
\node at (0,0) {\includegraphics[width=18em]{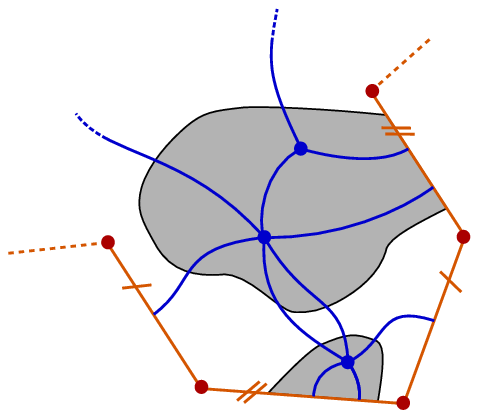}};
\node at (3.3,0.8) {\small $f_{1,1}$};
\node at (3.4,-2.2) {\small $f_{1,2}$};
\node at (1.0,-3.3) {\small $f_{1,3}$};
\node at (-1.4,-2.2) {\small $f_{1,4}$};
\node at (2.1,-2.5) {\small $v$};
\end{tikzpicture}
\hspace{2em}
b)\hspace{-1em}
\begin{tikzpicture}[baseline=6em]
\node at (0,0) {\includegraphics[width=18em]{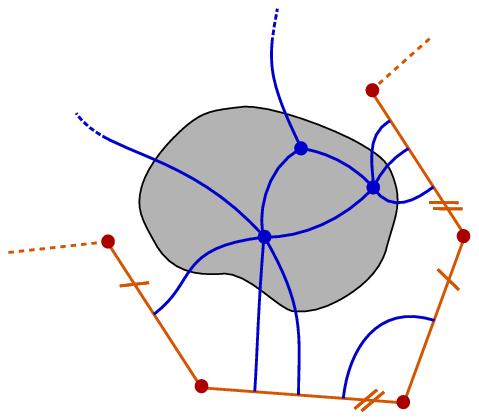}};
\node at (3.3,0.8) {\small $f_{1,1}$};
\node at (3.4,-2.2) {\small $f_{1,2}$};
\node at (1.0,-3.3) {\small $f_{1,3}$};
\node at (-1.4,-2.2) {\small $f_{1,4}$};
\node at (1.9,0.35) {\small $v$};
\end{tikzpicture}
	\caption{A vertex crossing a curve when contracting $f(Q)$. In passing from a) to b), the image of $Q$ (shaded area) is contracted to lie in the interior of the polygon $P$. In the process, the vertex $v$ crosses the edges $f_{1,1} \sim f_{1,3}$, which are identified in the surface $\Sigma$ (cf.\ Figure~\ref{fig:surface-polygon}).}
	\label{fig:dinaturality-and-vertex-crossing}
\end{figure}

Write $\tilde\psi$ also for the linear extension $\mathrm{span}_{\Bbbk} G \to H(\BSig)$. We first show that $\tilde\psi$ vanishes on the generators of $N(\BSig - \{p\})$ that lie in $\mathrm{span}_{\Bbbk} G$, so that $\tilde\psi$ descends to a map $\psi$ on $SN(\BSig - \{p\})$.
We distinguish two kinds of embeddings $f : Q \to \Sigma-\{p\}$: 1) $f(Q)$ is contained in the interior of the polygon $P$, or 2) $f(Q)$ intersects the cutting arcs in Figure~\ref{fig:surface-polygon}\,(a). 
Generators in $N(\BSig - \{p\})$ arising from case 1 are mapped to zero by $\tilde\psi$ by the properties of the map \eqref{eq:<>Q-def}. For case 2) one contracts $f(Q)$ until it is contained in $P$, deforming the graph (``stretching the edges'') along the deformation, see Figure~\ref{fig:dinaturality-and-vertex-crossing}. During the contraction, a vertex of the graph may cross one of the cutting arcs. That the coloured graph before and after the vertex crossing a cutting arc are mapped to the same element of $H(\BSig)$ follows from dinaturality of $\iota$ and $\jmath$. This reduces case 2 to case 1.

\begin{figure}[bt]
$$
\sum_{U_1,\dots,U_b,S_1,T_1,\dots,S_g,T_g}
\hspace{1em}
\begin{tikzpicture}[very thick,baseline=0em,scale=1.2]
\node at (0,0) {\includegraphics[width=18em]{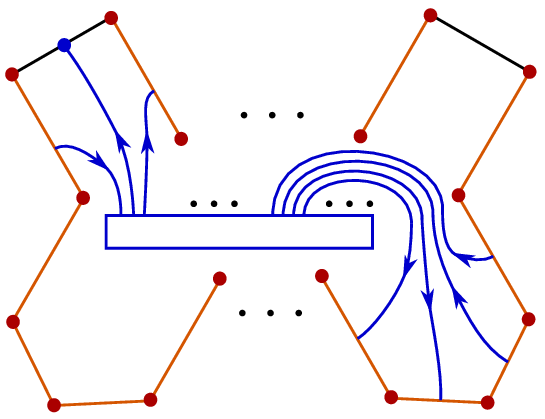}};
\node at (-2.9,0.75) {\small $e_{1,1}$};
\node at (-2.6,2.1) {\small $e_{1,2}$};
\node at (-1.05,1.5) {\small $e_{1,3}$};
\node at (2.9,0.75) {\small $e_{b,3}$};
\node at (2.6,2.1) {\small $e_{b,2}$};
\node at (1.05,1.5) {\small $e_{b,1}$};
\node at (-3.0,-0.5) {\small $f_{g,4}$};
\node at (-3.1,-1.8) {\small $f_{g,3}$};
\node at (-1.9,-2.5) {\small $f_{g,2}$};
\node at (-0.75,-1.7) {\small $f_{g,1}$};
\node at (3.0,-0.5) {\small $f_{1,1}$};
\node at (3.1,-1.8) {\small $f_{1,2}$};
\node at (1.9,-2.5) {\small $f_{1,3}$};
\node at (0.75,-1.7) {\small $f_{1,4}$};

\node at (-0.4,-0.25) {\small $\varphi$};

\node at (-1.9,1.6) {\footnotesize $V_{(1)}$};
\node at (-1.2,0.5) {\footnotesize $U_1$};
\node at (-2.3,0.9) {\footnotesize $U_1$};
\node at (2.45,-0.75) {\footnotesize $S_1$};
\node at (1.7,-1.5) {\footnotesize $S_1$};
\node at (2.3,-1.6) {\footnotesize $T_1$};
\node at (1.15,-1.0) {\footnotesize $T_1$};
\end{tikzpicture}
$$
	\caption{The linear combination of coloured graphs assigned to a morphism $\varphi \in H(\BSig)$
	(only the part of the coloured graph for the first boundary component and for the first handle are shown shown for better readability).	
		This figure also introduces a notation which will be used repeatedly below: instead of a vertex, a total order on the attached edges and a morphism labelling the vertex, we will write the morphism in a coupon and attache the edges to its top side.}
\label{fig:morph-to-graph}
\end{figure}

To show invertibility of $\psi$ we
construct a linear map $\phi : H(\BSig) \to \SN(\BSig - \{p\})$. This uses the explicit expressions for $A(V)$ and $H$ as sums over simple objects given in \eqref{eq:A(X)-via-simples} and \eqref{eq:H-via-simples}. The coloured graph in $\BSig-\{p\}$ assigned to a morphism in $H(\BSig)$ is shown in Figure~\ref{fig:morph-to-graph}.
By construction, $\psi \circ \phi = \id_{H(\BSig)}$. To see that also $\phi \circ \psi = \id_{\SN(\BSig - \{p\})}$, insert the explicit form of $\iota$ and $\jmath$ from \eqref{eq:A(X)-via-simples} and \eqref{eq:H-via-simples} and carry out the sum over the dual basis pair and the simple objects. 
\end{proof}

After being able to express string-net spaces of punctured surfaces in terms of Hom-spaces of $\Cc$, we will now see how the puncture can be removed. To do so, first note that there is a canonical surjection
\be
	\pi : \SN(\BSig-\{p\}) \longrightarrow \SN(\BSig) \quad , \quad [\Gamma] \mapsto [\Gamma] \ ,
\ee
and via this surjection, $N(\BSig-\{p\})$ is mapped to a subspace of $N(\BSig)$.

Next we construct an endomorphism $B_p$ of $\SN(\BSig-\{p\})$. For $[\Gamma] \in \SN(\BSig-\{p\})$, pick a small disc around $p$ which does not intersect $\Gamma$. Write $\Gamma(U)$ for the graph obtained from $\Gamma$ by adding the boundary of the disc as an additional edge, oriented anti-clockwise, and labelled by $U \in \Ic$. Then
\be\label{eq:idem-Bp}
	B_p([\Gamma]) = \sum_{U \in \Ic} \frac{\dim_r(U)}{\Dim(\Cc)} \, [\Gamma(U)] \ .
\ee
Here, some care has to be taken in the non-spherical case. For example, orienting the loop around $p$ clockwise would have required choosing $\dim_l(U)$ in the above definition. However, with this definition, $B_p$ has the same properties as in \cite[Thm.\,3.9]{Kirillov:2011mk}, namely:

\begin{lemma}
\begin{enumerate}
	\item $B_p$ is an idempotent.
	\item Let $\Gamma,\Gamma' \in \Graph(\BSig-\{p\})$ be two graphs which are identical outside of a disc around $p$, and which inside of the disc differ by one edge passing on different sides of $p$. That is, showing only the part of the graph inside the disc,
	\be
	\Gamma \,=~ 
	\begin{tikzpicture}[baseline=-0.3em]
\coordinate (t) at (0,1);
\coordinate (b) at (0,-1);
\coordinate (l) at (-0.5,0);
\coordinate (r) at (0.5,0);
\begin{scope}[very thick,blue!80!black,decoration={markings,mark=at position 0.5 with {\arrow{>}}}]
\draw[postaction={decorate}] (b)  .. controls ++(0,0.5) and ++(0,-0.5) .. (r) .. controls ++(0,0.5) and ++(0,-0.5) .. (t);
\end{scope}
\draw[very thick,orange!80!black,dashed] (0,0) circle (1);
\draw[very thick,black,fill=black] (0,0) circle (0.1);
\node at ([shift={(0,-0.35)}]0,0) {\small $p$};
\end{tikzpicture}
	\quad , \quad
	\Gamma' \,=~
	\begin{tikzpicture}[baseline=-0.3em]
\coordinate (t) at (0,1);
\coordinate (b) at (0,-1);
\coordinate (l) at (-0.5,0);
\coordinate (r) at (0.5,0);
\begin{scope}[very thick,blue!80!black,decoration={markings,mark=at position 0.5 with {\arrow{>}}}]
\draw[postaction={decorate}] (b)  .. controls ++(0,0.5) and ++(0,-0.5) .. (l) .. controls ++(0,0.5) and ++(0,-0.5) .. (t);
\end{scope}
\draw[very thick,orange!80!black,dashed] (0,0) circle (1);
\draw[very thick,black,fill=black] (0,0) circle (0.1);
\node at ([shift={(0,-0.35)}]0,0) {\small $p$};
\end{tikzpicture}
	\quad .
	\ee
	Then $B_p([\Gamma]) = B_p([\Gamma'])$.
\end{enumerate}
\end{lemma}

\begin{proof}
Part 2 amounts to the identities of equivalence classes of graphs given below. In these pictures, only a rectangular region of the surface and graph are shown. Outside of that region, the graph remains unchanged. The indication of the total order around a vertex will be omitted if it is clear from the labelling. The objects $U,V \in \Cc$ are simple, while $W \in \Cc$ is not necessarily simple.
\begin{align}
B_p([\Gamma]) &= 
\sum_{U \in \Ic} \frac{\dim_r(U)}{\Dim(\Cc)}
~
\begin{tikzpicture}[baseline=0em]
\coordinate (v1) at (0.5,0.5);
\coordinate (v2) at (0.5,-0.5);
\coordinate (tr) at (1.5,2);
\coordinate (bl) at (-1.5,-2);
\coordinate (t) at (0,2);
\coordinate (b) at (0,-2);
\coordinate (r) at (1.2,0);
\coordinate (l) at (-1.2,0);
\begin{scope}[very thick,blue!80!black,decoration={markings,mark=at position 0.5 with {\arrow{>}}}]
\draw [postaction={decorate}](v1)  arc  (0:180:0.5);
\end{scope}
\begin{scope}[very thick,blue!80!black,decoration={markings,mark=at position 0.5 with {\arrow{<}}}]
\draw [postaction={decorate}](v2)  arc  (0:-180:0.5);
\end{scope}
\begin{scope}[very thick,blue!80!black,decoration={markings,mark=at position 0.5 with {\arrow{>}}}]
\draw (v2) -- (v1);
\draw ([shift={(-1,0)}]v2) -- ([shift={(-1,0)}]v1);
\draw[postaction={decorate}] (b)  .. controls ++(0,0.5) and ++(0,-1) .. (r) .. controls ++(0,1) and ++(0,-0.5) .. (t);
\end{scope}
\draw[very thick,orange!80!black,dashed] (bl) rectangle (tr);
\draw[very thick,black,fill=black] (0,0) circle (0.1);
\node at ([shift={(0,-0.35)}]0,0) {\small $p$};
\node at ([shift={(-0.15,0.3)}]b) {\small $W$};
\node at ([shift={(-0.75,0.0)}]0,0) {\small $U$};
\end{tikzpicture}
~\overset{(1)}=
\sum_{U,V \in \Ic} \sum_\alpha \frac{\dim_r(U)}{\Dim(\Cc)}
~
\begin{tikzpicture}[baseline=0em]
\coordinate (v1) at (0.5,0.5);
\coordinate (v2) at (0.5,-0.5);
\coordinate (tr) at (1.5,2);
\coordinate (bl) at (-1.5,-2);
\coordinate (t) at (0,2);
\coordinate (b) at (0,-2);
\coordinate (r) at (1.2,0);
\coordinate (l) at (-1.2,0);
\begin{scope}[very thick,blue!80!black,decoration={markings,mark=at position 0.5 with {\arrow{>}}}]
\draw [postaction={decorate}](v1)  arc  (0:180:0.5);
\end{scope}
\begin{scope}[very thick,blue!80!black,decoration={markings,mark=at position 0.5 with {\arrow{<}}}]
\draw [postaction={decorate}](v2)  arc  (0:-180:0.5);
\end{scope}
\begin{scope}[very thick,blue!80!black,decoration={markings,mark=at position 0.5 with {\arrow{>}}}]
\draw[postaction={decorate}] (v2) -- (v1);
\draw ([shift={(-1,0)}]v2) -- ([shift={(-1,0)}]v1);
\draw[postaction={decorate}] (b)  .. controls ++(0,0.5) and ++(1,-1) .. (v2);
\draw[postaction={decorate}] (v1)  .. controls ++(1,1) and ++(0,-0.5) .. (t);
\end{scope}
\draw[very thick,blue!80!black,fill=blue!80!black] (v1) circle (0.1);
\draw[very thick,blue!80!black,fill=blue!80!black] (v2) circle (0.1);
\draw[very thick,orange!80!black,dashed] (bl) rectangle (tr);
\draw[very thick,black,fill=black] (0,0) circle (0.1);
\node at ([shift={(-0.15,0.3)}]b) {\small $W$};
\node at ([shift={(-0.2,-0.3)}]t) {\small $W$};
\node at (-0.75,0.0) {\small $U$};
\node at (0.8,0.0) {\small $V$};
\node at ([shift={(-0.3,0)}]v1) {\small $\bar\alpha$};
\node at ([shift={(-0.3,0)}]v2) {\small $\alpha$};
\end{tikzpicture}
\nonumber \\[1em]
&\overset{(2)}=
\sum_{V \in \Ic}  \frac{\dim_r(V)}{\Dim(\Cc)} \sum_{U,\alpha}
 \frac{\dim_r(U)}{\dim_r(V)} 
~
\begin{tikzpicture}[baseline=0em]
\coordinate (v1) at (0.5,0.5);
\coordinate (v2) at (0.5,-0.5);
\coordinate (tr) at (1.5,2);
\coordinate (bl) at (-1.5,-2);
\coordinate (t) at (0,2);
\coordinate (b) at (0,-2);
\coordinate (r) at (1.2,0);
\coordinate (l) at (-1.2,0);
\coordinate (c1) at (-1.2,1);
\coordinate (c2) at (-1.2,-1);
\begin{scope}[very thick,blue!80!black,decoration={markings,mark=at position 0.5 with {\arrow{>}}}]
\draw [postaction={decorate}](v1)  arc  (0:180:0.5);
\end{scope}
\begin{scope}[very thick,blue!80!black,decoration={markings,mark=at position 0.5 with {\arrow{<}}}]
\draw [postaction={decorate}](v2)  arc  (0:-180:0.5);
\end{scope}
\begin{scope}[very thick,blue!80!black,decoration={markings,mark=at position 0.5 with {\arrow{>}}}]
\draw (v1) -- (v2);
\draw[postaction={decorate}] ([shift={(-1,0)}]v1) -- ([shift={(-1,0)}]v2);
\draw[postaction={decorate}] (b) .. controls ++(0,0.5) and ++(0,-0.5) .. (c2) .. controls ++(0,0.6) and ++(-0.6,0.6) .. ([shift={(-1,0)}]v2);
\draw[postaction={decorate}] ([shift={(-1,0)}]v1) .. controls ++(-0.6,-0.6) and ++(0,-0.6) .. (c1) .. controls ++(0,0.5) and ++(0,-0.5) .. (t);
\end{scope}
\draw[very thick,blue!80!black,fill=blue!80!black] ([shift={(-1,0)}]v1) circle (0.1);
\draw[very thick,blue!80!black,fill=blue!80!black] ([shift={(-1,0)}]v2) circle (0.1);
\draw[very thick,orange!80!black,dashed] (bl) rectangle (tr);
\draw[very thick,black,fill=black] (0,0) circle (0.1);
\node at ([shift={(0.2,0.3)}]b) {\small $W$};
\node at ([shift={(0.25,-0.3)}]t) {\small $W$};
\node at (-0.8,0.0) {\small $U$};
\node at (0.8,0.0) {\small $V$};
\node at ([shift={(-0.7,0)}]v1) {\small $\bar\alpha$};
\node at ([shift={(-0.7,0)}]v2) {\small $\alpha$};
\end{tikzpicture}
~\overset{(3)}=
\sum_{V \in \Ic}  \frac{\dim_r(V)}{\Dim(\Cc)}
~
\begin{tikzpicture}[baseline=0em]
\coordinate (v1) at (0.5,0.5);
\coordinate (v2) at (0.5,-0.5);
\coordinate (tr) at (1.5,2);
\coordinate (bl) at (-1.5,-2);
\coordinate (t) at (0,2);
\coordinate (b) at (0,-2);
\coordinate (r) at (1.2,0);
\coordinate (l) at (-1.2,0);
\begin{scope}[very thick,blue!80!black,decoration={markings,mark=at position 0.5 with {\arrow{>}}}]
\draw [postaction={decorate}](v1)  arc  (0:180:0.5);
\end{scope}
\begin{scope}[very thick,blue!80!black,decoration={markings,mark=at position 0.5 with {\arrow{<}}}]
\draw [postaction={decorate}](v2)  arc  (0:-180:0.5);
\end{scope}
\begin{scope}[very thick,blue!80!black,decoration={markings,mark=at position 0.5 with {\arrow{>}}}]
\draw (v2) -- (v1);
\draw ([shift={(-1,0)}]v2) -- ([shift={(-1,0)}]v1);
\draw[postaction={decorate}] (b)  .. controls ++(0,0.5) and ++(0,-1) .. (l) .. controls ++(0,1) and ++(0,-0.5) .. (t);
\end{scope}
\draw[very thick,orange!80!black,dashed] (bl) rectangle (tr);
\draw[very thick,black,fill=black] (0,0) circle (0.1);
\node at ([shift={(0,-0.35)}]0,0) {\small $p$};
\node at ([shift={(0.25,0.3)}]b) {\small $W$};
\node at ([shift={(0.2,-0.3)}]t) {\small $W$};
\node at (0.75,0.0) {\small $V$};
\end{tikzpicture}
\nonumber\\[1em]
& = B_p([\Gamma']) \ .
\end{align}
In step 1 a sum over a basis $\{\alpha\}$ of $\Cc(UW,V)$ and the corresponding dual basis $\{ \bar\alpha \}$ of $\Cc(V,UW)$ has been inserted. The two bases are dual in the sense that $\alpha \circ \bar\beta = \delta_{\alpha,\beta}\, \id_V$. Step 2 is just a deformation of the coloured graph, where the vertices $\alpha$ and $\bar\alpha$ are dragged to the other side of the loop, and in step 3 we used \eqref{eq:basis-sum-with-duals-1} from the appendix.

Part 1 of the lemma follows from Part 2, as $B_p^2$ amounts to two loops around $p$, one of which can be changed to a small loop not encircling $p$, which then produces a factor of $\dim_l(U)$.
\end{proof}

The second key ingredient to compute string-net spaces is the following theorem.

\begin{theorem}\label{thm:remove-puncture}
The restriction  $\pi : \mathrm{im}(B_p) \to \SN(\BSig)$ is an isomorphism.
\end{theorem}

The proof is the same as that of \cite[Thm.\,3.9\,(3)]{Kirillov:2011mk} and we omit it here.

\medskip

We can thus compute the string-net spaces $\SN(\BSig)$ by first appealing to Theorem~\ref{thm:punctured-to-Hom} to transport the action of the idempotent $B_p$ on $SN(\BSig-\{p\})$ to the idempotent 
\be\label{eq:tilde-proj-def}
\tilde B_p := \psi \circ B_p \circ \psi^{-1} \quad : \quad H(\BSig) \to H(\BSig) \ .
\ee 
One can then work out the image of $\tilde B_p$, e.g.\ by working in a suitable basis of $H(\BSig)$, and use Theorem~\ref{thm:remove-puncture} to finally obtain $\mathrm{im}\tilde B_p \cong \SN(\BSig)$. The example of the sphere and torus given below, as well as the example in Section~\ref{sec:Zr-graded-vsp}, are based on this approach.

\subsubsection*{String-net space of a sphere}

Denote by $S^2$ the marked surface given by a sphere with empty boundary, and denote by $[S^2]$ the class in $\SN(S^2)$ represented by the empty graph on $S^2$. 

\begin{proposition}\label{prop:SN-S2-1d-or-0d}
If $\Cc$ is spherical, then $\SN(S^2)$ is one-dimensional with basis $[S^2]$. If $\Cc$ is not spherical, then $\SN(S^2) = \{0\}$.
\end{proposition}

\begin{proof}
Given an arbitrary coloured graph $\Gamma$ on $S^2$ one can always find an embedding $f: Q \to S^2$ whose image contains the whole graph. 
Thus $\langle - \rangle_{\mathbf{Q}}$ assigns to it an element in $\Cc(\one,\one) = \Bbbk \, \id$, and in $\SN(S^2)$ we then have $[\Gamma] = \langle f^{-1}(\Gamma) \rangle_{\mathbf{Q}} \, [S^2]$. This shows that $\SN(S^2)$ is at most one-dimensional and is spanned by $[S^2]$.

We will give two different ways to complete the proof. 
The first way relates to the original motivation to introduce the notion ``spherical'' \cite{Barrett:1993zf,Barrett:1993ab}. Namely, a small clockwise loop labelled $X \in \Cc$ is equal to $\dim_r(X) \, [S^2]$, but by dragging it around the sphere it can be deformed to a small anti-clockwise loop, which is then equal to $\dim_l(X)\, [S^2]$. If there is an object for which $\dim_r(X) \neq \dim_l(X)$, we must have $[S^2]=0$. If such an object does not exist, $\Cc$ is spherical and $\SN(S^2)$ is one-dimensional, see \cite{Barrett:1993ab} and \cite[Cor.\,3.8]{Kirillov:2011mk}.

The second way uses the projector $B_p$. Namely, $B_p$ is an idempotent on $\SN(S^2 - \{p\})$ whose image is isomorphic to $\SN(S^2)$. Since $\SN(S^2 - \{p\})$ is topologically a disc, it has basis $[S^2 - \{p\}]$ (Theorem~\ref{thm:punctured-to-Hom}). We compute
\begin{align}
B_p([S^2-\{p\}]) 
&= 
\sum_{U \in \Ic} \frac{\dim_r(U)}{\Dim(\Cc)}
~
\begin{tikzpicture}[baseline=0em]
\coordinate (p) at (-0.5,0.5);
\shade [ball color=orange!50!white] (0,0) circle [radius=1.7];
\node at (1.5,1.5) {\small $S^2$};
\begin{scope}[very thick,blue!80!black,decoration={markings,mark=at position 0.25 with {\arrow{>}}}]
\draw [postaction={decorate}] (p) circle (0.7);
\end{scope}
\draw[very thick,black,fill=black] (p) circle (0.1);
\node at ([shift={(0.2,-0.2)}]p) {\small $p$};
\node at ([shift={(0.7,0.7)}]p) {\small $U$};
\end{tikzpicture}
~
= 
\sum_{U \in \Ic} \frac{\dim_r(U)}{\Dim(\Cc)} 
~
\begin{tikzpicture}[baseline=0em]
\coordinate (p) at (-0.5,0.5);
\shade [ball color=orange!50!white] (0,0) circle [radius=1.7];
\node at (1.5,1.5) {\small $S^2$};
\begin{scope}[very thick,blue!80!black,decoration={markings,mark=at position 0.25 with {\arrow{<}}}]
\draw [postaction={decorate}] ([shift={(1.0,-0.7)}]p) circle (0.6);
\end{scope}
\draw[very thick,black,fill=black] (p) circle (0.1);
\node at ([shift={(0.2,-0.2)}]p) {\small $p$};
\node at ([shift={(0.8,0.1)}]p) {\small $U$};
\end{tikzpicture}
\nonumber\\
&= 
\Big(\sum_{U \in \Ic} \frac{\dim_r(U)^2}{\Dim(\Cc)} \Big) [S^2-\{p\}] \ .
\end{align}
By Lemma~\ref{lem:dim-sum-zero}, this sum is either one or zero, depending on whether $\Cc$ is spherical or not.
\end{proof}

\subsubsection*{String-net space of an annulus}

Consider an annulus centred at the origin of the complex plane, say with radii $\frac12$ and $1$ and let $B$ be given by the two points $\tfrac12$ and $1$ on the real axis. For $U,V \in \Cc$, let $\mathbf{A}(U,V)$ be the marked surface with $(U,-)$ assigned to $\frac12 \in B$ and $(V,+)$ to $1 \in B$. 

It was noticed in \cite{Kirillov:2011mk} that the string-net space $SN(\mathbf{A}(U,V))$ describes morphisms in the Drinfeld centre $\Zc(\Cc)$ of $\Cc$. The argument remains valid for non-spherical $\Cc$ and we quickly recall it here, using the language of monads.

For $W \in \Cc$ and $f : U \to W^\vee \otimes V \otimes W$ let $\Gamma(f)$ be the coloured graph on $\mathbf{A}(U,V)$ given by
\be
	\Gamma(f) = 
\begin{tikzpicture}[baseline=0em]
\coordinate (s) at (-0.9,0);
\coordinate (f) at (0.5,0);
\draw[very thick,black] (0,0) circle (2.5);
\draw[very thick,black] (s) circle (0.6);
\draw[very thick,blue!80!black] ([shift={(0,-0.5)}]f) rectangle ++(0.7,1) node[pos=.5,black] {\small $f$};
\draw[very thick,blue!80!black,fill=blue!80!black] ([shift={(0.6,0)}]s) circle (0.1);
\draw[very thick,blue!80!black,fill=blue!80!black] (2.5,0) circle (0.1);
\begin{scope}[very thick,blue!80!black,decoration={markings,mark=at position 0.5 with {\arrow{>}}}]
\draw [postaction={decorate}] ([shift={(0.6,0)}]s) -- (f);
\draw [postaction={decorate}] ([shift={(0.7,0)}]f) -- (2.5,0);
\end{scope}
\begin{scope}[very thick,blue!80!black,decoration={markings,mark=at position 0.5 with {\arrow{<}}}]
\draw [postaction={decorate}](0,2)  arc  (90:270:2);
\draw (0,2) .. controls ++(1.2,0) and ++(1.2,0.4) .. ([shift={(0.7,0.3)}]f);
\draw (0,-2) .. controls ++(1.2,0) and ++(1.2,-0.4) .. ([shift={(0.7,-0.3)}]f);
\end{scope}
\node at (s) {\footnotesize $(U,-)$};
\node at (0.2,0.2) {\small $U$};
\node at (2.0,0.2) {\small $V$};
\node at (1.0,-1.4) {\small $W$};
\node at (3.1,0) {\footnotesize $(V,+)$};
\end{tikzpicture}
\ee
The same argument used in the proof of Theorem~\ref{thm:punctured-to-Hom} can be applied to show that
\be\label{eq:SNA-C(UAV)}
	\psi : SN( \mathbf{A}(U,V) ) \to \Cc(U,A(V))
	\quad , \quad
	\Gamma(f) \mapsto \iota(V)_W \circ f
\ee
is an isomorphism. Using \eqref{eq:forget-hatA-adjunction}, we get
\be
	\Cc( U, AV) = \Cc( U, \forget \hat A V) \xrightarrow{~\sim~} \Zc(\Cc)(\hat A(U),\hat A(V)) \ .
\ee
The inverse map to this isomorphism is provided by applying the forgetful functor and then composing with the unit of the monad $A$. The unit is given by the dinatural transformation $\iota(V)_\one : V = \one^\vee \otimes V \otimes \one \to A(V)$. Combining this with the inverse of the map $\psi$ in \eqref{eq:SNA-C(UAV)} as described in the proof of Theorem~\ref{thm:punctured-to-Hom}, we get the following statement, cf.~\cite[Sec.\,6]{Kirillov:2011mk}.

\begin{proposition}\label{prop:ZC-Hom-SN}
The linear map $\hat\phi : \Zc(\Cc)(\hat A(U),\hat A(V)) \to SN( \mathbf{A}(U,V) )$ given by
\be\label{eq:Hom-ZC-SN-iso}
g \longmapsto \Gamma(f)
\quad , ~~ \text{where} ~~
f = \big[ U \xrightarrow{=} \one^\vee U \one \xrightarrow{\iota_\one(U)} A(U) \xrightarrow{\forget(g)} A(V) \big] \ ,
\ee
is an isomorphism.
\end{proposition}

We will be particularly interested in the case $U=V$. To describe a basis in this case, recall that since $\Cc$ is semisimple and $\Dim(\Cc)\neq 0$, also $\Zc(\Cc)$ is semisimple \cite{Muger2001b}. Let $\Jc$ be a choice of representatives of the isomorphism classes of simple objects in $\Zc(\Cc)$. We can then decompose
\be\label{eq:ZC-Hom-decomp}
	\Zc(\Cc)(\hat A(U),\hat A(U))
	\cong
	\bigoplus_{Z \in \Jc} \Zc(\Cc)(\hat A(U),Z) \otimes_{\Bbbk} \Zc(\Cc)(Z,\hat A(U)) \ .
\ee
Pick a basis $\{\alpha_Z \}$ of $\Zc(\Cc)(Z,\hat A(U))$ and denote by $\{\bar\alpha_Z\}$ the dual basis of $\Zc(\Cc)(\hat A(U),Z)$. Then $\{ \alpha_Z \circ \bar\beta_Z \}_{Z,\alpha_Z,\beta_Z}$ is a basis of $\Zc(\Cc)(\hat A(U),\hat A(U))$, and hence after transport with \eqref{eq:Hom-ZC-SN-iso} we get a basis of $SN( \mathbf{A}(U,U) )$.

\begin{remark}\label{rem:canonical-projector-centre}
Let $Y \in \Zc(\Cc)$. In the case $U=V=\forget(Y)$, the Hom-space $\Zc(\Cc)(\hat A(U),\hat A(V))$ contains an idempotent $p_Y$ whose image is $Y$. It can be explicitly given as (see \cite[Lem.\,8.3]{Kirillov:2011mk}\footnote{
	The explicit half-braiding on $\hat A(X)$ can be found in \eqref{eq:Ahat-halfbraid}. In verifying that $p_Y$ is compatible with the half-braiding, one will need to use that the sum in \eqref{eq:proj-AY-to-Y} involves $\dim_r(U)$ (rather than e.g.\ $\dim_r(V)$), cf.\ \eqref{eq:half-loops-morph-in-Z}.})
\be\label{eq:proj-AY-to-Y}
	p_Y := 
\sum_{U \in \Ic} \frac{\dim_r(U)}{\Dim(\Cc)}
\begin{tikzpicture}[baseline=3.2em]
\coordinate (vv) at (-0.7,0);
\coordinate (y) at (0,0);
\coordinate (v) at (0.7,0);
\coordinate (vvt) at ([shift={(0,2.8)}]vv);
\coordinate (yt) at ([shift={(0,2.8)}]y);
\coordinate (vt) at ([shift={(0,2.8)}]v);
\draw[very thick,black] (y) -- (yt);
\begin{scope}[very thick,black,decoration={markings,mark=at position 0.6 with {\arrow{>}}}]
\draw [postaction={decorate}] (v) -- ([shift={(0,0.2)}]v) .. controls ++(0,0.5) and ++(0,0.6) .. ([shift={(0,1)}]vv) -- (vv);
\end{scope}
\begin{scope}[very thick,black,decoration={markings,mark=at position 0.6 with {\arrow{>}}}]
\draw [postaction={decorate}] (vvt) -- ([shift={(0,-0.2)}]vvt) .. controls ++(0,-0.5) and ++(0,-0.6) .. ([shift={(0,-1)}]vt) -- (vt);
\end{scope}
\draw[very thick,black] (0,1) circle (0.1);
\draw[very thick,black] ([shift={(0,2.8)}] 0,-1) circle (0.1);
\node[below] at (vv) {\small $V^\vee$};
\node[below] at (y) {\small $Y$};
\node[below] at (v) {\small $V$};
\node[above] at (vvt) {\small $U^\vee$};
\node[above] at (yt) {\small $Y$};
\node[above] at (vt) {\small $U$};
\end{tikzpicture}
\quad .
\ee
In the pictorial notation above, the half-braiding $c_{Y,-}$ of $Y$ is represented as an encircled crossing of strings. That is, from bottom to top the two half-braidings in \eqref{eq:proj-AY-to-Y} are $c_{Y,V} : YV \to VY$ and $c_{Y,U^\vee} : YU^\vee \to U^\vee Y$.
The element in $SN( \mathbf{A}(\forget(Y),\forget(Y)) )$ corresponding to $p_Y$ is
\be
	\hat\phi(p_Y) = \sum_{U \in \Ic} \frac{\dim_r(U)}{\Dim(\Cc)} 
~
\begin{tikzpicture}[baseline=0em]
\draw[very thick,black] (0,0) circle (2);
\draw[very thick,black] (0,0) circle (0.8);
\draw[very thick,blue!80!black,fill=blue!80!black] (0.8,0) circle (0.1);
\draw[very thick,blue!80!black,fill=blue!80!black] (2,0) circle (0.1);
\draw[very thick,blue!80!black] (1.4,0) circle (0.15);
\begin{scope}[very thick,blue!80!black,decoration={markings,mark=at position 0.3 with {\arrow{>}}}]
\draw [postaction={decorate}] (0.8,0) -- (2,0);
\end{scope}
\begin{scope}[very thick,blue!80!black,decoration={markings,mark=at position 0.25 with {\arrow{<}}}]
\draw[postaction={decorate}] (0,0) circle (1.4);
\end{scope}
\node at (0.0,0) {\footnotesize $(\forget Y,-)$};
\node at (1.0,-1.4) {\small $U$};
\node at (2.8,0) {\footnotesize $(\forget Y,+)$};
\end{tikzpicture}
	\ ,
\ee
where the encircled crossing represents the half braiding $c_{Y,U^\vee}$.
\end{remark}

\subsubsection*{String-net space of a torus}

Let $T$ be a 2-torus and let $p \in T$ be a choice of marked point. Theorem~\ref{thm:punctured-to-Hom} gives the first isomorphism in
\be\label{eq:SN(T-p)-via-CUAU}
	\SN(T-\{p\}) \xrightarrow{~\psi~} \Cc(1,H) \cong \bigoplus_{U \in \Ic} C(U,A(U)) \ .
\ee
The second isomorphism follows from the explicit form \eqref{eq:H-via-simples} of $H$.
For $X \in \Zc(\Cc)$ define (we write $X$ also for the object of $\Cc$ obtained after applying the forgetful functor)
\be
	h_X ~=~ \sum_{U \in \Ic} \frac{\dim_r(U)}{\Dim(\Cc)} 
~~
\begin{tikzpicture}[baseline=0.8cm]
\coordinate (yv) at (-1.4,0);
\coordinate (vv) at (-0.7,0);
\coordinate (y) at (0,0);
\coordinate (v) at (0.7,0);
\coordinate (vvt) at ([shift={(0,1.2)}]vv);
\coordinate (yt) at ([shift={(0,1.2)}]y);
\coordinate (vt) at ([shift={(0,1.2)}]v);
\coordinate (yvt) at ([shift={(0,1.2)}]yv);
\draw[very thick,black] (y) -- (yt);
\draw[very thick,black] (yv) -- (yvt);
\draw[very thick,black] ([shift={(-0.35,0.6)}]yt) -- ++(0,0.8);
\begin{scope}[very thick,black,decoration={markings,mark=at position 0.5 with {\arrow{<}}}]
\draw [postaction={decorate}](y)  arc  (0:-180:0.7);
\end{scope}
\begin{scope}[very thick,black,decoration={markings,mark=at position 0.6 with {\arrow{>}}}]
\draw [postaction={decorate}] (vvt) -- ([shift={(0,-0.2)}]vvt) .. controls ++(0,-0.5) and ++(0,-0.6) .. ([shift={(0,-1)}]vt) -- (vt);
\end{scope}
\draw[very thick,black] ([shift={(0,1.2)}] 0,-1) circle (0.1);
\draw[very thick,black] ([shift={(-0.3,0)}]yvt) rectangle ++(2.7,0.6) node[pos=.5] {\small $\jmath_{X,U}$};
\node at ([shift={(0.3,-0.4)}]yt) {\small $X$};
\node at ([shift={(0.3,-0.4)}]vt) {\small $U$};
\node at (-0.5,-0.1) {\small $c_{X,U^\vee}$};
\node at ([shift={(-0.35,1.65)}]yt) {\small $H$}; 
\end{tikzpicture}	
~~	 \in ~ \Cc(\one,H) \ .
\ee
The encircled crossing represents the half-braiding of $X$, evaluated at $U^\vee$.
The preimage of $h_X$ in $\SN(T-\{p\})$ is given by
\be\label{eq:psi-1(hX)}
	\psi^{-1}(h_X) = \sum_{U \in \Ic} \frac{\dim_r(U)}{\Dim(\Cc)}  
\begin{tikzpicture}[baseline=1.4cm]
\coordinate (tr) at (2.6,3);
\coordinate (tl) at (0,3);
\coordinate (hx) at (1.3,0);
\coordinate (hy) at (0,1.5);
\begin{scope}[very thick,orange!80!black]
\draw (0,0) rectangle (tr);
\draw (0.7,-0.2) -- ++(0.2,0.4);
\draw (0.8,-0.2) -- ++(0.2,0.4);
\draw ([shift={(0.7,-0.2)}]tl) -- ++(0.2,0.4);
\draw ([shift={(0.8,-0.2)}]tl) -- ++(0.2,0.4);
\draw ([shift={(-0.2,-0.9)}]tl) -- ++(0.4,0.2);
\draw ([shift={(-0.2,-0.9)}]tr) -- ++(0.4,0.2);
\end{scope}
\coordinate (xy) at ([shift={(hy)}]hx);
\draw[very thick,blue!80!black] (xy) circle (0.15);
\begin{scope}[very thick,blue!80!black,decoration={markings,mark=at position 0.3 with {\arrow{>}}}]
\draw[postaction={decorate}] (hx) -- ([shift={(hy)}]xy);
\draw[postaction={decorate}] (hy) -- ([shift={(hx)}]xy);
\end{scope}
\draw[very thick,black,fill=black] (tl) circle (0.1);
\draw[very thick,black,fill=black] (tr) circle (0.1);
\draw[very thick,black,fill=black] ([shift={(hx)}]hx) circle (0.1);
\draw[very thick,black,fill=black] (0,0) circle (0.1);
\node[left] at (tl) {\small $p$}; 
\node[left] at (0,0) {\small $p$}; 
\node[right] at (tr) {\small $p$}; 
\node[right] at ([shift={(hx)}]hx) {\small $p$}; 
\node at ([shift={(0.5,0.25)}]hy) {\small $U$}; 
\node at ([shift={(0.25,0.5)}]hx) {\small $X$}; 
\end{tikzpicture}
	 \ ,
\ee
where the corner point $p$ is not part of the surface.
We have:

\begin{lemma}\label{lem:Bp-action-punct-torus}
The set $\{ h_Z \}_{Z \in \Jc}$ is linearly independent in $\Cc(\one,H)$, and $\{ \psi^{-1}(h_Z) \}_{Z \in \Jc}$ spans the image of the idempotent $B_p$ in $\SN(T-\{p\})$.
\end{lemma}

The proof of this is given in Appendix~\ref{app:Bp-action-punct-torus}. Combining this lemma with Theorem~\ref{thm:remove-puncture} immediately gives the following statement (which in the spherical case follows from results in \cite{Kirillov:2011mk}).

\begin{proposition}\label{prop:SN-basis-torus}
$\{ \pi(\psi^{-1}(h_Z)) \}_{Z \in \Jc}$ is a basis of $\SN(T)$.
\end{proposition}

The coloured graph for $\pi(\psi^{-1}(h_Z))$ is the same as in \eqref{eq:psi-1(hX)}, except that now the point $p$ corresponding to the four corners of the square is not removed.

\medskip

In summary, we see that while the string-net space for the sphere behaves differently in the spherical and non-spherical case, for the annulus and torus, no difference arises. The example treated in the next section illustrates that for higher genus the behaviour is in general again different.

\section{Example: $\Zb_r$-graded vector spaces}\label{sec:Zr-graded-vsp}

In this section, let $\Cc_r$, $r \in \Zb_{>0}$, denote the braided monoidal category of $\Zb_r$-graded $\Cb$-vector spaces with its usual tensor product functor, and with associator and symmetric braiding given by that of the underlying vector spaces. The braiding is of course not needed in the definition of string-net spaces, but it will be useful in the projector computation below.

Let $\Cb_u$ denote the one-dimensional $\Cb$-vector space of grade $u \in \Zb_r$ and let $1_u \in \Cb_u$ be its canonical basis. The tensor unit of $\Cc_r$ is $\one = \Cb_0$. The dual $X^\vee$ of a $\Zb_r$-graded vector space is its vector space dual $X^*$ with grading convention $(X^\vee)_u = (X_{-u})^*$. We denote the basis element of $(\Cb_u)^\vee$ dual to $1_u \in \Cb_u$ by $1_{-u}^*$.
The pivotal structures on $\Cc_r$ are parametrised by an $r$'th root of unity $\zeta \in \Cb$ by setting $\delta_{\Cb_u}(1_u) = \zeta^u \, 1_u^{**}$ (see e.g.\ \cite[Sec.\,1.7.3]{Turaev-Virelizier-book}).
We choose the left duality maps to be the same as for vector spaces. Explicitly, on simple objects:
\be
\evL_{\Cb_u}(1_{-u}^* \otimes 1_u) = 1_0 \in \Cb_0
\quad  , \quad
\coevL_{\Cb_u}(1_0) = 1_u \otimes 1_{-u}^* \ .
\ee
The right duality maps are obtained via \eqref{eq:right-duals} and depend on $\zeta$:
\be
\evR_{\Cb_u}(1_{u} \otimes 1_{-u}^*) = \zeta^u \, 1_0 
\quad  , \quad
\coevR_{\Cb_u}(1_0) = \zeta^{-u} \, 1_{-u}^* \otimes 1_{u} \ .
\ee
This gives
\be
	\dim_l(\Cb_u) = \zeta^{-u} 
	\quad , \quad
	\dim_r(\Cb_u) = \zeta^u 
		\quad , \quad
	\Dim(\Cc_r) = r \ .
\ee
Note that $\Cc_r$ is spherical iff $\zeta^u = \zeta^{-u}$ for all $u \in \Zb_r$.
We will be interested in the case that $\zeta$ is a primitive $r$'th root of unity,
\be\label{eq:zeta-primitive}
	\zeta^u = 1 \qquad \Leftrightarrow \qquad u \equiv 0 \mod r \ ,
\ee
and we will assume this from now on. Under this assumption one finds that 
\be
	\Cc \text{ spherical}
	\quad \Leftrightarrow
	\quad
	r \in \{ 1,2 \} \ .
\ee
For $r=1$ we obtain just ungraded $\Cb$-vector spaces, and for $r=2$ the category $\Cc_r$ is equivalent to super-vector spaces as a pivotal category (but not as a braided category as the braiding of $\Cc_r$ does not involve parity signs).

\medskip

\begin{figure}[btp]
a)	
\vspace{-3em}
$$
B_p(\psi^{-1}(\varphi)) 
~=~ 
\sum_{U,S_1,T_1,\dots,S_g,T_g \in \Ic}  \frac{\dim_r(U)}{\Dim(\Cc)} 
~~
\begin{tikzpicture}[very thick,baseline=1em,scale=1.2]
\node at (0,0) {\includegraphics[width=18em]{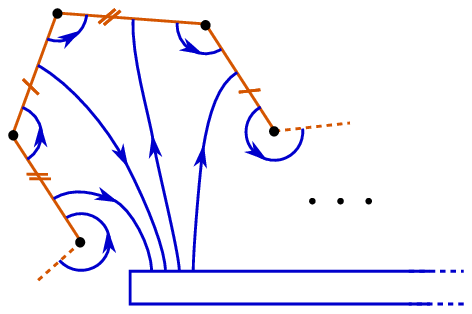}};
\node at (-2.8,-0.6) {\small $f_{1,1}$};
\node at (-3.0,1.3) {\small $f_{1,2}$};
\node at (-1.2,2.1) {\small $f_{1,3}$};
\node at (0.3,1.3) {\small $f_{1,4}$};

\node at (-2,-0.3) {\small $S_1$};
\node at (-1.8,0.8) {\small $T_1$};
\node at (-1.05,1) {\small $S_1$};
\node at (-0.25,-0.6) {\small $T_1$};

\node at (-2,-1.75) {\small $U$};
\node at (-2.45,0.4) {\small $U$};
\node at (-2,1.45) {\small $U$};
\node at (-0.5,1.2) {\small $U$};
\node at (0.5,-0.25) {\small $U$};

\node at (-2.3,-1.1) {\small $p$};
\node at (-3.2,0.3) {\small $p$};
\node at (-2.4,2.2) {\small $p$};
\node at (-0.4,2.0) {\small $p$};
\node at (0.6,0.6) {\small $p$};

\node at (0.2,-1.8) {\small $\varphi$};
\end{tikzpicture}
$$
\vspace{1em}
b)
\vspace{-2em}
$$
\tilde B_p(\varphi) 
~=~ 
\sum_{U,S_1,T_1,\dots,S_g,T_g \in \Ic}  \frac{\dim_r(U)}{\Dim(\Cc)} 
~~
\begin{tikzpicture}[baseline=3em,very thick,black]
\tikzmath{\h1=3;\h2=1;\dx=0.2;\step=1;\c1=1;\c2=0.4;\off=6;}

\foreach \x in {0,\step,{2*\step}} {
	\draw (\x,0)--(\x,\h1);
	\begin{scope}[decoration={markings,mark=at position 0.46 with {\arrow{>}}}]
	\draw[postaction={decorate}] 
	({\x+\dx},\h1) 
	.. controls ++(0,{-\c1}) and ++({-\c2},0) .. 
	({\x+1/2*\step},\h2) 
	.. controls ++(\c2,0) and ++(0,{-\c1}) .. 
	({\x+\step-\dx},\h1);
	\end{scope}
}  	
\draw ({3*\step},0)--({3*\step},\h1);
\begin{scope}[decoration={markings,mark=at position 0.9 with {\arrow{>}}}]
\draw[postaction={decorate}] 
({3*\step+\dx},\h1) 
.. controls ++(0,{-\c1}) and ++({-\c2},0) .. 
({3*\step+1/2*\step},\h2);
\end{scope}
\draw[very thick,black] ({-\dx-0.2},\h1) rectangle ++({3*\step+2*\dx+0.4},0.6) node[pos=.5] {\small $\jmath$};
\draw ({3/2 * \step},{\h1+0.6})--++(0,1);

\foreach \x in {{\off},{\step+\off},{2*\step+\off}} {
	\draw ({\x+\step},0)--({\x+\step},\h1);
	\begin{scope}[decoration={markings,mark=at position 0.46 with {\arrow{>}}}]
	\draw[postaction={decorate}] 
	({\x+\dx},\h1) 
	.. controls ++(0,{-\c1}) and ++({-\c2},0) .. 
	({\x+1/2*\step},\h2) 
	.. controls ++(\c2,0) and ++(0,{-\c1}) .. 
	({\x+\step-\dx},\h1);
	\end{scope}
}  	
\draw ({\off},0)--({\off},\h1);
\begin{scope}[decoration={markings,mark=at position 0.9 with {\arrow{>}}}]
\draw[postaction={decorate}] 
({3*\step+\dx},\h1) 
.. controls ++(0,{-\c1}) and ++({-\c2},0) .. 
({3*\step+1/2*\step},\h2);
\end{scope}
\draw[very thick,black] ({-\dx-0.2+\off},\h1) rectangle ++({3*\step+2*\dx+0.4},0.6) node[pos=.5] {\small $\jmath$};
\begin{scope}[decoration={markings,mark=at position 0.2 with {\arrow{>}}}]
\draw[postaction={decorate}] 
({\off-1/2*\step},\h2) 
.. controls ++(\c2,0) and ++(0,{-\c1}) .. 
({\off-\dx},\h1);
\end{scope}
\draw ({\off+3/2 * \step},{\h1+0.6})--++(0,1);

\begin{scope}[decoration={markings,mark=at position 0.5 with {\arrow{<}}}]
\draw[postaction={decorate}] 
({-\dx},\h1) 
.. controls ++(0,{-\c1}) and ++(0,1) .. 
(-0.6,-0.6)
.. controls ++(0,-1) and ++(-1,0) .. 
({(3*\step+1/2*\step+\off-1/2*\step)/2},-2)
.. controls ++(1,0) and ++(0,-1) .. 
(\off+3*\step+0.6,-0.6)
.. controls ++(0,1) and ++(0,-\c1) .. 
({\off+3*\step+\dx},\h1);
\end{scope}

\draw[very thick,black] (-0.3,-0.6) rectangle ++(3*\step+\off+0.6,0.6) node[pos=.5] {\small $\varphi$};

\node at (0.35,0.4) {\small $S_1^\vee$}; 
\node at (0.35+\step,0.4) {\small $T_1^\vee$}; 
\node at (0.3+2*\step,0.4) {\small $S_1$}; 
\node at (0.3+3*\step,0.4) {\small $T_1$}; 
\node at ({3/2 * \step},{\h1+1.85}) {\small $H$};
\node at (-0.25-\dx,\h1-0.5) {\small $U$}; 
\node at (\step-0.25-\dx,\h1-0.5) {\small $U$}; 
\node at (2*\step-0.25-\dx,\h1-0.5) {\small $U$}; 
\node at (3*\step-0.25-\dx,\h1-0.5) {\small $U$}; 

\node at (\off-0.35,0.4) {\small $S_g^\vee$}; 
\node at (\off-0.35+\step,0.4) {\small $T_g^\vee$}; 
\node at (\off-0.3+2*\step,0.4) {\small $S_g$}; 
\node at (\off-0.3+3*\step,0.4) {\small $T_g$}; 
\node at ({\off+3/2 * \step},{\h1+1.85}) {\small $H$};
\node at (\off-0.25-\dx,\h1-0.5) {\small $U$}; 
\node at (\off+\step-0.25-\dx,\h1-0.5) {\small $U$}; 
\node at (\off+2*\step-0.25-\dx,\h1-0.5) {\small $U$}; 
\node at (\off+3*\step-0.25-\dx,\h1-0.5) {\small $U$}; 

\node at ({(3*\step+1/2*\step+\off-1/2*\step)/2},\h2) {\bf \dots}; 
\node at ({(3*\step+1/2*\step+\off-1/2*\step)/2},{\h1+0.3}) {\bf \dots}; 
\end{tikzpicture}
$$

\caption{a) The action of $B_p$ on $\psi^{-1}(\varphi)$ for $\varphi \in\Cc_r(\one,H^{\otimes g})$. b) The corresponding action of $\tilde B_p$ on $\Cc_r(\one,H^{\otimes g})$. }
\label{fig:Zr-example-Bp}
\end{figure}

We would now like to compute the string-net space for a closed and unpunctured surface $\Sigma_g$ of genus $g$ (since in this case there is no marking we write $\Sigma_g$ also for the corresponding marked surface). We will do this by applying Theorems~\ref{thm:punctured-to-Hom} and~\ref{thm:remove-puncture}. We have $H(\Sigma_g) = \Cc(\one, H^{\otimes g})$ with
\be\label{eq:HSig-Zr}
	H^{\otimes g} = \Big(\bigoplus_{s,t \in \Zb_r} (\Cb_s)^\vee \otimes (\Cb_t)^\vee \otimes \Cb_s \otimes\Cb_t \Big)^{\otimes g} 
	\cong 
	\big((\Cb_0)^{ r^2}\big)^{\otimes g} \cong (\Cb_0)^{ r^{2g}} \ .
\ee

The action of the idempotent $B_p$ on $\SN(\Sigma_g-\{p\})$ from \eqref{eq:idem-Bp} and of $\tilde B_p$ on $\Cc(\one,H^{\otimes g})$ from \eqref{eq:tilde-proj-def} are given in Figure~\ref{fig:Zr-example-Bp}. 

\begin{lemma}
For $\varphi \in\Cc_r(\one,H^{\otimes g})$ we have
\be
	\tilde B_p(\varphi) = \begin{cases} 
		\varphi &; \text{ $2-2g$ divisible by $r$ ,}\\
		0 &;\text{ else.}
		\end{cases}
\ee
\end{lemma}

\begin{proof}
Using the symmetric braiding of $\Cc_r$ we can write the following identity on endomorphisms of $\Cb_u \otimes X \otimes (\Cb_u)^\vee$ in $\Cc_r$,
\be
\begin{tikzpicture}[baseline=3.2em]
\coordinate (vv) at (-0.7,0);
\coordinate (y) at (0,0);
\coordinate (v) at (0.7,0);
\coordinate (vvt) at ([shift={(0,2.5)}]vv);
\coordinate (yt) at ([shift={(0,2.5)}]y);
\coordinate (vt) at ([shift={(0,2.5)}]v);
\draw[very thick,black] (v) -- (vt);
\draw[very thick,black] (y) -- (yt);
\draw[very thick,black] (vv) -- (vvt);
\node[below] at (vv) {\small $\Cb_u$};
\node[below] at (y) {\small $X$};
\node[below] at ([shift={(0.15,0.05)}]v) {\small $(\Cb_u)^\vee$};
\node[above] at (vvt) {\small $\Cb_u$};
\node[above] at (yt) {\small $X$};
\node[above] at ([shift={(0.15,-0.08)}]vt) {\small $(\Cb_u)^\vee$};
\end{tikzpicture}
=~
\zeta^{-u} ~~
\begin{tikzpicture}[baseline=3.2em]
\coordinate (vv) at (-0.7,0);
\coordinate (y) at (0,0);
\coordinate (v) at (0.7,0);
\coordinate (vvt) at ([shift={(0,2.5)}]vv);
\coordinate (yt) at ([shift={(0,2.5)}]y);
\coordinate (vt) at ([shift={(0,2.5)}]v);
\draw[very thick,black] (y) -- (yt);
\begin{scope}[very thick,black,decoration={markings,mark=at position 0.6 with {\arrow{<}}}]
\draw [postaction={decorate}] (v) -- ([shift={(0,0.2)}]v) .. controls ++(0,0.5) and ++(0,0.6) .. ([shift={(0,1)}]vv) -- (vv);
\end{scope}
\begin{scope}[very thick,black,decoration={markings,mark=at position 0.6 with {\arrow{<}}}]
\draw [postaction={decorate}] (vvt) -- ([shift={(0,-0.2)}]vvt) .. controls ++(0,-0.5) and ++(0,-0.6) .. ([shift={(0,-1)}]vt) -- (vt);
\end{scope}
\node[below] at (vv) {\small $\Cb_u$};
\node[below] at (y) {\small $X$};
\node[below] at ([shift={(0.15,0.05)}]v) {\small $(\Cb_u)^\vee$};
\node[above] at (vvt) {\small $\Cb_u$};
\node[above] at (yt) {\small $X$};
\node[above] at ([shift={(0.15,-0.08)}]vt) {\small $(\Cb_u)^\vee$};
\end{tikzpicture}
\quad .
\ee
Inserting this identity $4g$ times in Figure~\ref{fig:Zr-example-Bp}\,(b) results in a factor of $\zeta^{-4gu}$. One then moves the coevaluations crossing $S_i^\vee$, $i=1,\dots,g$ to an evaluation crossing $S_i$ using dinaturality of $\jmath$. Together with the coevaluation already crossing $S_i$ this gives $2g$ times the right dimension of $\Cb_u$, that is, a factor of $\zeta^{2gu}$. The $\Cb_u$ loop enclosing $\varphi$ contributes another $\zeta^u$, as does the prefactor $\dim_r(\Cb_u)$ in the sum in Figure~\ref{fig:Zr-example-Bp}\,(b). The result is
\be
	\tilde B_p(\varphi) = \sum_{u \in \Zb_r} \tfrac{1}r \, \zeta^{(-2g+2)u} \, \varphi \ ,
\ee
which together with \eqref{eq:zeta-primitive} completes the proof.
\end{proof}

Combining this lemma with 
Theorems~\ref{thm:punctured-to-Hom} and~\ref{thm:remove-puncture}
shows:

\begin{proposition}\label{prop:SN-Zr-genusg}
If $2-2g$ is divisible by $r$, then the map 
\be
	\pi \circ \psi^{-1} : \Cc_r(\one,H^{\otimes g}) \longrightarrow \SN(\Sigma_g,\Cc_r)
\ee 
is an isomorphism. If $2-2g$ is not divisible by $r$, then $\SN(\Sigma_g,\Cc_r) = \{0\}$.
\end{proposition}

The condition on $g$ in the above proposition is precisely the condition for an $r$-spin structure to exist on $\Sigma_g$. This indicates a connection between string-nets for $\Cc_r$ and $r$-spin structures which we start to explore in the next section.

\section{Application: $r$-spin structures}\label{sec:r-spin}

As above, let $\Sigma_g$ denote a closed surface of genus $g$. Let $r \in \Zb_{>0}$ and write $\R^r(\Sigma_g)$ for the set of isomorphism classes of $r$-spin structures on $\Sigma_g$ (see e.g.\ \cite[Sec.\,2.1]{Runkel:2018feb} or \cite[Sec.\,3]{Novak:2015phd} for more details on $r$-spin structures). Their number is
\be
\label{eq:r-spin-count}
	|\R^r(\Sigma_g)| = \begin{cases} 
		r^{2g} &; 2-2g \text{ divisible by $r$}\\
		0 &;\text{ else}
	\end{cases}
\ee
Comparing \eqref{eq:r-spin-count} to Proposition~\ref{prop:SN-Zr-genusg} and to \eqref{eq:HSig-Zr} shows that
\be
	\dim\!\big(\, \SN(\Sigma_g,\Cc_r) \big) = |\R^r(\Sigma_g)|
\ee
holds for all $g$. In fact, we can do better than just comparing dimensions. 
Denote by $\VR^r(\Sigma_g)$ the $\Cb$-vector space freely spanned by $\R^r(\Sigma_g)$. The mapping class group of $\Sigma_g$ acts by push-forward on $r$-spin structures, that is, we have an action of $\MCG(\Sigma_g)$ on $\VR^r(\Sigma_g)$. One can compare this to the action on string-net spaces from Remark~\ref{rem:mcg-action}. The result is:

\begin{theorem}\label{thm:MCG-acts-on-subspace}
	$\VR^r(\Sigma_g)$ and $\SN(\Sigma_g,\Cc_r)$ are equivalent as $\MCG(\Sigma_g)$-representations.
\end{theorem}

The proof requires a bit of preparation and we give it at the end of this section. Before that, let us make another observation:
Denote by $\MCG_r(\Sigma_g)$ the subgroup of $\MCG(\Sigma_g)$ that leaves all $r$-spin structures on $\Sigma_g$ invariant, i.e.\ the kernel of the representation $\MCG(\Sigma_g) \to GL(\VR^r)$. As shown in \cite{Sipe:1986}, the quotient $Q_r := \MCG(\Sigma_g)/\MCG_r(\Sigma_g)$ is finite and is an extension of $(2\Zb_r)^{2g}$ by the symplectic group $Sp(g,\Zb_r)$, where $2\Zb_r$ denotes the subgroup of $\Zb_r$ generated by $2$. By construction, the quotient acts faithfully on $\VR^r$, and so we get the following corollary to Theorem~\ref{thm:MCG-acts-on-subspace}.

\begin{corollary}
The $\MCG(\Sigma_g)$-action on $\SN(\Sigma_g,\Cc_r)$ factors through $Q_r$, and $Q_r$ acts faithfully on $\SN(\Sigma_g,\Cc_r)$.
\end{corollary}

The proof of Theorem~\ref{thm:MCG-acts-on-subspace} is based on a combinatorial model for $r$-spin structures developed in \cite{Novak:2015phd,Runkel:2018feb}. We briefly review this model, restricting ourselves to the closed surfaces $\Sigma_g$.

A PLCW-decomposition of $\Sigma_g$ is a cell decomposition subject to certain conditions for which we refer to \cite{Kirillov:2012pl} or \cite[Sec.\,2.2]{Runkel:2018feb}. A {\em marked PLCW-decomposition} of $\Sigma_g$ is a PLCW-decomposition together with
\begin{itemize}
\item a choice of orientation for each edge of the decomposition,
\item a choice of preferred edge for each face, giving a total order on the boundary edges of the face (before they are identified in the cell complex),
\item an assignment of {\em edge indices}, that is, a choice of an element $s_e \in \Zb_r$ for each edge $e$.
\end{itemize}	
We will refer to this choice of data as a {\em marking} of a PLCW decomposition. A marking is {\em admissible} if the following condition, whose ingredients we proceed to explain, holds at each vertex $v$,
\be\label{eq:edge-index-admissible}
	\sum_e \hat s_e ~\equiv~ D_v - N_v + 1 \mod r \ .
\ee
The sum is over all edges $e$ that have $v$ as a boundary vertex. The $\hat s_e \in \Zb_r$ are defined as
\begin{align}
\hat{s}_e=
\begin{cases}
-1&\text{ if $e$ starts and ends at $v$,}\\
s_e&\text{ if $e$ is pointing out of $v$,}\\
-1-s_e&\text{ if $e$ is pointing into $v$.}
\end{cases}
\label{eq:modedgeind}
\end{align}
$D_v$ denotes 
the number of faces whose preferred edge has $v$ as its boundary vertex in clockwise direction (with respect to the orientation of the face). $N_v$ is the number of edges starting at $v$ plus the number of edges ending at $v$ (this results in edges starting and ending at $v$ to be counted twice).

Denote by $\mathcal{M}$ the set of admissible markings of a given PLCW-decomposition of $\Sigma_g$. On $\mathcal{M}$ one introduces an equivalence relation ``$\sim$'' which relates markings for different edge orientations and total orders of boundary edges, as well as edge indices which are related by so-called deck transformations, see \cite[Sec.\,2.3]{Runkel:2018feb} for details. The outcome of the combinatorial construction is a bijection \cite[Thm.\,2.13]{Runkel:2018feb} (building on \cite{Novak:2015phd})
\be\label{eq:marked-PLCW-vs-rspin}
	\mathcal{M} / \hspace{-4pt} \sim  ~~ \xrightarrow{~\sim~} ~ \R^r(\Sigma_g) \ .
\ee

Let $\Cc$ be a pivotal fusion category. 
The next step will be to build a $\Cc$-coloured graph in $\Sigma_g$ from a marked PLCW-decomposition. This graph will depend on the choice a $\Delta$-separable Frobenius algebra $F = (F,\mu,\eta,\Delta,\eps)$ in $\Cc$. Our notation for the structure morphisms is 
\begin{align}
\text{product}~~ & \mu : F \otimes F \to F \ , & \text{unit}~~ & \eta : \one \to F \ ,
\nonumber
\\
\text{coproduct}~~ & \Delta : F \to F \otimes F \ , & \text{counit}~~ & \eps : F \to \one \ ,
\end{align}
see e.g.\ \cite{Fuchs:2009} for more on Frobenius algebras in this setting.
$F$ is called {\em $\Delta$-separable} if $\mu \circ \Delta = \id$. 
A Frobenius algebra in a pivotal category possesses a distinguished Frobenius algebra automorphism, the Nakayama automorphism $N : F \to F$ given by (see e.g.\ \cite{Fuchs:2009})
\be\label{eq:Nak-defn}
N = ~
\begin{tikzpicture}[baseline=2.5em,very thick,black]
\coordinate (mu) at (0,0);
\coordinate (eps) at (0,0.8);
\coordinate (eta) at (0,1.6);
\coordinate (Delta) at (0,2.4);
\coordinate (sF) at (-0.5,-1.2);
\coordinate (tF) at (-0.5,3.6);

\draw (mu)--(eps);
\draw (eta)--(Delta);

\draw[fill=white] (eps) circle (0.1);
\draw[fill=white] (eta) circle (0.1);
\draw[fill=black] (mu) circle (0.1);
\draw[fill=black] (Delta) circle (0.1);

\draw[fill=black] ([shift={(1,-0.07)}] Delta) -- ([shift={(1,0.07)}] mu);

\begin{scope}[decoration={markings,mark=at position 0.5 with {\arrow{>}}}]
\draw [postaction={decorate}](Delta)  arc  (170:0:0.5);
\end{scope}

\begin{scope}[decoration={markings,mark=at position 0.5 with {\arrow{<}}}]
\draw [postaction={decorate}](mu)  arc  (-170:0:0.5);
\end{scope}

\draw (sF) .. controls ++(0,0.5) and ++(-0.5,-0.5) .. (mu);
\draw (tF) .. controls ++(0,-0.5) and ++(-0.5,0.5) .. (Delta);

\node[below] at (sF)  {\small $F$};
\node[above] at (tF)  {\small $F$};
\node at ([shift={(-0.25,0.2)}] mu) {\small $\mu$};
\node at ([shift={(-0.25,-0.2)}] Delta) {\small $\Delta$};
\node at ([shift={(-0.3,0)}] eps) {\small $\eps$};
\node at ([shift={(-0.3,0)}] eta) {\small $\eta$};
\end{tikzpicture}
\qquad , \qquad
N^{-1}
=~
\begin{tikzpicture}[baseline=2.5em,very thick,black]
\coordinate (mu) at (0,0);
\coordinate (eps) at (0,0.8);
\coordinate (eta) at (0,1.6);
\coordinate (Delta) at (0,2.4);
\coordinate (sF) at (0.5,-1.2);
\coordinate (tF) at (0.5,3.6);

\draw (mu)--(eps);
\draw (eta)--(Delta);

\draw[fill=white] (eps) circle (0.1);
\draw[fill=white] (eta) circle (0.1);
\draw[fill=black] (mu) circle (0.1);
\draw[fill=black] (Delta) circle (0.1);

\draw[fill=black] ([shift={(-1,-0.07)}] Delta) -- ([shift={(-1,0.07)}] mu);

\begin{scope}[decoration={markings,mark=at position 0.5 with {\arrow{>}}}]
\draw [postaction={decorate}](Delta)  arc  (10:180:0.5);
\end{scope}

\begin{scope}[decoration={markings,mark=at position 0.5 with {\arrow{<}}}]
\draw [postaction={decorate}](mu)  arc  (-10:-180:0.5);
\end{scope}

\draw (sF) .. controls ++(0,0.5) and ++(0.5,-0.5) .. (mu);
\draw (tF) .. controls ++(0,-0.5) and ++(0.5,0.5) .. (Delta);

\node[below] at (sF)  {\small $F$};
\node[above] at (tF)  {\small $F$};
\node at ([shift={(0.25,0.2)}] mu) {\small $\mu$};
\node at ([shift={(0.25,-0.2)}] Delta) {\small $\Delta$};
\node at ([shift={(0.3,0)}] eps) {\small $\eps$};
\node at ([shift={(0.3,0)}] eta) {\small $\eta$};
\end{tikzpicture}
\ee
These string diagrams also introduce the graphical notation for the structure morphisms of $F$ which we will use.

\begin{example}\label{ex:Zr-group-alg}
Consider the case $\Cc=\Cc_r$. Then the group algebra $F = \Cb \Zb_r$ becomes a Frobenius algebra via the coproduct and counit
\be
	\Delta(1_a) = \tfrac{1}{r} \sum_{b \in \Zb_r} 1_{a+b} \otimes 1_{-b}
	\quad , \quad
	\eps(1_a) = r \, \delta_{a,0} \ .
\ee
The normalisation factor ensures that $\mu \circ \Delta=\id$, so that $F$ is $\Delta$-separable. The Nakayama automorphism can now be computed to be
\be
	N(1_a) = \zeta^{-a} \, 1_a \ . 
\ee
Since $\zeta$ is assumed to be primitive, the Nakayama automorphism has order $r$.
\end{example}

Given a marked PLCW-decomposition $\mathbf{m}$ of $\Sigma_g$ we replace faces and edges by the following components of $\Cc$-coloured graphs:
\be\label{eq:C-graph-for-marked-PLCW}
\begin{tikzpicture}[baseline=0em]
\node at (0,0) {\includegraphics[width=9em]{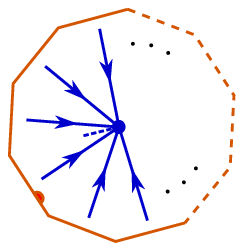}};
\node at (0.5,0.3) {\small $M^{(n)}$};
\node at (-0.55,1.2) {\small $F$};
\node at (0.05,-1.3) {\small $F$};
\end{tikzpicture}
\hspace{5em}
\begin{tikzpicture}[baseline=-.25em]
\coordinate (b) at (0,-1);
\coordinate (t) at (0,1);
\coordinate (l) at (-1,0);
\coordinate (r) at (1.5,0);
\coordinate (c) at (0.5,0);

\draw[very thick,blue!80!black,fill=blue!80!black] (c) circle (0.1);

\begin{scope}[very thick,orange!80!black,decoration={markings,mark=at position 0.7 with {\arrow{>}}}]
\draw [postaction={decorate}](v1)  (b) -- (t);
\end{scope}

\begin{scope}[very thick,blue!80!black,decoration={markings,mark=at position 0.7 with {\arrow{>}}}]
\draw [postaction={decorate}](v1)  (c) -- (r);
\draw [postaction={decorate}](v1)  (c) -- (l);
\end{scope}

\draw[very thick,blue!80!black,dashed] (c) -- ++(90:0.7);

\node at ([shift={(0.05,-0.4)}]c) {\small $E_{s_e}$};
\node at ([shift={(-0.25,-0.4)}]t) {\small $s_e$};
\end{tikzpicture}
~=~
\begin{tikzpicture}[baseline=-.25em]
\coordinate (b) at (0,-1);
\coordinate (t) at (0,1);
\coordinate (l) at (-1,0);
\coordinate (r) at (1.5,0);
\coordinate (c) at (0.5,0);

\draw[very thick,blue!80!black,fill=blue!80!black] (c) circle (0.1);

\begin{scope}[very thick,orange!80!black,decoration={markings,mark=at position 0.7 with {\arrow{>}}}]
\draw [postaction={decorate}](v1)  (b) -- (t);
\end{scope}

\begin{scope}[very thick,blue!80!black,decoration={markings,mark=at position 0.7 with {\arrow{>}}}]
\draw [postaction={decorate}](v1)  (c) -- (r);
\draw [postaction={decorate}](v1)  (c) -- (l);
\end{scope}

\draw[very thick,blue!80!black,dashed] (c) -- ++(-90:0.7);

\node at ([shift={(0.4,0.4)}]c) {\small $E_{-s_e-1}$};
\node at ([shift={(-0.25,-0.4)}]t) {\small $s_e$};
\end{tikzpicture}
\ee
Here, in polygon on the left hand side, the semi-disc attached to one edge indicates the start of the total order, and for the edge on the right hand side, $s_e \in \Zb_r$ denotes the edge index. The morphisms $M^{(n)} \in \Cc(\one,(F^\vee)^{\otimes n})$ and $E_u \in \Cc(\one,F \otimes F)$ are determined in terms of $F$ as follows. Write $\mu^{(n)} : F^{\otimes n} \to F$ for the iterated product of $F$. Then
\be\label{eq:MnEu-explicit}
	M^{(n)} = \big[ \one \xrightarrow{ (\eps \circ \mu^{(n)})^\vee } 
	(F^{\otimes n})^\vee \xrightarrow{\cong} 
	(F^\vee)^{\otimes n} \big]
	\quad , \quad
	E_u = \big[ \one \xrightarrow{ \eta \circ \Delta } F \otimes F 
	\xrightarrow{ N^u \otimes \id } F \otimes F \big]
	\ .
\ee
That the equality on the right hand side in \eqref{eq:C-graph-for-marked-PLCW} holds follows form the explicit form \eqref{eq:Nak-defn} of the Nakayama automorphism.

The above prescription defines an element $\hat\sigma_F(\mathbf{m}) \in \SN(\Sigma_g,\Cc)$ which in fact only depends on the $r$-spin structure defined by the marked PLCW-decomposition $\mathbf{m}$:

\begin{lemma}\label{lem:hat-sig-dep-on-spin}
Suppose two admissible marked PLCW-decompositions $\mathbf{m}$ and $\mathbf{m}'$ define isomorphic $r$-spin structures on $\Sigma_g$ via the assignment \eqref{eq:marked-PLCW-vs-rspin}. Then $\hat\sigma_F(\mathbf{m})=\hat\sigma_F(\mathbf{m}')$.
\end{lemma}

The proof uses that a finite number of moves link $\mathbf{m}$ and $\mathbf{m}'$ as detailed in \cite[Lem.\,2.11\,\&\,Prop.\,2.16]{Runkel:2018feb}. One needs to show invariance of $\hat\sigma_F(-)$ under these moves. This has been done in a different setting in \cite[Sec.\,3]{Runkel:2018feb} and in \cite{Novak:2015ela} for $r=2$. We will omit the proof here and give details elsewhere.

So far we have described the solid arrows in the diagram
\be
\begin{tikzcd}
	& \{ \text{marked PLCW-dec.} \} \arrow[dl] \arrow[dr,"\hat\sigma_F"] \\
	\R^r(\Sigma_g) \arrow[rr,dashed,"\sigma_F"] && \SN(\Sigma_g)
\end{tikzcd}
\quad .
\ee
On all three spaces in the diagram the mapping class group acts by push-forward, and by construction the two solid arrows intertwine the action of $\MCG(\Sigma_g)$. Finally, by Lemma~\ref{lem:hat-sig-dep-on-spin} the map $\hat\sigma_F$ descends to $\R^r(\Sigma_g)$. We collect the results of the above discussion in the following proposition.

\begin{proposition}\label{prop:r-spin-to-string-net}
Given a $\Delta$-separable Frobenius algebra in a pivotal fusion category $\Cc$,
the above construction gives a linear map
\be\label{eq:sigmaF-map-def}
	\sigma_F : \VR^r(\Sigma_g) \longrightarrow \SN(\Sigma_g) \ ,
\ee
which is an intertwiner for the $\MCG(\Sigma_g)$-action.
\end{proposition}

\begin{figure}[bt]

\begin{center}
\begin{tikzpicture}[baseline=0em]
\node at (0,0) {\includegraphics[width=18em]{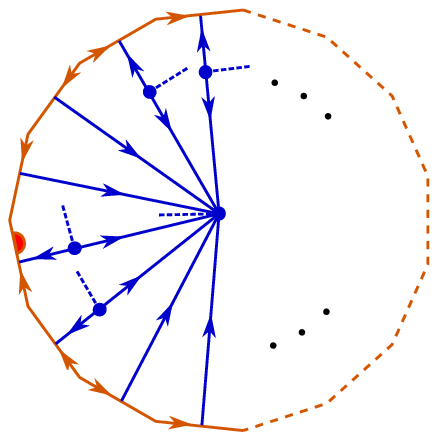}};
\node at (-3.8,0.9) {\small $f_{1,1}$};
\node at (-3.2,2.3) {\small $f_{1,2}$};
\node at (-1.8,3.45) {\small $f_{1,3}$};
\node at (-0.2,3.85) {\small $f_{1,4}$};
\node at (-0.4,-3.85) {\small $f_{g,1}$};
\node at (-1.8,-3.4) {\small $f_{g,2}$};
\node at (-3.4,-2) {\small $f_{g,3}$};
\node at (-3.9,-0.7) {\small $f_{g,4}$};
\node at (-1.5,2) {\small $E_{s_{1,1}}$};
\node at (0.3,2.3) {\small $E_{s_{1,2}}$};
\node at (-2.0,-0.8) {\small $E_{s_{g,2}}$};
\node at (-1.7,-1.95) {\small $E_{s_{g,1}}$};
\node at (0.7,0.3) {\small $M^{(4g)}$};
\end{tikzpicture}
\end{center}
	
	\caption{Marked PLCW-decomposition of $\Sigma_g$ into a single face, $2g$ edges and on vertex. The edges $f_{i,a}$ and $f_{i,a+2}$, $a=1,2$, are identified. The edge sign for the edge $f_{i,1}=f_{i,3}$ is $s_{i,1}$ and that for $f_{i,2}=f_{i,4}$ is $s_{i,2}$. Also shown is the coloured graph corresponding to this decomposition via the rules in \eqref{eq:C-graph-for-marked-PLCW}.}
	\label{fig:coloured-graph-for-r-spin}
\end{figure}

\begin{proof}[Proof of Theorem~\ref{thm:MCG-acts-on-subspace}]
We take $\Cc=\Cc_r$ and $F$ as in Example~\ref{ex:Zr-group-alg}.
In light of Proposition~\ref{prop:r-spin-to-string-net}, it remains to show that for these choices of $\Cc$ and $F$, the map $\sigma_F$ is an isomorphism.

For $\Sigma_g$ we take the PLCW decomposition with one face and with marking as shown in Figure~\ref{fig:coloured-graph-for-r-spin}. To determine the admissible edge indices, we evaluate condition \eqref{eq:edge-index-admissible}. Since each edge starts and ends at $v$, we have $\hat s_e=-1$ for each edge. Furthermore, $D_v = 1$ and $N_v = 4g$, so that $\eqref{eq:edge-index-admissible}$ becomes $-2g \equiv 1 - 4g + 1 \mod r$. Thus if $2-2g \equiv 0 \mod r$, the $2g$ edge indices are unconstrained, and if $2-2g \not\equiv 0 \mod r$, there are no admissible edge indices.

Together with Proposition~\ref{prop:SN-Zr-genusg} and with \eqref{eq:marked-PLCW-vs-rspin}, 
this shows that for $2-2g \not\equiv 0 \mod r$ both sides of \eqref{eq:sigmaF-map-def} are zero-dimensional, proving the claim in this case. We can thus assume that $2-2g \equiv 0 \mod r$.

From Proposition~\ref{prop:SN-Zr-genusg} we know that $\Cc_r(\one,H^{\otimes g}) \cong \SN(\Sigma_g)$. Via this isomorphism, the coloured graph in Figure~\ref{fig:coloured-graph-for-r-spin} gets mapped to the element
\begin{align}
	\beta(s_{1,1},\dots,s_{g,2}) := \big[ 
	&
	\one \xrightarrow{~\eta~} F 
	\xrightarrow{ \chi(s_{1,1},s_{1,2}) } HF
	\xrightarrow{ \id \otimes \chi(s_{2,1},s_{2,2}) } HHF
	\longrightarrow
\nonumber\\	
&
	\cdots
	\xrightarrow{ \id \otimes \chi(s_{g,1},s_{g,2}) } H^{\otimes g}F
	\xrightarrow{ \id \otimes \eps} H^{\otimes g} \big]
\end{align}
where
\be
	\chi(a,b) 
	~=~~ 
\begin{tikzpicture}[baseline=-4em]
\coordinate (yv) at (-1.8,0);
\coordinate (vv) at (-1.1,0);
\coordinate (y) at (0,0);
\coordinate (v) at (0.9,0);
\coordinate (Ft) at (1.9,1.4);
\coordinate (Fb) at (1.9,-5);

\draw[very thick,black] ([shift={(-0.3,0)}]yv) rectangle ++(3.3,0.6) node[pos=.5] {\small $\jmath_{F,F}$};
\draw[very thick,black] ([shift={(-0.35,0.6)}]y) -- ++(0,0.8);

\draw[very thick,black] ([shift={(-0.85,-1)}]y) rectangle ++(1.15,0.6) node[pos=.5] {\footnotesize $N^{-a-1}$};
\draw[very thick,black] (y) -- ++(0,-0.4);
\draw[very thick,black] ([shift={(-0.4,-1)}]v) rectangle ++(1.15,0.6) node[pos=.5] {\footnotesize $N^{-b-1}$};
\draw[very thick,black] (v) -- ++(0,-0.4);

\draw[very thick,black] (Ft) -- (Fb);

\draw[fill=black] ([shift={(0,1)}]Fb) circle (0.1);
\draw[fill=black] ([shift={(0,1.6)}]Fb) circle (0.1);
\draw[fill=black] ([shift={(0,2.2)}]Fb) circle (0.1);
\draw[fill=black] ([shift={(0,2.8)}]Fb) circle (0.1);

\begin{scope}[very thick,black,decoration={markings,mark=at position 0.8 with {\arrow{>}}}]
\draw [postaction={decorate}] (yv) -- ([shift={(0,-3)}]yv) .. controls ++(0,-1) and ++(-0.6,-1) .. ([shift={(0,1)}]Fb);
\draw [postaction={decorate}] (vv) -- ([shift={(0,-2.4)}]vv) .. controls ++(0,-1) and ++(-0.6,-1) .. ([shift={(0,1.6)}]Fb);
\draw ([shift={(0,-1)}]y)  .. controls ++(0,-0.8) and ++(-0.6,0.4) .. ([shift={(0,2.2)}]Fb);
\draw ([shift={(0,-1)}]v)  .. controls ++(0,-0.5) and ++(-0.6,0.4) .. ([shift={(0,2.8)}]Fb);
\end{scope}

\node at ([shift={(0.3,1)}]Fb) {\small $\mu$}; 
\node at ([shift={(0.3,1.6)}]Fb) {\small $\mu$}; 
\node at ([shift={(0.3,2.2)}]Fb) {\small $\Delta$}; 
\node at ([shift={(0.3,2.8)}]Fb) {\small $\Delta$}; 
\node at ([shift={(-0.35,1.65)}]y) {\small $H$}; 
\node at ([shift={(0,0.25)}]Ft) {\small $F$}; 
\node at ([shift={(0,-0.25)}]Fb) {\small $F$}; 
\end{tikzpicture}
\qquad .
\ee
To obtain this expression, one substitutes \eqref{eq:MnEu-explicit} in Figure~\ref{fig:coloured-graph-for-r-spin} and uses the identity on the right hand side of \eqref{eq:C-graph-for-marked-PLCW}.
Substituting the explicit expressions in $\Cc_r$, one finds, for $x \in F$,
\be
	(\chi(a,b))(x) = v_{a,b} \otimes x
	\quad , \quad v_{a,b} = \sum_{s,t \in \Zb_r} \zeta^{sa+tb} 
	\,
	1_{-s}^* \otimes 1_{-t}^* \otimes 1_s \otimes 1_t \in H \ .
\ee
Since $\{v_{a,b}\}_{a,b \in \Zb_r}$ is a basis of $H$, it is easy to see that the $\beta(s_{1,1},\dots,s_{g,2})$ give a basis of $\Cc_r(\one,H^{\otimes g})$ as each $s_{i,j}$ runs over all values in $\Zb_r$.
\end{proof}

\section{Application: background charge}\label{sec:bg-charge}

For a spherical fusion category $\Cc$, it is described in \cite[Sec.\,7]{Kirillov:2011mk} how to extend string-net spaces to include surfaces with marked points. Each marked point is equipped with a tangent vector and is labelled by a pair $(Z,\pm)$, where $Z \in \Zc(\Cc)$. The marked points form one-valent vertices on which edges labelled by $\forget(Z)$ can start or end (depending on the sign).\footnote{
	Actually, in \cite{Kirillov:2011mk} the edges are unoriented, so marked points are just labelled by $Z$, the edge orientation is absorbed into replacing $Z$ by $Z^\vee$. The inclusion of signs is relevant in our setting where edges are oriented.}

In the non-spherical case we can introduce such marked points in the same way. Namely, we replace a marked point labelled by $(Z,+)$ with a circular boundary containing a marked point $(\forget(Z),+)$ and a sum similar to the projector $B_p$ (with $V = U^\vee$, hence the different orientation)
\be
\sum_{V \in \Ic}  \frac{\dim_l(V)}{\Dim(\Cc)} 
~~
\begin{tikzpicture}[baseline=0em]
\coordinate (hb) at (1.3,0);
\coordinate (end) at (2.2,0);

\draw [line width=4pt,gray!60!white] (0,0) circle (0.4);
\draw [very thick,black] (0,0) circle (0.5);
\draw[very thick,blue!80!black,fill=blue!80!black] (0.5,0) circle (0.1);

\begin{scope}[very thick,blue!80!black,decoration={markings,mark=at position 0.6 with {\arrow{>}}}]
\draw [postaction={decorate}]  (hb) -- (0.5,0);
\draw [postaction={decorate}]  (end) -- (hb);
\draw [dashed]  (end) -- ++(0.7,0);
\end{scope}

\draw[very thick,blue!80!black] (hb) circle (0.15);

\begin{scope}[very thick,blue!80!black,decoration={markings,mark=at position 0.3 with {\arrow{<}}}]
\draw [postaction={decorate}](0,1)  arc  (90:270:1);
\draw (0,1) .. controls ++(0.7,0) and ++(0.5,1) .. (hb);
\draw (0,-1) .. controls ++(0.7,0) and ++(-0.5,-1) .. (hb);
\end{scope}

\node at (0,0.75) {\small $V$}; 
\node at ([shift={(0.1,-0.3)}]end) {\small $\forget(Z)$}; 
\end{tikzpicture}
\qquad ,
\ee
where the encircled crossing represents the half-braiding of $Z$ evaluated at $V$.
That this is a valid description of marked points ultimately follows from the analysis of the string-net space of the annulus, cf.\ Remark~\ref{rem:canonical-projector-centre}, but we will not develop this in more detail here. Instead, we take it as motivation for the following construction. 

\medskip

For $X \in \Cc$, denote by $\mathbf{D}(X)$ the marked surface given by a disc with one point on the boundary marked $(X,+)$. We define the {\em string-net space on a sphere with one point marked by $(Z,+)$}, to be the subspace of $\SN(\mathbf{D}(\forget(Z)),\Cc)$ 
given by
\be
\SN(S^2(Z),\Cc) := \mathrm{span}_{\Cb}\big\{ \,
\Psi_\Gamma \,\big| \, \text{$\Gamma$ col.\ graph on $\mathbf{D}(\forget(Z))$} \, \big\} \ .
\ee
Here, $\Psi_\Gamma \in \SN(\mathbf{D}(\forget(Z)),\Cc)$ is given by
\be
~~
\Psi_\Gamma \,=\, 
 \sum_{V \in \Ic} \frac{\dim_l(V)}{\Dim(\Cc)} 
~
\begin{tikzpicture}[baseline=0em]
\coordinate (h) at (0,1.5);
\coordinate (hb) at (0,0.8);
\coordinate (e) at (0,-0.8);

\shade [ball color=orange!50!white] (0,0.1) circle [radius=2.2];
\draw [very thick,black,fill=gray!60!white] (h) ellipse (0.4 and 0.3);
\draw[very thick,blue!80!black,fill=blue!80!black] ([shift={(0,-0.3)}]h) circle (0.1);
\begin{scope}[decoration={markings,mark=at position 0.0 with {\arrow{<}}}]
\draw [very thick,blue!80!black,postaction={decorate}] ([shift={(0,-0.1)}]h) ellipse (0.8 and 0.6);
\end{scope}
\begin{scope}[decoration={markings,mark=at position 0.4 with {\arrow{>}}}]
\draw[very thick,blue!80!black,postaction={decorate}] ([shift={(0,0.8)}]e) -- ([shift={(0,-0.3)}]h);
\end{scope}
\draw[very thick,blue!80!black] (hb) circle (0.15);
\draw [very thick,dashed,blue!80!black,fill=blue!50!white] (e) ellipse (1.1 and 0.8);

\node at ([shift={(1.0,-0.4)}]h) {\small $V$};
\node at ([shift={(0.7,-1.2)}]h) {\small $\forget(Z)$};
\node at (e) {\small $\Gamma$};
\end{tikzpicture}
~~,
\ee
where we presented $\mathbf{D}(\forget(Z))$ as a sphere with a small disc removed and used that any coloured graph can be deformed to lie in the shaded area marked ``$\Gamma$''.
The encircled crossing is the half-braiding $c_{Z,V} : Z V \to VZ$.

From Section~\ref{sec:stringnet} we know that $\SN(\mathbf{D}(X),\Cc)\cong\Cc(\one,X)$, which is independent of the pivotal structure on $\Cc$. Below we will study a class of examples which illustrate that this is not the case for $\SN(S^2(Z),\Cc)$.

\medskip

The specific situation we consider is that the pivotal fusion category is given by a modular fusion category whose pivotal structure has been deformed. 
Let thus $\Mc$ be a modular fusion category, i.e.\ a fusion category, which is equipped with the structure of a ribbon category, and which has a non-degenerate braiding.${}^{\ref{fn:mod-fus-cat}}$
Since $\Mc$ is a ribbon category, it is in particular spherical. 
Denote by $\delta^{\Mc} : (-) \to (-)^{\vee\vee}$ the pivotal structure of $\Mc$ and by $\dim^{\Mc}(-)$ the corresponding dimension of objects in $\Mc$. Since $\Mc$ is spherical, left and right dimensions coincide, and we drop the index. 

Fix an invertible object $J \in \Mc$. Its dimension satisfies $\dim_{\Mc}(J) \in \{\pm 1\}$. For each $X \in \Mc$ we define the following endomorphism of $X$:
\be
	(\eta^J)_X = \frac{1}{\dim^{\Mc}(J)} 
	~~
\begin{tikzpicture}[baseline=-.2em,very thick,black]
\begin{scope}[decoration={markings,mark=at position 0.7 with {\arrow{>}}}]
\draw [postaction={decorate}](0.5,0)  arc  (0:180:0.5);
\end{scope}
\fill[white] (-0.1,0.4) rectangle ++(0.2,0.2);
\draw (0,1.2) -- (0,-1.2);
\fill[white] (-0.1,-0.6) rectangle ++(0.2,0.2);
\begin{scope}[decoration={markings,mark=at position 0.3 with {\arrow{<}}}]
\draw [postaction={decorate}](0.5,0)  arc  (0:-180:0.5);
\end{scope}
\node at (0.7,0) {\small $J$};
\node[above] at (0,1.2) {\small $X$};
\node[below] at (0,-1.2) {\small $X$};
\end{tikzpicture}
\ee
One verifies that $\eta^J$ is a natural monoidal isomorphism of the identity functor, see \cite[Lem.\,3.31]{Drinfeld:0906}. The prefactor $1/\dim^{\Mc}(J)$ is required for monoidality. For $X$ simple one can write
\be
	(\eta^J)_X = \frac{s^{\Mc}_{J,X}}{\dim^{\Mc}(J)\,\dim^{\Mc}(X)} \id_X \ ,
\ee
where $s^{\Mc}$ is the invariant of the Hopf link in $\Mc$,
$s^{\Mc}_{U,V} = \mathrm{tr}^{\Mc}(c_{V,U} \circ c_{U,V})$.

The pivotal fusion category $\Cc$ is defined to be equal to $\Mc$ as a fusion category. The pivotal structure on $\Cc$ is 
\be
	\delta^\Cc_X := \big[ X \xrightarrow{(\eta^J)_X } X \xrightarrow{\delta^{\Mc}_X} X^{\vee\vee} \big] \ .
\ee
Since $\Mc$ is modular, $\eta_J = \eta_K$ implies that $J \cong K$, and so we get a family of pivotal structures on $\Cc$ parametrised by isomorphism classes of invertible objects in $\Mc$. In fact, these are all pivotal structures possible on $\Cc$, see again \cite[Lem.\,3.31]{Drinfeld:0906}.
The quantum dimensions of $\Cc$ are given by, for $X \in \Cc$,
\be
	\dim_l^{\Cc}(X) = \frac{s^{\Mc}_{J^\vee,X}}{\dim^{\Mc}(J)}
	\quad , \quad
	\dim_r^{\Cc}(X) = \frac{s^{\Mc}_{J,X}}{\dim^{\Mc}(J)} \ .
\ee

Since $\Mc$ is modular, 
as a monoidal category the Drinfeld centre of $\Cc$ is $\Zc(\Cc) \cong \Cc \boxtimes \Cc^\mathrm{rev}$, and so its simple objects are pairs $(U,V)$ with $U,V \in \Cc$ simple. We have:

\begin{proposition}\label{prop:S2-ZC-markedpoint}
Let $Z \in \Zc(\Cc)$ be simple. Then $\SN(S^2(Z),\Cc)$ is one-dimensional if $Z \cong (J\otimes J, J^\vee \otimes J^\vee)$ and zero-dimensional else.
\end{proposition}

\begin{proof}
Every $v \in \SN(S^2(Z),\Cc)$ can be written as the following coloured graph inside a disc with a boundary point labelled $(\forget(Z),+)$,
\be
	v ~=~  \begin{tikzpicture}[baseline=0em]
	\draw [very thick,black] (0,0) circle (1.0);
	\draw[very thick,blue!80!black,fill=blue!80!black] (0,0) circle (0.1);
	\draw[very thick,blue!80!black,fill=blue!80!black] (0,1) circle (0.1);
	\begin{scope}[very thick,blue!80!black,decoration={markings,mark=at position 0.5 with {\arrow{>}}}]
	\draw [postaction={decorate}] (0,0)--(0,1);
	\end{scope}
	\node at ([shift={(0,0.35)}]0,1) {\small $(\forget(Z),+)$};
	\node at ([shift={(0.35,0)}]0,0) {\small $\tilde\varphi$};
	\end{tikzpicture}
	\quad .
\ee
Here, $\tilde\varphi \in \Cc(\one,\forget(Z))$ 
is given in terms of an arbitrary morphism $\varphi \in \Cc(\one,\forget(Z))$ as 
\be
\tilde\varphi 
~=~ 
\sum_{R \in \Ic} \frac{\dim^\Cc_l(R)}{\Dim(\Cc)} 
~~
\begin{tikzpicture}[baseline=1em]
\coordinate (l) at (-1,0);
\coordinate (r) at (1,0);
\coordinate (t) at (0,2);
\coordinate (catbox) at (1.5,2);

\draw[very thick,black] ([shift={(-0.5,-0.3)}]0,0) rectangle ++(1.0,0.6) node[pos=.5] {\footnotesize $\varphi$};

\draw[very thick,black] (0,0.3) -- (t);
\begin{scope}[very thick,black,decoration={markings,mark=at position 0.4 with {\arrow{<}}}]
\draw [postaction={decorate}] (l) -- ([shift={(0,0.8)}]l) .. controls ++(0,1) and ++(0,0.6) .. (r);
\end{scope}
\draw[very thick,black] (0,1.02) circle (0.15);
\begin{scope}[very thick,decoration={markings,mark=at position 0.5 with {\arrow{>}}}]
\draw [postaction={decorate}](l)  arc  (-180:0:1);
\end{scope}
\node[above] at (t) {\small $\forget(Z)$};
\node at ([shift={(0.25,0)}]r) {\small $R$};
\draw[very thick,black] (catbox) -- ++(0,0.5);
\draw[very thick,black] (catbox) -- ++(0.5,0);
\node at ([shift={(0.25,0.3)}]catbox) {$\Cc$};
\end{tikzpicture}
\quad .
\ee
These diagrams are to be understood in $\Cc$, and we stress this by the ``$\Cc$'' in the top right corner. This is in contrast to the next diagram, which is taken in the modular category $\Mc$, and where we further evaluate $\tilde\varphi$ by taking $Z \cong (U,V)$, so that $\forget(Z) \cong U \otimes V \in \Mc$:
\begin{align}
	\tilde\varphi &= 
	\frac{1}{\Dim(\Mc)} \sum_{R \in \Ic}
\frac{s^{\Mc}_{J^\vee,R}}{\dim^{\Mc}(J)}
~~
\begin{tikzpicture}[baseline=1em]
\coordinate (l) at (-1,0);
\coordinate (r) at (1,0);
\coordinate (t) at (0,2);
\coordinate (catbox) at (1.5,2);
\draw[very thick,black] ([shift={(-0.5,-0.3)}]0,0) rectangle ++(1.0,0.6) node[pos=.5] {\footnotesize $\varphi$};
\draw[very thick,black] (0.3,0.3) -- ([shift={(0.3,0)}]t);
\fill[white] (0.2,0.68) rectangle ++(0.2,0.3);
\begin{scope}[very thick,black,decoration={markings,mark=at position 0.4 with {\arrow{<}}}]
\draw [postaction={decorate}] ([shift={(0,-0.1)}]l) -- ([shift={(0,0.8)}]l) .. controls ++(0,1) and ++(0,0.6) .. (r);
\end{scope}
\draw[very thick,black] ([shift={(-0.3,-0.6)}]r) rectangle ++(0.7,0.6) node[pos=.5] {\footnotesize $\delta_R^{-1}$};
\begin{scope}[very thick,decoration={markings,mark=at position 0.5 with {\arrow{<}}}]
\draw [postaction={decorate}]([shift={(0,-0.6)}]l)  arc  (-180:0:1);
\end{scope}
\fill[white] (-0.4,1.0) rectangle ++(0.2,0.3);
\draw[very thick,black] (-0.3,0.3) -- ([shift={(-0.3,0)}]t);
\node[above] at ([shift={(0.3,0)}]t) {\small $V$};
\node[above] at ([shift={(-0.3,0)}]t) {\small $U$};
\node at ([shift={(0.15,0.4)}]r) {\small $R$};
\node at ([shift={(0.35,-1)}]r) {\small $R^{\vee\vee}$};
\draw[very thick,black] (catbox) -- ++(0,0.5);
\draw[very thick,black] (catbox) -- ++(0.6,0);
\node at ([shift={(0.35,0.3)}]catbox) {$\Mc$};
\node at ([shift={(0.05,-0.35)}]l) {\small $R^\vee$};
\end{tikzpicture}
\nonumber\\
&= 
	\frac{1}{\Dim(\Mc)} \sum_{R \in \Ic}
	\frac{s^{\Mc}_{J^\vee,R}}{\dim^{\Mc}(J)} \,
	\frac{s^{\Mc}_{J^\vee,R}}{\dim^{\Mc}(J)\dim^{\Mc}(R)} \,
	\frac{s^{\Mc}_{R,U}}{\dim^{\Mc}(U)}
	\cdot \varphi
\nonumber\\
&=
	\frac{1}{\Dim(\Mc)} \sum_{R \in \Ic}
\frac{s^{\Mc}_{(JJ)^\vee,R} \, s^{\Mc}_{R,U}}{\dim^{\Mc}(U)} \cdot \varphi
= \delta_{JJ,U} \cdot \varphi
	\ .
\end{align}
Note that $\varphi$ can be non-zero only if $V \cong U^\vee$.
\end{proof}

\begin{remark}\label{rem:bg-charge}
The term ``background charge'' in the title of this section refers to the idea that there is some ambient ``charge'' on the sphere which needs to be ``compensated'' by the objects of $\Zc(\Cc)$ labelling the marked points in order to get a non-zero state space. Hence a sphere with one marked point needs to be labelled by an object other than the tensor unit in $\Zc(\Cc)$.

The original setting in which this was treated is a two-dimensional conformal field theory called Feigin-Fuks free boson, or free boson with background charge, see \cite[9.1.3]{DiFr}. In the language of vertex operator algebras, this corresponds to the fact that the Heisenberg VOA allows for a one-complex-parameter family of Virasoro elements, see e.g.\ \cite[Sec.\,2.5.9]{FB}. The parameter determines for which value of the $U(1)$-charge $q$ the space of one-point conformal blocks on the sphere with insertion of the Heisenberg module of charge $q$ is non-zero.

To obtain the situation treated in this section one can consider a lattice-extension of the Heisenberg VOA. Its representation theory (with standard Virasoro element) is a modular fusion category, which we can choose for $\Mc$. There is now a discrete choice of possible deformations of the Virasoro element \cite[Thm.\,3.1]{Dong:2006}. Taking the Drinfeld centre $\Zc(\Cc) \cong \Cc \boxtimes \Cc^{\mathrm{rev}}$ in the above construction amounts to considering representations of a holomorphic and an anti-holomorphic copy of the VOA.

As mentioned in the introduction, another interesting situation to consider is when $\Mc$ is the modular fusion category of representations of (the even part of) an $N=2$ superconformal rational VOA and $J$ describes the topological twist of the Virasoro element.
\end{remark}

\newpage

\appendix

\section{Appendix: Proof of Lemma~\ref{lem:Bp-action-punct-torus}}
\label{app:Bp-action-punct-torus}

As a preparation we will need the explicit form of the half braiding on $\hat A(M)$, $M \in \Cc$ (see \cite[Thm.\,2.3]{Balsam:2010}\footnote{
	The different dimension prefactors in \cite[Thm.\,2.3]{Balsam:2010} arise from the different normalisation of dual basis pairs, cf.\ \cite[Lem.\,1.1]{Balsam:2010}.}):
\vspace{-1em}
\be\label{eq:Ahat-halfbraid}
c_{\hat A(M),W} ~=~
\sum_{i,j\in\Ic} \sum_\alpha
~~
\begin{tikzpicture}[baseline=0em,very thick,black]
\coordinate (t1) at (-1.4,1.3);
\coordinate (t2) at (0,1.3);
\coordinate (t3) at ([shift={(0.8,0)}]t2);
\coordinate (t4) at ([shift={(0.8,0)}]t3);
\coordinate (b2) at (0,-1.3);
\coordinate (b3) at ([shift={(0.8,0)}]b2);
\coordinate (b4) at ([shift={(0.8,0)}]b3);
\coordinate (b5) at ([shift={(1.4,0)}]b4);
\coordinate (ab) at (-0.8,0);
\coordinate (a) at ([shift={(0,1)}]b4);

\draw (b3)--(t3);
\draw (b4)--(t4);
\draw (0,0)--(t2);

\begin{scope}[decoration={markings,mark=at position 0.5 with {\arrow{<}}}]
\draw [postaction={decorate}](ab)  arc  (-180:0:0.4);
\end{scope}

\begin{scope}[decoration={markings,mark=at position 0.5 with {\arrow{>}}}]
\draw [postaction={decorate}](ab)  arc  (30:180:0.4);
\end{scope}

\draw (ab) .. controls ++(0.5,0.5) and ++(0,-0.5) .. (t1);
\draw ([shift={(-0.747,-0.19)}]ab) .. controls ++(0,-0.5) and ++(0,0.5) .. (b2);
\draw (a) .. controls ++(0.5,-0.5) and ++(0,0.5) .. (b5);

\draw[fill=white] ([shift={(-0.07,-0.07)}]a) rectangle ++(0.14,0.14);
\draw[fill=white] ([shift={(-0.07,-0.07)}]ab) rectangle ++(0.14,0.14);

\node[above] at (t1)  {\small $W$};
\node[below] at (b5)  {\small $W$};
\node[above] at (t2)  {\small $j^\vee$};
\node[below] at (b2)  {\small $i^\vee$};
\node[above] at (t3)  {\small $M$};
\node[below] at (b3)  {\small $M$};
\node[above] at (t4)  {\small $j$};
\node[below] at (b4)  {\small $i$};
\node at ([shift={(-0.25,0.2)}] a) {\small $\alpha$};
\node at ([shift={(-0.25,-0.2)}] ab) {\small $\bar\alpha$};
\end{tikzpicture}
\ee
Here $W \in \Cc$, $i,j \in \Ic$, the $\alpha$ are basis elements of $\Cc(i \otimes W,j)$, and the $\bar\alpha$ are the dual basis of $\Cc(j,i \otimes W)$ in the sense that $\alpha \circ \bar\beta = \delta_{\alpha,\beta}\,\id_j$. The following two expressions are morphisms $\hat A(Z) \to Z$ and $Z \to \hat A(Z)$ in $\Zc(\Cc)$, respectively.
\be\label{eq:half-loops-morph-in-Z}
\sum_{U \in \Ic} 
\begin{tikzpicture}[baseline=2em]
\coordinate (vv) at (-0.7,0);
\coordinate (y) at (0,0);
\coordinate (v) at (0.7,0);
\coordinate (yt) at ([shift={(0,2.0)}]y);
\draw[very thick,black] (y) -- (yt);
\begin{scope}[very thick,black,decoration={markings,mark=at position 0.6 with {\arrow{>}}}]
\draw [postaction={decorate}] (v) -- ([shift={(0,0.2)}]v) .. controls ++(0,0.5) and ++(0,0.6) .. ([shift={(0,1)}]vv) -- (vv);
\end{scope}
\draw[very thick,black] (0,1) circle (0.1);
\node[below] at (vv) {\small $U^\vee$};
\node[below] at (y) {\small $Z$};
\node[above] at (yt) {\small $Z$};
\node[below] at (v) {\small $U$};
\end{tikzpicture}
\qquad , \qquad
\sum_{U \in \Ic} \frac{\dim_r(U)}{\Dim(\Cc)} 
\begin{tikzpicture}[baseline=4em]
\coordinate (vv) at (-0.7,0);
\coordinate (y) at (0,0);
\coordinate (v) at (0.7,0);
\coordinate (vvt) at ([shift={(0,2.8)}]vv);
\coordinate (yt) at ([shift={(0,2.8)}]y);
\coordinate (vt) at ([shift={(0,2.8)}]v);
\draw[very thick,black] (0,0.8) -- (yt);
\begin{scope}[very thick,black,decoration={markings,mark=at position 0.6 with {\arrow{>}}}]
\draw [postaction={decorate}] (vvt) -- ([shift={(0,-0.2)}]vvt) .. controls ++(0,-0.5) and ++(0,-0.6) .. ([shift={(0,-1)}]vt) -- (vt);
\end{scope}
\draw[very thick,black] ([shift={(0,2.8)}] 0,-1) circle (0.1);
\node[above] at (vvt) {\small $U^\vee$};
\node[above] at (yt) {\small $Z$};
\node[above] at (vt) {\small $U$};
\node[below] at (0,0.8) {\small $Z$};
\end{tikzpicture}
\quad .
\ee
The dimension factor in the second expression is required to get compatibility with the half-braiding, for which one needs to use 
\be\label{eq:basis-sum-with-duals-1}
\sum_{U \in \Ic} \frac{\dim_r(U)}{\dim_r(V)} \, \sum_\alpha
~
\begin{tikzpicture}[baseline=1.5em,very thick,black]
\coordinate (a) at (0,1.5);
\coordinate (ab) at (0,0);
\coordinate (t1) at (0,2.3);
\coordinate (b1) at (0,-0.8);

\draw (a)--(t1);
\draw (ab)--(b1);
\draw ([shift={(1,0)}]a)--([shift={(1,0)}]t1);
\draw ([shift={(1,0)}]ab)--([shift={(1,0)}]b1);

\begin{scope}[decoration={markings,mark=at position 0.5 with {\arrow{>}}}]
\draw [postaction={decorate}](ab)  arc  (170:0:0.5);
\end{scope}

\begin{scope}[decoration={markings,mark=at position 0.5 with {\arrow{<}}}]
\draw [postaction={decorate}](a)  arc  (-170:0:0.5);
\end{scope}

\draw (ab) .. controls ++(-0.5,0.5) and ++(-0.5,-0.5) .. (a);

\draw[fill=white] ([shift={(-0.07,-0.07)}]a) rectangle ++(0.14,0.14);
\draw[fill=white] ([shift={(-0.07,-0.07)}]ab) rectangle ++(0.14,0.14);

\node[below] at (b1)  {\small $V$};
\node[above] at (t1)  {\small $V$};
\node[below] at ([shift={(1,0)}]b1)  {\small $W^\vee$};
\node[above] at ([shift={(1,0)}]t1)  {\small $W^\vee$};
\node at ([shift={(-0.25,0.2)}] a) {\small $\alpha$};
\node at ([shift={(-0.25,-0.2)}] ab) {\small $\bar\alpha$};
\node at (-0.1,0.7) {\small $U$};
\end{tikzpicture}
=
~
\id_{V \otimes W^\vee} \ ,
\ee
see \cite[Lem.\,4.9]{Turaev-Virelizier-book} for a proof.

\subsubsection*{The $h_Z$ are linearly independent:}

Pick a simple object $V \in \Ic$ of $\Cc$ and split the identity on $\hat A(V)$ as
\be\label{eq:id-AV-decomp}
\id_{\hat A(V)} = \sum_{Z \in \Jc} \sum_{\gamma_Z}
\big[ \hat A(V) \xrightarrow{\bar\gamma_Z} Z \xrightarrow{\gamma_Z} \hat A(V) \big] \ ,
\ee
where $\gamma_Z$ and $\bar\gamma_Z$ denote a dual basis pair in $\Zc(\Cc)(Z,\hat A(V))$ and $\Zc(\Cc)(\hat A(V),Z)$, respectively. Write $\pi_V$ for the projection
\be
	H = \bigoplus_{U \in \Ic} U^\vee \! \otimes A(U) \xrightarrow{~ \pi_V ~} V^\vee \otimes A(V) \ .
\ee
For $Y$ a simple object in $\Zc(\Cc)$ we have 
\vspace{-1em}
\be
	(\id \otimes \bar\gamma_Z) \circ \pi_V \circ h_Y
	~=~
	\sum_{U \in \Ic} \frac{\dim_r(U)}{\Dim(\Cc)} \, \sum_\rho	
~
\begin{tikzpicture}[very thick,baseline=5em]
\coordinate (yv) at (-1.4,0);
\coordinate (vv) at (-0.7,0);
\coordinate (y) at (0,0);
\coordinate (v) at (0.7,0);
\coordinate (vvt) at ([shift={(0,3)}]vv);
\coordinate (yt) at ([shift={(0,3)}]y);
\coordinate (vt) at ([shift={(0,3)}]v);
\coordinate (yvt) at ([shift={(0,3)}]yv);
\coordinate (a) at  ([shift={(0,-0.8)}]yt);
\coordinate (ab) at  ([shift={(0,-0.8)}]yvt);

\draw[black] (y) -- (yt);
\draw[black] (yv) -- ([shift={(0,1.4)}]yvt);
\draw[black] ([shift={(0,0.6)}]yt) -- ++(0,0.8);
\begin{scope}[black,decoration={markings,mark=at position 0.5 with {\arrow{<}}}]
\draw [postaction={decorate}](y)  arc  (0:-180:0.7);
\end{scope}
\begin{scope}[black,decoration={markings,mark=at position 0.54 with {\arrow{>}}}]
\draw [postaction={decorate}] (vvt) -- ([shift={(0,-1)}]vvt) .. controls ++(0,-0.5) and ++(0,-0.6) .. ([shift={(0,-1.8)}]vt) -- (vt);
\end{scope}
\draw[black] (0,1.2) circle (0.1);
\draw[black] ([shift={(-0.3,0)}]vvt) rectangle ++(2.0,0.6) node[pos=.5] {\small $\bar\gamma_Z$};

\draw[dotted] (-1.15,0.5) rectangle ++(2.3,3.5);

\draw[fill=white] ([shift={(-0.07,-0.07)}]a) rectangle ++(0.14,0.14);
\draw[fill=white] ([shift={(-0.07,-0.07)}]ab) rectangle ++(0.14,0.14);
\node at ([shift={(-0.25,-0.1)}] a) {\small $\rho$};
\node at ([shift={(-0.35,-0.1)}] ab) {\small $\bar\rho^\vee$};

\node at ([shift={(0.25,-0.4)}]yt) {\small $V$};
\node at ([shift={(0.25,-1.4)}]yt) {\small $Y$};
\node at ([shift={(0.25,-3.0)}]yt) {\small $Y$};
\node at ([shift={(0.2,-0.4)}]vt) {\small $U$};
\node at ([shift={(0,1.65)}]yt) {\small $Z$}; 
\node at ([shift={(0.1,1.65)}]yvt) {\small $V^\vee$};
\end{tikzpicture}
\qquad ,
\ee
where $\{\rho\}$ is a basis of $\Cc(\forget(Y),V)$ and $\{ \bar\rho \}$ the dual basis of $\Cc(V,\forget(Y))$.
Using \eqref{eq:half-loops-morph-in-Z} we see that the morphism in the dotted box is a morphism in $\Zc(\Cc)$, and since $Y,Z$ are simple, we find
\be\label{eq:gamma-proj-prop}
	(\id \otimes \bar\gamma_Z) \circ \pi_V \circ h_Y = 0
	\quad \text{if} \quad Y \not\cong Z \ .
\ee
Let now $L := \sum_{Z \in \Jc} c_Z \,\gamma_Z$, $c_Z \in \Bbbk$, be some linear combination such that $L=0$. Then for all $V \in \Cc$,
\be
	0 = (\id \otimes \bar\gamma_Z) \circ \pi_V \circ L 
	\overset{\eqref{eq:gamma-proj-prop}}= c_Z \cdot 
	(\id \otimes \bar\gamma_Z) \circ \pi_V \circ h_Z
	\ .
\ee
On the other hand, using \eqref{eq:id-AV-decomp} and \eqref{eq:gamma-proj-prop} we get 
\be
0
~=~
c_Z \sum_{\gamma_Z} (\id \otimes \gamma_Z) \circ (\id \otimes \bar\gamma_Z) \circ \pi_V \circ h_Z
~=~
c_Z\,\pi_V \circ h_Z  \ .
\ee 
Thus in order to check that $c_Z=0$ it remains to verify that $\pi_V \circ h_Z \neq 0$ for some $V$.

To this end we choose $V = \one$ and use the explicit form \eqref{eq:H-via-simples} of $\jmath$ in the definition $h_Z$, as well as the fact that the half-braiding with $\one$ is trivial (we write $Z$ instead of $\forget(Z)$ in the following equalities),
\be
\pi_\one \circ h_Z 
=
\sum_{V,\rho}
\big[ \one \xrightarrow{\coevR_{\!\!Z}} Z^\vee Z \xrightarrow{\bar\rho^\vee\otimes \rho} V^\vee V = \one^\vee V^\vee \one V \big] 
=
\sum_{V \in \Ic}
\dim\!\big( \Cc(V,Z) \big)\, \coevR_V \ .
\ee
This is non-zero as $\Cc(V,Z) \neq 0$ holds for at least one $V$.

\subsubsection*{The $h_Z$ span the image of $\tilde B_p$:}

From \eqref{eq:SN(T-p)-via-CUAU} we see that $SN(T-\{p\})$ is spanned by elements $\bigoplus_U f_U \in \bigoplus_U \Cc(U , A(U))$ via
\vspace{-.5em}
\be
	{\textstyle \bigoplus_U f_U}
~	
	\longmapsto
~	
\sum_{U,V \in\Ic}
\begin{tikzpicture}[very thick,baseline=0em]
\tikzmath{
	\rx=3.6;\ry=3.9;
	\xid = 0.3; \yid = 0.2;
	\hcross = 0;
}
\begin{scope}[orange!80!black]
\draw (-\rx/2,-\ry/2) rectangle ++(\rx,\ry);
\draw (\xid * \rx -1/2*\rx,-\ry/2-0.2) -- ++(0.2,0.4);
\draw (\xid * \rx -1/2*\rx+0.1,-\ry/2-0.2) -- ++(0.2,0.4);
\draw (\xid * \rx -1/2*\rx,\ry/2-0.2) -- ++(0.2,0.4);
\draw (\xid * \rx -1/2*\rx+0.1,\ry/2-0.2) -- ++(0.2,0.4);
\draw (\rx/2-0.2,\yid * \ry - \ry/2) -- ++(0.4,0.2);
\draw (-\rx/2-0.2,\yid * \ry - \ry/2) -- ++(0.4,0.2);
\end{scope}

\begin{scope}[blue!80!black,decoration={markings,mark=at position 0.54 with {\arrow{>}}}]
\draw [postaction={decorate}] (0,\hcross) -- (0,\ry/2);
\draw [postaction={decorate}] (0,-\ry/2) -- (0,\hcross-0.6);
\draw [postaction={decorate}] (-\rx/2,\hcross+0.4) -- 
(-\rx/2+0.8,\hcross+0.4) .. controls ++(0.3,0) and ++(0,0.3) .. 
(-0.3,\hcross);
\draw [postaction={decorate}] 
(0.3,\hcross) 
.. controls ++(0,0.3) and ++(-0.3,0) .. 
(\rx/2-0.8,\hcross+0.4) 
--
(\rx/2,\hcross+0.4);
\end{scope}

\draw[blue!80!black,fill=white] (-0.5,\hcross-0.6) rectangle ++(1.0,0.6) node[pos=.5,black] {\small $f_U$};

\draw[very thick,black,fill=black] (-\rx/2,-\ry/2) circle (0.08);
\draw[very thick,black,fill=black] (\rx/2,-\ry/2) circle (0.08);
\draw[very thick,black,fill=black] (-\rx/2,\ry/2) circle (0.08);
\draw[very thick,black,fill=black] (\rx/2,\ry/2) circle (0.08);

\node[left] at (-\rx/2,-\ry/2) {\small $p$};
\node[left] at (-\rx/2,\ry/2) {\small $p$};
\node[right] at (\rx/2,-\ry/2) {\small $p$};
\node[right] at (\rx/2,\ry/2) {\small $p$};

\node at (0.25,-\ry/2+0.4) {\small $U$};
\node at (0.25,\ry/2-0.4) {\small $U$};
\node at (-\rx/2+0.3,\hcross+0.1) {\small $V$};
\node at (\rx/2-0.3,\hcross+0.1) {\small $V$};
\end{tikzpicture}
\qquad .
\ee
Next use the isomorphism $\Cc(U , A(U)) \cong \Zc(\Cc)(\hat A(U) , \hat A(U))$ and the decomposition of elements in that space as in \eqref{eq:ZC-Hom-decomp} to see that $SN(T-\{p\})$ can also be spanned by the images 
\vspace{-1em}
\be
\alpha_Z \otimes \beta_Z 
~\longmapsto~
\sum_{V \in\Ic}
\begin{tikzpicture}[very thick,baseline=0em]
\tikzmath{
	\rx=3.6;\ry=4.2;
	\xid = 0.3; \yid = 0.3;
	\hcross = 1;
	\hmor = -0.5;
}
\begin{scope}[orange!80!black]
\draw (-\rx/2,-\ry/2) rectangle ++(\rx,\ry);
\draw (\xid * \rx -1/2*\rx,-\ry/2-0.2) -- ++(0.2,0.4);
\draw (\xid * \rx -1/2*\rx+0.1,-\ry/2-0.2) -- ++(0.2,0.4);
\draw (\xid * \rx -1/2*\rx,\ry/2-0.2) -- ++(0.2,0.4);
\draw (\xid * \rx -1/2*\rx+0.1,\ry/2-0.2) -- ++(0.2,0.4);
\draw (\rx/2-0.2,\yid * \ry - \ry/2) -- ++(0.4,0.2);
\draw (-\rx/2-0.2,\yid * \ry - \ry/2) -- ++(0.4,0.2);
\end{scope}

\begin{scope}[blue!80!black,decoration={markings,mark=at position 0.54 with {\arrow{>}}}]
\draw [postaction={decorate}] (0,-\ry/2) -- (0,\hmor-0.6);
\draw [postaction={decorate}] (0,\hmor) -- (0,\hcross-0.6);
\draw [postaction={decorate}] (0,\hcross) -- (0,\ry/2);

\draw [postaction={decorate}] (-\rx/2,\hcross+0.4) -- 
(-\rx/2+0.8,\hcross+0.4) .. controls ++(0.3,0) and ++(0,0.3) .. 
(-0.3,\hcross);
\draw [postaction={decorate}] 
(0.3,\hcross) 
.. controls ++(0,0.3) and ++(-0.3,0) .. 
(\rx/2-0.8,\hcross+0.4) 
--
(\rx/2,\hcross+0.4);
\end{scope}

\draw[blue!80!black,fill=white] (-0.5,\hcross-0.6) rectangle ++(1.0,0.6) node[pos=.5,black] {\small $\beta_Z$};

\draw[blue!80!black,fill=white] (-0.5,\hmor-0.6) rectangle ++(1.0,0.6) node[pos=.5,black] {\small $\alpha_Z'$};

\draw[very thick,black,fill=black] (-\rx/2,-\ry/2) circle (0.08);
\draw[very thick,black,fill=black] (\rx/2,-\ry/2) circle (0.08);
\draw[very thick,black,fill=black] (-\rx/2,\ry/2) circle (0.08);
\draw[very thick,black,fill=black] (\rx/2,\ry/2) circle (0.08);

\node at (0.25,-\ry/2+0.3) {\small $U$};
\node at (0.35,{(\hmor+\hcross-0.6)/2}) {\small $Z$};
\node at (0.25,\ry/2-0.4) {\small $U$};
\node at (-\rx/2+0.3,\hcross+0.1) {\small $V$};
\node at (\rx/2-0.3,\hcross+0.1) {\small $V$};
\end{tikzpicture}
~=~
\sum_{V \in\Ic}
\begin{tikzpicture}[very thick,baseline=0em]
\tikzmath{
	\rx=3.6;\ry=4.2;
	\xid = 0.3; \yid = 0.3;
	\hcross = 1;
	\hmor = 0.7;
}
\begin{scope}[orange!80!black]
\draw (-\rx/2,-\ry/2) rectangle ++(\rx,\ry);
\draw (\xid * \rx -1/2*\rx,-\ry/2-0.2) -- ++(0.2,0.4);
\draw (\xid * \rx -1/2*\rx+0.1,-\ry/2-0.2) -- ++(0.2,0.4);
\draw (\xid * \rx -1/2*\rx,\ry/2-0.2) -- ++(0.2,0.4);
\draw (\xid * \rx -1/2*\rx+0.1,\ry/2-0.2) -- ++(0.2,0.4);
\draw (\rx/2-0.2,\yid * \ry - \ry/2) -- ++(0.4,0.2);
\draw (-\rx/2-0.2,\yid * \ry - \ry/2) -- ++(0.4,0.2);
\end{scope}

\begin{scope}[blue!80!black,decoration={markings,mark=at position 0.54 with {\arrow{>}}}]
\draw [postaction={decorate}] (0,-\ry/2) -- (0,\hcross-0.6-1.8);
\draw [postaction={decorate}] (0,\hcross-1.8) -- (0,\hmor-0.6);
\draw [postaction={decorate}] (0,\hmor) -- (0,\ry/2);
\end{scope}

\begin{scope}[blue!80!black,decoration={markings,mark=at position 0.2 with {\arrow{>}}}]
\draw [postaction={decorate}] (-\rx/2,\hcross+0.4) -- 
(-\rx/2+0.8,\hcross+0.4) .. controls ++(0.8,0) and ++(-0.8,1.8) .. 
(-0.35,-1.8+\hcross);
\end{scope}

\begin{scope}[blue!80!black,decoration={markings,mark=at position 0.8 with {\arrow{>}}}]
\draw [postaction={decorate}] 
(0.35,-1.8+\hcross) 
.. controls ++(0.8,1.8) and ++(-0.8,0) .. 
(\rx/2-0.8,\hcross+0.4) 
--
(\rx/2,\hcross+0.4);
\end{scope}

\draw[blue!80!black,fill=white] (-0.5,-1.8+\hcross-0.6) rectangle ++(1.0,0.6) node[pos=.5,black] {\small $\beta_Z$};

\draw[blue!80!black,fill=white] (-0.5,\hmor-0.6) rectangle ++(1.0,0.6) node[pos=.5,black] {\small $\alpha_Z'$};

\draw[very thick,black,fill=black] (-\rx/2,-\ry/2) circle (0.08);
\draw[very thick,black,fill=black] (\rx/2,-\ry/2) circle (0.08);
\draw[very thick,black,fill=black] (-\rx/2,\ry/2) circle (0.08);
\draw[very thick,black,fill=black] (\rx/2,\ry/2) circle (0.08);

\draw[dotted] (-1,-1.65) rectangle ++(2,2.6);

\node at (0.25,-\ry/2+0.2) {\small $Z$};
\node at (0.25,{(\hmor+\hcross-0.6-1.8)/2+0.15}) {\small $U$};
\node at (0.25,\ry/2-0.4) {\small $Z$};
\node at (-\rx/2+0.3,\hcross+0.1) {\small $V$};
\node at (\rx/2-0.3,\hcross+0.1) {\small $V$};
\end{tikzpicture}
\ .
\ee
Here, $\alpha'_Z = \alpha_Z \circ \iota_1(U) : U \to Z$, cf.~\eqref{eq:Hom-ZC-SN-iso}, and
the morphism in the dotted box is given by $\hat A(\alpha_Z') \circ \beta_Z$ and is a morphism in $\Zc(\Cc)$. We can replace it by an arbitrary morphism $F_Z \in \Zc(\Cc)(Z,\hat A(Z))$ and altogether obtain a surjection $\textstyle{\bigoplus_Z F_Z} \mapsto \SN(T-\{p\})$.

We now compute the action of the idempotent $B_p$ on the image of the summand $F_Z$. We decompose $F_Z = \bigoplus_j f_j : Z \to A(Z)$ with $f_j : Z \to j^\vee Z j$. The $i$-loop from $B_p$ is oriented clockwise in the computation below, and consequently $\dim_l(i)$ is used in the prefactor instead of $\dim_r(U)$. This amounts to summing over $i = U^\vee$ in \eqref{eq:idem-Bp}. We show the $i$-loop as a green dashed line to make the deformation of edges easier to follow.  
\allowdisplaybreaks
\begin{align}
&\sum_{i,j} \frac{\dim_l(i)}{\Dim(\Cc)}	
~
\begin{tikzpicture}[very thick,baseline=0em]
\tikzmath{
	\rx=4;\ry=4.5;
	\xid = 0.3; \yid = 0.3;
	\hcross = 0;
}
\begin{scope}[orange!80!black]
\draw (-\rx/2,-\ry/2) rectangle ++(\rx,\ry);
\draw (\xid * \rx -1/2*\rx,-\ry/2-0.2) -- ++(0.2,0.4);
\draw (\xid * \rx -1/2*\rx+0.1,-\ry/2-0.2) -- ++(0.2,0.4);
\draw (\xid * \rx -1/2*\rx,\ry/2-0.2) -- ++(0.2,0.4);
\draw (\xid * \rx -1/2*\rx+0.1,\ry/2-0.2) -- ++(0.2,0.4);
\draw (\rx/2-0.2,\yid * \ry - \ry/2) -- ++(0.4,0.2);
\draw (-\rx/2-0.2,\yid * \ry - \ry/2) -- ++(0.4,0.2);
\end{scope}
\begin{scope}[green!60!black,densely dashed,decoration={markings,mark=at position 0.54 with {\arrow{>}}}]
\draw [postaction={decorate}] (-\rx/2,-\ry/2+0.7) arc (90:0:0.7) node[pos=.6,black,right] {\small $i$};
\draw [postaction={decorate}] (-\rx/2+0.7,\ry/2) arc (0:-90:0.7) node[pos=.4,black,right] {\small $i$};
\draw [postaction={decorate}] (\rx/2-0.7,-\ry/2) arc (180:90:0.7) node[pos=.4,black,left] {\small $i$};
\draw [postaction={decorate}] (\rx/2,\ry/2-0.7) arc (270:180:0.7) node[pos=.6,black,left] {\small $i$};
\end{scope}
\begin{scope}[blue!80!black,decoration={markings,mark=at position 0.54 with {\arrow{>}}}]
\draw [postaction={decorate}] (0,\hcross) -- (0,\ry/2);
\draw [postaction={decorate}] (0,-\ry/2) -- (0,\hcross-0.6);
\draw [postaction={decorate}] (-\rx/2,\hcross+0.4) -- 
(-\rx/2+0.8,\hcross+0.4) .. controls ++(0.3,0) and ++(0,0.3) .. 
(-0.3,\hcross);
\draw [postaction={decorate}] 
(0.3,\hcross) 
.. controls ++(0,0.3) and ++(-0.3,0) .. 
(\rx/2-0.8,\hcross+0.4) 
--
(\rx/2,\hcross+0.4);
\end{scope}
\draw[blue!80!black,fill=white] (-0.5,\hcross-0.6) rectangle ++(1.0,0.6) node[pos=.5,black] {\small $f_j$};
\draw[very thick,black,fill=black] (-\rx/2,-\ry/2) circle (0.08);
\draw[very thick,black,fill=black] (\rx/2,-\ry/2) circle (0.08);
\draw[very thick,black,fill=black] (-\rx/2,\ry/2) circle (0.08);
\draw[very thick,black,fill=black] (\rx/2,\ry/2) circle (0.08);
\node at (0.25,-\ry/2+0.4) {\small $Z$};
\node at (0.25,\ry/2-0.4) {\small $Z$};
\node at (-\rx/2+0.3,\hcross+0.1) {\small $j$};
\node at (\rx/2-0.3,\hcross+0.1) {\small $j$};
\end{tikzpicture}
~\overset{(1)}=~
\sum_{i,j,k,\alpha} \frac{\dim_l(i)}{\Dim(\Cc)}	
~
\begin{tikzpicture}[very thick,baseline=0em]
\tikzmath{
	\rx=4;\ry=4.5;
	\xid = 0.2; \yid = 0.3;
	\hcross = 0;
}
\begin{scope}[orange!80!black]
\draw (-\rx/2,-\ry/2) rectangle ++(\rx,\ry);
\draw (\xid * \rx -1/2*\rx,-\ry/2-0.2) -- ++(0.2,0.4);
\draw (\xid * \rx -1/2*\rx+0.1,-\ry/2-0.2) -- ++(0.2,0.4);
\draw (\xid * \rx -1/2*\rx,\ry/2-0.2) -- ++(0.2,0.4);
\draw (\xid * \rx -1/2*\rx+0.1,\ry/2-0.2) -- ++(0.2,0.4);
\draw (\rx/2-0.2,\yid * \ry - \ry/2) -- ++(0.4,0.2);
\draw (-\rx/2-0.2,\yid * \ry - \ry/2) -- ++(0.4,0.2);
\end{scope}
\begin{scope}[green!60!black,densely dashed,decoration={markings,mark=at position 0.6 with {\arrow{>}}}]
\draw [postaction={decorate}] (-0.4,\ry/2) .. controls ++(0,-1.7) and ++(1.7,0) .. (-\rx/2,\hcross+0.4+0.3) node[pos=.6,black,above] {\small $i$};
\draw [postaction={decorate}] (\rx/2,\hcross+0.4+0.3) .. controls ++(-1.7,0) and ++(0,-1.7) .. (0.4,\ry/2) node[pos=.4,black,above] {\small $i$};
\draw [postaction={decorate}] (-0.42,\hcross) .. controls ++(-0.9,-1.5) and ++(0,1.0) .. (-0.4,-\ry/2) node[pos=.4,black,left] {\small $i$};
\draw [postaction={decorate}] (0.4,-\ry/2) .. controls ++(0,1.0) and ++(0.9,-1.5) .. (0.42,\hcross) node[pos=.4,black,right] {\small $i$};
\end{scope}
\begin{scope}[blue!80!black,decoration={markings,mark=at position 0.54 with {\arrow{>}}}]
\draw [postaction={decorate}] (0,\hcross-0.6) -- (0,\ry/2);
\draw [postaction={decorate}] (0,-\ry/2) -- (0,\hcross-0.6-0.6);
(\rx/2-0.8,\hcross+0.4) 
--
(\rx/2,\hcross+0.4);
\end{scope}
\begin{scope}[blue!80!black,decoration={markings,mark=at position 0.2 with {\arrow{>}}}]
\draw [postaction={decorate}] (-\rx/2,\hcross+0.4) -- 
(-\rx/2+0.8,\hcross+0.4) .. controls ++(0.8,0) and ++(0,1) .. 
(-0.35,-0.8+\hcross);
\end{scope}
\begin{scope}[blue!80!black,decoration={markings,mark=at position 0.8 with {\arrow{>}}}]
\draw [postaction={decorate}] 
(0.35,-0.6+\hcross) 
.. controls ++(0,1) and ++(-0.8,0) .. 
(\rx/2-0.6,\hcross+0.4) 
--
(\rx/2,\hcross+0.4);
\end{scope}
\draw[blue!80!black,fill=white] (-0.5,-0.6+\hcross-0.6) rectangle ++(1.0,0.6) node[pos=.5,black] {\small $f_j$};
\draw[very thick,black,fill=black] (-\rx/2,-\ry/2) circle (0.08);
\draw[very thick,black,fill=black] (\rx/2,-\ry/2) circle (0.08);
\draw[very thick,black,fill=black] (-\rx/2,\ry/2) circle (0.08);
\draw[very thick,black,fill=black] (\rx/2,\ry/2) circle (0.08);
\draw[blue!80!black,fill=white] (0.42-0.07,\hcross-0.07) rectangle ++(0.14,0.14) node[pos=0.5,black,right]{\small $\alpha$};
\draw[blue!80!black,fill=white] (-0.42-0.07,\hcross-0.07) rectangle ++(0.14,0.14) node[pos=0.5,black,left]{\small $\bar\alpha$};
\node at (0.2,\ry/2-0.4) {\small $Z$};
\node at (-\rx/2+0.3,\hcross+0.1) {\small $k$};
\node at (\rx/2-0.3,\hcross+0.1) {\small $k$};
\node at (-0.22,\hcross-0.35) {\small $j$};
\node at (0.22,\hcross-0.35) {\small $j$};
\end{tikzpicture}
\nonumber\\[1em]
&\overset{(2)}=
~
\sum_{i,k,l,\beta} \frac{\dim_l(l)}{\Dim(\Cc)}	
~
\begin{tikzpicture}[very thick,baseline=0em]
\tikzmath{
	\rx=4;\ry=4.5;
	\xid = 0.2; \yid = 0.3;
	\hcross = -0.5;
}
\begin{scope}[orange!80!black]
\draw (-\rx/2,-\ry/2) rectangle ++(\rx,\ry);
\draw (\xid * \rx -1/2*\rx,-\ry/2-0.2) -- ++(0.2,0.4);
\draw (\xid * \rx -1/2*\rx+0.1,-\ry/2-0.2) -- ++(0.2,0.4);
\draw (\xid * \rx -1/2*\rx,\ry/2-0.2) -- ++(0.2,0.4);
\draw (\xid * \rx -1/2*\rx+0.1,\ry/2-0.2) -- ++(0.2,0.4);
\draw (\rx/2-0.2,\yid * \ry - \ry/2) -- ++(0.4,0.2);
\draw (-\rx/2-0.2,\yid * \ry - \ry/2) -- ++(0.4,0.2);
\end{scope}
\begin{scope}[green!60!black,densely dashed,decoration={markings,mark=at position 0.6 with {\arrow{>}}}]
\draw [postaction={decorate}] (0,\ry/2-0.5) .. controls ++(-1,0.6) and ++(1.0,0) .. (-\rx/2,\hcross+0.9+0.3) node[pos=.7,black,above] {\small $i$};
\draw [postaction={decorate}] (\rx/2-1.3,\hcross+0.8) .. controls ++(0,1) and ++(0.5,-0.3) .. (0,\ry/2-0.5) node[pos=.3,black,right] {\small $i$};
\draw (\rx/2,\hcross+0.9+0.3) .. controls ++(-0.5,0) and ++(0,0.5) .. (\rx/2-0.6,\hcross+0.8) node[pos=.3,black,above] {\small $i$};
\end{scope}
\begin{scope}[blue!80!black,decoration={markings,mark=at position 0.54 with {\arrow{>}}}]
\draw [postaction={decorate}] (0,\hcross-0.6) -- (0,\ry/2);
\draw [postaction={decorate}] (0,-\ry/2) -- (0,\hcross-0.6-0.6);
(\rx/2-0.8,\hcross+0.4) 
--
(\rx/2,\hcross+0.4);
\end{scope}
\begin{scope}[blue!80!black,decoration={markings,mark=at position 0.2 with {\arrow{>}}}]
\draw [postaction={decorate}] (-\rx/2,\hcross+0.4) -- 
(-\rx/2+0.8,\hcross+0.4) .. controls ++(0.8,0) and ++(0,1) .. 
(-0.35,-0.8+\hcross);
\end{scope}
\begin{scope}[blue!80!black,decoration={markings,mark=at position 0.5 with {\arrow{>}}}]
\draw [postaction={decorate}] 
(0.35,-0.6+\hcross) .. controls ++(0,1) and ++(-0.3,-0.4) .. (\rx/2-1.3,\hcross+0.8);
\draw [postaction={decorate}] 
(\rx/2-0.6,\hcross+0.8) .. controls ++(-0.3,-0.4) and ++(0.3,-0.4) .. (\rx/2-1.3,\hcross+0.8) node[pos=0.5,black,below]{\small $l$};
\draw [postaction={decorate}] 
(\rx/2-0.6,\hcross+0.8) .. controls ++(0.3,-0.6) and ++(-0.3,0) .. (\rx/2,\hcross+0.4);
\end{scope}
\draw[blue!80!black,fill=white] (-0.5,-0.6+\hcross-0.6) rectangle ++(1.0,0.6) node[pos=.5,black] {\small $f_k$};
\draw[black,fill=black] (-\rx/2,-\ry/2) circle (0.08);
\draw[black,fill=black] (\rx/2,-\ry/2) circle (0.08);
\draw[black,fill=black] (-\rx/2,\ry/2) circle (0.08);
\draw[black,fill=black] (\rx/2,\ry/2) circle (0.08);
\draw[blue!80!black] (0,\ry/2-0.5) circle (0.15);
\draw[blue!80!black,fill=white] (\rx/2-0.6-0.07,\hcross+0.8-0.07) rectangle ++(0.14,0.14) node[pos=0.5,black,right]{\small $\bar\beta$};
\draw[blue!80!black,fill=white] (\rx/2-1.3-0.07,\hcross+0.8-0.07) rectangle ++(0.14,0.14) node[pos=0.5,black,left]{\small $\beta$};
\node at (-0.2,\ry/2-2) {\small $Z$};
\node at (-\rx/2+0.3,\hcross+0.1) {\small $k$};
\node at (\rx/2-0.3,\hcross+0.1) {\small $k$};
\node at (0.5,\hcross-0.35) {\small $k$};
\end{tikzpicture}
~
\overset{(3)}=
~
\sum_{i,l} \frac{\dim_l(l)}{\Dim(\Cc)}	
~
\begin{tikzpicture}[very thick,baseline=0em]
\tikzmath{
	\rx=4;\ry=4.5;
	\xid = 0.2; \yid = 0.3;
	\hcross = 0.3;
}
\begin{scope}[orange!80!black]
\draw (-\rx/2,-\ry/2) rectangle ++(\rx,\ry);
\draw (\xid * \rx -1/2*\rx,-\ry/2-0.2) -- ++(0.2,0.4);
\draw (\xid * \rx -1/2*\rx+0.1,-\ry/2-0.2) -- ++(0.2,0.4);
\draw (\xid * \rx -1/2*\rx,\ry/2-0.2) -- ++(0.2,0.4);
\draw (\xid * \rx -1/2*\rx+0.1,\ry/2-0.2) -- ++(0.2,0.4);
\draw (\rx/2-0.2,\yid * \ry - \ry/2) -- ++(0.4,0.2);
\draw (-\rx/2-0.2,\yid * \ry - \ry/2) -- ++(0.4,0.2);
\end{scope}
\begin{scope}[green!60!black,densely dashed,decoration={markings,mark=at position 0.6 with {\arrow{>}}}]
\draw [postaction={decorate}] (0,\ry/2-0.5) .. controls ++(-1,0.6) and ++(0,1.0) .. (-0.35,\hcross) node[pos=.8,black,left] {\small $i$};
\draw [postaction={decorate}] (0.35,\hcross) .. controls ++(0,1) and ++(0.6,-0.4) .. (0,\ry/2-0.5) node[pos=.3,black,right] {\small $i$};
\end{scope}
\begin{scope}[blue!80!black,decoration={markings,mark=at position 0.7 with {\arrow{>}}}]
\draw [postaction={decorate}] (0,-\ry/2) -- (0,\hcross-0.6);
\end{scope}
\begin{scope}[blue!80!black,decoration={markings,mark=at position 0.4 with {\arrow{>}}}]
\draw [postaction={decorate}] (0,\hcross) -- (0,\ry/2);
\end{scope}
\begin{scope}[blue!80!black,decoration={markings,mark=at position 0.2 with {\arrow{<}}}]
\draw [postaction={decorate}] (-\rx/2,-\ry/2+0.7) -- 
(-\rx/2+0.8,-\ry/2+0.7) node[pos=.9,black,above] {\small $l$} .. controls ++(0.8,0) and ++(-0.8,0.4) .. (0,-\ry/2+0.7);
\end{scope}
\begin{scope}[blue!80!black,decoration={markings,mark=at position 0.8 with {\arrow{<}}}]
\draw [postaction={decorate}] (0,-\ry/2+0.7)  .. controls ++(0.8,-0.4) and ++(-0.8,0) .. (\rx/2-0.8,-\ry/2+0.7) -- (\rx/2,-\ry/2+0.7) node[pos=.1,black,above] {\small $l$};
\end{scope}
\draw[blue!80!black,fill=white] (-0.5,\hcross-0.6) rectangle ++(1.0,0.6) node[pos=.5,black] {\small $f_i$};
\draw[dotted] (-0.8,-0.5) rectangle ++(1.6,2.6);
\draw[black,fill=black] (-\rx/2,-\ry/2) circle (0.08);
\draw[black,fill=black] (\rx/2,-\ry/2) circle (0.08);
\draw[black,fill=black] (-\rx/2,\ry/2) circle (0.08);
\draw[black,fill=black] (\rx/2,\ry/2) circle (0.08);
\draw[blue!80!black] (0,\ry/2-0.5) circle (0.15);
\draw[blue!80!black] (0,-\ry/2+0.7) circle (0.15);
\node at (-0.2,-\ry/2+1.4) {\small $Z$};
\end{tikzpicture}
\nonumber
\end{align}
Here in step 1 we deformed the $i$-loop to run near the $Z$ and $j$ lines. Then we expanded $j \otimes i$ into a sum over $k$ and a basis $\{\alpha\}$ of $\Cc(j \otimes i,k)$. 

In step 2 the explicit form of the half braiding on $\hat A(Z)$ is used, as well as the fact that $F_Z$ is a morphism in $\Zc(\Cc)$, resulting in $i$ crossing $Z$ via the half-braiding of $Z$. Furthermore, $i^\vee \otimes k$ was expanded in a sum over $l^\vee$ using the identity
\be
\sum_{l \in \Ic} \frac{\dim_l(l)}{\dim_l(i)} \, \sum_\beta
\begin{tikzpicture}[baseline=0em,very thick,black]
\coordinate (a) at (0,-0.9);
\coordinate (ab) at (0,0.9);
\coordinate (b2) at (-0.5,-1.8);
\coordinate (t2) at (-0.5,1.8);
\coordinate (b1) at (-1.3,-1.8);
\coordinate (t1) at (-1.3,1.8);

\draw ([shift={(1,-0.07)}] ab) -- ([shift={(1,0.07)}] a);
\draw (t1) -- ++(0,-0.9);
\draw (b1) -- ++(0,0.9);

\begin{scope}[decoration={markings,mark=at position 0.5 with {\arrow{>}}}]
\draw [postaction={decorate}](ab)  arc  (170:0:0.5);
\draw [postaction={decorate}](a)  arc  (0:180:0.65);
\end{scope}

\begin{scope}[decoration={markings,mark=at position 0.5 with {\arrow{<}}}]
\draw [postaction={decorate}](a)  arc  (-170:0:0.5);
\draw [postaction={decorate}](ab)  arc  (0:-180:0.65);
\end{scope}

\draw (b2) .. controls ++(0,0.5) and ++(-0.5,-0.5) .. (a);
\draw (t2) .. controls ++(0,-0.5) and ++(-0.5,0.5) .. (ab);

\draw[fill=white] ([shift={(-0.07,-0.07)}]a) rectangle ++(0.14,0.14);
\draw[fill=white] ([shift={(-0.07,-0.07)}]ab) rectangle ++(0.14,0.14);
\node at ([shift={(0.25,0.2)}] a) {\small $\beta$};
\node at ([shift={(0.25,-0.2)}] ab) {\small $\bar\beta$};

\node[below] at (b1)  {\small $i^\vee$};
\node[above] at (t1)  {\small $i^\vee$};
\node[below] at (b2)  {\small $k$};
\node[above] at (t2)  {\small $k$};
\node at ([shift={(0.25,0.6)}] ab) {\small $l$};
\node at ([shift={(0.25,-0.6)}] a) {\small $l$};
\end{tikzpicture}
~~=~
\id_{i^\vee \otimes k}
\ ,
\ee
where $\{ \beta \}$ is a basis of $\Cc(k \otimes l,i)$ and $\{ \bar\beta\}$ the corresponding dual basis of  $\Cc(i,k \otimes l)$. This can be shown similarly to \eqref{eq:basis-sum-with-duals-1}.

In step 3 the $\bar\beta$ vertex is dragged through the vertical edge, and one sees once more the explicit form of the half braiding for $\hat A(Z)$ appear. One can thus move the $l$-loop past $F_Z$, resulting in the crossing of $l$ and $Z$ via the half-braiding of $Z$.

The sum over $i$ of the morphism in the dotted box is $F_Z$ composed with the left hand side in \eqref{eq:half-loops-morph-in-Z} and so is a morphism $Z \to Z$ in $\Zc(\Cc)$.
Since $Z$ is simple, the final expression is proportional to $\psi^{-1}(h_Z)$, cf.\ \eqref{eq:psi-1(hX)}, and 
thus the $\psi^{-1}(h_Z)$ span the image of $B_p$ in $\SN(T-\{p\})$. This completes the proof of Lemma~\ref{lem:Bp-action-punct-torus}.

\newpage

\newcommand\arxiv[2]      {\href{http://arXiv.org/abs/#1}{#2}}
\newcommand\doi[2]        {\href{http://dx.doi.org/#1}{#2}}
\newcommand\httpurl[2]    {\href{http://#1}{#2}}

\end{document}